\documentclass[11 pt]{amsart}
\usepackage{amsmath,amsthm,amssymb,bbm,enumerate,mathrsfs}
\usepackage{a4wide}
\usepackage{tikz}
\usetikzlibrary{calc,shapes}
\usepackage{hyperref}
\usepackage[numbers,sort&compress]{natbib}
\usepackage[shortlabels]{enumitem}

\tikzset{
  bigblack/.style={circle, draw=black!100,fill=black!100,thick, inner sep=1.5pt, minimum size=6mm},
  medblack/.style={circle, draw=black!100,fill=black!100,thick, inner sep=1.5pt, minimum size=4mm},
  blackcirc/.style={circle, draw=black!100,thick, inner sep=1.5pt, minimum size=6mm},  
}

\makeatletter
\pgfdeclareshape{myNode}{
  \inheritsavedanchors[from=rectangle] 
  \inheritanchorborder[from=rectangle]
  \inheritanchor[from=rectangle]{center}
  \inheritanchor[from=rectangle]{north}
  \inheritanchor[from=rectangle]{south}
  \inheritanchor[from=rectangle]{west}
  \inheritanchor[from=rectangle]{east}
  \backgroundpath{
    \southwest \pgf@xa=\pgf@x \pgf@ya=\pgf@y
    \northeast \pgf@xb=\pgf@x \pgf@yb=\pgf@y
    \pgfsetcornersarced{\pgfpoint{5pt}{5pt}}
    \pgfpathmoveto{\pgfpoint{\pgf@xa}{\pgf@ya}}
    \pgfpathlineto{\pgfpoint{\pgf@xa}{\pgf@yb}}
    \pgfpathlineto{\pgfpoint{\pgf@xb}{\pgf@yb}}
    \pgfsetcornersarced{\pgfpoint{5pt}{5pt}}
    \pgfpathlineto{\pgfpoint{\pgf@xb}{\pgf@ya}}
    \pgfpathclose
 }
}
\makeatother

\title{A Sharp Threshold for Bootstrap Percolation in a Random Hypergraph}
\author{
Natasha Morrison
\and Jonathan A. Noel
}
\date{}

\address{Department of Mathematics and Statistics, University of Victoria, David Turpin Building, 3800 Finnerty Road, Victoria, B.C., Canada V8P 5C2.}

\email[Natasha Morrison]{nmorrison@uvic.ca
}
\email[Jonathan A. Noel]{noelj@uvic.ca
}

\newtheoremstyle{case}{}{}{\normalfont}{}{\itshape}{:}{ }{}

\newtheorem{thm}[equation]{Theorem}
\newtheorem{lem}[equation]{Lemma}
\newtheorem{prop}[equation]{Proposition}
\newtheorem{cor}[equation]{Corollary}
\newtheorem{claim}[equation]{Claim}
\newtheorem{subclaim}[equation]{Subclaim}
\newtheorem*{Iprop}{Proposition~\ref{Ibound}}
\newtheorem*{usefullem}{Lemma~\ref{Xcount}}

\theoremstyle{definition}
\newtheorem{defn}[equation]{Definition}
\newtheorem*{process}{The First Phase Process}
\newtheorem*{process2}{The Second Phase Process in the Subcritical Case}
\newtheorem*{process3}{The Second Phase Process in the Supercritical Case}
\newtheorem{obs}[equation]{Observation}
\newtheorem*{ack}{Acknowledgements}
\newtheorem{rem}[equation]{Remark}

\newtheoremstyle{case}{}{}{\normalfont}{}{\itshape}{\normalfont:}{ }{}

\theoremstyle{case}

\numberwithin{equation}{section}

\newcommand{\plog}{\log^{O(1)}}
\newcommand{\conf}{\Upsilon}
\newcommand{\degre}{\operatorname{deg}}
\newcommand{\Erbt}{\frac{(t+1)^{K/10}}{\log^{K/5}(d)}}
\newcommand\given[1][]{\:#1\vert\:}
\newcommand{\Var}{\operatorname{Var}}
\newcommand{\overbar}[1]{\mkern1.5mu\overline{\mkern-1.5mu#1\mkern-1.5mu}\mkern 1.5mu}

\begin{document}

\begin{abstract}
Given a hypergraph $\mathcal{H}$, the \emph{$\mathcal{H}$-bootstrap process} starts with an initial set of \emph{infected} vertices of $\mathcal{H}$ and, at each step, a \emph{healthy} vertex $v$ becomes infected if there exists a hyperedge of $\mathcal{H}$ in which $v$ is the unique healthy vertex. We say that the set of initially infected vertices \emph{percolates} if every vertex of $\mathcal{H}$ is eventually infected. We show that this process exhibits a sharp threshold when $\mathcal{H}$ is a hypergraph obtained by randomly sampling hyperedges from an approximately $d$-regular $r$-uniform hypergraph satisfying some mild degree and codegree conditions; this confirms a conjecture of Morris. As a corollary, we obtain a sharp threshold for a variant of the graph bootstrap process for strictly $2$-balanced graphs which generalises a result of Kor\'{a}ndi, Peled and Sudakov. Our approach involves an application of the differential equations method.
\end{abstract}

\maketitle

\section{Introduction}

Given a hypergraph $\mathcal{H}$, the \emph{$\mathcal{H}$-bootstrap process} begins with an initial set of \emph{infected} vertices of $\mathcal{H}$ (a vertex that is not infected is \emph{healthy}) and, in each step, a healthy vertex becomes infected if there exists a hyperedge of $\mathcal{H}$ in which it is the unique healthy vertex. The set of initially infected vertices is said to \emph{percolate} if every vertex of $\mathcal{H}$ is eventually infected. This process was first studied by Balogh, Bollob\'{a}s, Morris and Riordan~\cite{LinAlg} and is motivated by numerous connections to other variants of bootstrap percolation; see Subsection~\ref{subsec:connections}.

The focus of this paper is on estimating the \emph{critical probability} of the $\mathcal{H}$-bootstrap process, denoted $p_c(\mathcal{H})$, which is the infimal density at which a random subset of $V(\mathcal{H})$ is likely to percolate. More formally,
\[p_c(\mathcal{H}):=\inf\left\{p\in (0,1): \mathbb{P}\left(V(\mathcal{H})_p \text{ percolates}\right)\geq 1/2\right\}\]
where, for a finite set $X$ and $p\in [0,1]$, we let $X_p$ denote a random subset of $X$ obtained by including each element with probability $p$ independently of one another. Specifically, we are interested in estimating the quantity $p_c\left(\mathcal{H}_q\right)$, where $\mathcal{H}$ is a ``sufficiently well behaved'' hypergraph  and, for $q\in [0,1]$,  $\mathcal{H}_q$ denotes the hypergraph obtained from $\mathcal{H}$ by including each hyperedge of $\mathcal{H}$ independently with probability $q$. Our main result applies to all $r$-uniform hypergraphs satisfying some mild conditions. To precisely state these conditions we require a few standard definitions.

Given a hypergraph $\mathcal{H}$ and a set $S\subseteq V(\mathcal{H})$, the \emph{codegree} of $S$, denoted $\degre(S)$, is the number of hyperedges $e$ of $\mathcal{H}$ with $S\subseteq e$. For $v\in V(\mathcal{H})$, the \emph{degree} of $v$ is defined to be $\degre(\{v\})$ and is denoted by $\degre(v)$. For an $r$-uniform hypergraph $\mathcal{H}$ and $1\leq \ell\leq r$, let $\Delta_\ell(\mathcal{H}):=\max\{\degre(S): S\subseteq V(\mathcal{H}), |S|=\ell\}$ and $\delta_\ell(\mathcal{H}):=\min\{\degre(S): S\subseteq V(\mathcal{H}), |S|=\ell\}$. We often write $\Delta_1(\mathcal{H})$ as $\Delta(\mathcal{H})$ and $\delta_1(\mathcal{H})$ as $\delta(\mathcal{H})$. We say that $\mathcal{H}$ is \emph{$d$-regular} if $\delta(\mathcal{H})=\Delta(\mathcal{H})=d$. Given a vertex $v$ of an $r$-uniform hypergraph $\mathcal{H}$, define the \emph{neighbourhood} (or \emph{link}) of $v$ to be $N_{\mathcal{H}}(v):= \{e\setminus\{v\}: e\in E\left(\mathcal{H}\right) \text{ and } v\in e\}$.

The following theorem is a corollary of our main result (Theorem~\ref{hypmainThm} below), but captures the main essence of the paper. In particular, it confirms (in a strong form) a conjecture of Morris~\cite{MorrisPC}.

\begin{thm}
\label{mainThmWeak}
Let $r\geq3$ and $s, \alpha, \beta >0$ be fixed. Let $\mathcal{H}$ be a $d$-regular, $r$-uniform hypergraph on $N$ vertices and let $q:= \alpha \cdot d^{-1/r-1}$. Suppose: 
\begin{enumerate}[(a)]
\item $N \le d^{\beta}$,
\item $\left|N_{\mathcal{H}}(u)\cap N_{\mathcal{H}}(v)\right|\leq  d^{1 - s}$ for distinct $u,v\in V(\mathcal{H})$, and
\item $\Delta_\ell(\mathcal{H})\leq d^{1-\frac{\ell-1}{r-1} - s}$ for $2\leq \ell\leq r-1$.
\end{enumerate}
Then
\[p_c\left(\mathcal{H}_{q}\right) = \left(\frac{r-2}{\alpha^{1/(r-2)}(r-1)^{(r-1)/(r-2)}} + o(1)\right)\cdot d^{-1/(r-1)},\]
with high probability as $d \rightarrow \infty$.
\end{thm}

Perhaps the most interesting feature of this theorem is that, for a certain range of $q$ and a broad class of hypergraphs $\mathcal{H}$, the main term of the asymptotics of the critical probability depends only on the value of $r$ and not on the specific underlying structure of $\mathcal{H}$. The following definition is useful for stating the full version of our result.

\begin{defn}
\label{hypwellBDef}
Given an integer $r\geq2$ and real numbers $d,\nu>0$ and $\rho\in [0,1]$,  we say that an $r$-uniform hypergraph $\mathcal{H}$ is \emph{$\left(d,\rho,\nu\right)$-well behaved} if the following conditions hold:
\begin{enumerate}[(a)]
\item\label{maxDeg} $\Delta(\mathcal{H})\leq d$,
\item\label{almostRegular} $\delta(\mathcal{H})\geq d(1-\rho)$,
\item\label{codegBound} $\Delta_\ell(\mathcal{H})\leq \rho\cdot d^{1-\frac{\ell-1}{r-1}}$ for $2\leq \ell\leq r-1$, 
\item\label{neighSim} $\left|N_{\mathcal{H}}(u)\cap N_{\mathcal{H}}(v)\right|\leq \rho\cdot d$ for distinct $u,v\in V(\mathcal{H})$, and
\item \label{Nnottoobig}$|V\left(\mathcal{H}\right)|\leq \nu$. 
\end{enumerate}
\end{defn}

Observe that the conditions in Theorem~\ref{mainThmWeak}  simply amount to $\mathcal{H}$ being $r$-uniform, $d$-regular and $\left(d,d^{-s},d^{\beta}\right)$-well behaved. We are now ready to state the main result of the paper. Here, and throughout the paper, $\log$ denotes the natural (base $e$) logarithm. 

\begin{thm}
\label{hypmainThm}
For fixed $r\geq3$ and real numbers $c,\alpha,\beta,\varepsilon>0$ there exists a positive constant $K=K(r,c,\alpha)$ such that, for $d$ sufficiently large, if $\mathcal{H}$ is an $r$-uniform $\left(d,\log^{-K}(d),d^\beta\right)$-well behaved hypergraph, then
\begin{itemize}
\item if $c^{r-2}\alpha<\frac{(r-2)^{r-2}}{(r-1)^{r-1}}$, then  $\mathbb{P}\left(V(\mathcal{H})_{c/d^{1/(r-1)}}\text{ percolates in } \mathcal{H}_{\alpha/d^{1/(r-1)}}\right)\leq \varepsilon$.
\item if $c^{r-2}\alpha>\frac{(r-2)^{r-2}}{(r-1)^{r-1}}$, then $\mathbb{P}\left(V(\mathcal{H})_{c/d^{1/(r-1)}}\text{ percolates in } \mathcal{H}_{\alpha/d^{1/(r-1)}}\right)\geq 1-\varepsilon$.
\end{itemize}
\end{thm}

\begin{rem}
The value of $K$ is simply chosen large enough so that $\log^K(d)$ grows faster than any relevant constant power of $\log(d)$ that appears throughout the proof. We have not attempted to optimise the dependence of $K$ on $r,c$ and $\alpha$. 
\end{rem}
 
Theorem~\ref{hypmainThm} is stronger than Theorem~\ref{mainThmWeak} in two ways. Firstly, it applies to a much wider class of hypergraphs (allowing larger codegrees, neighbourhood intersections, etc.) and, secondly, it implies that the probability of percolation transitions from close to zero to close to one within a small window of the critical probability; i.e.~the process exhibits a sharp threshold. 

\subsection{Connections to Other Bootstrap Processes}
\label{subsec:connections}

The $\mathcal{H}$-bootstrap process is motivated by its connection to the so called \emph{graph bootstrap process} introduced by Bollob\'{a}s~\cite{wsat} in 1968 (under the name ``weak saturation''). Given graphs $G$ and $F$, the \emph{$F$-bootstrap process} on $G$ starts with an initial set of infected edges of $G$ and, at each step, a healthy edge becomes infected if there exists a copy of $F$ in $G$ in which it is the unique healthy edge. Clearly, the $F$-bootstrap process on $G$ is equivalent to the $\mathcal{H}_{G,F}$-bootstrap process where $\mathcal{H}_{G,F}$ is a hypergraph in which each vertex of $\mathcal{H}_{G,F}$ corresponds to an edge of $G$ and the hyperedges of $\mathcal{H}_{G,F}$ are precisely the edge sets of copies of $F$ in $G$. 

The original motivation behind the $F$-bootstrap process stemmed from its connections to the notion of ``saturation'' in extremal combinatorics. Because of this, most of the known results on the $F$-bootstrap process are extremal in nature (see, e.g.~\cite{Alon,Kalai1,Kalai2,MNS,MoshShap,MorrisonNoel}). Balogh, Bollob\'{a}s and Morris~\cite{GraphBootstrap} were the first to analyse the behaviour of the graph bootstrap process relative to a random initial infection. This line of research is motivated by connections between the $F$-bootstrap process and the well-studied \emph{$r$-neighbour bootstrap process} which was introduced by physicists Chalupa, Leath and Reich~\cite{Chalupa} in the late 1970s and has found many applications to modeling real-world propagation phenomena; for more background see, e.g.,~\cite{Aiz, sharp, 3Dim,Cerf,Cerf2,Holroyd,Gnp,PowerLaw,neural,MorrisonNoel}.  The central probabilistic problem for the $F$-bootstrap process in $G$ is to estimate the \emph{critical probability} defined by
\[p_c(G,F):=\inf\left\{p\in (0,1): \mathbb{P}(G_p \text{ percolates})\geq 1/2\right\}\]
where $G_p$ is the graph obtained from $G$ by including each edge of $G$ with probability $p$ independently of one another. Following the initial paper of Balogh, Bollob\'{a}s and Morris~\cite{GraphBootstrap}, probabilistic questions regarding the $F$-bootstrap process have been studied by Gunderson, Koch and Przykucki~\cite{timeGraphBoot}, Angel and Kolesnik~\cite{Angel} and Kolesnik~\cite{kol2}.

The topic of the current paper (i.e. the conjecture of Morris~\cite{MorrisPC}) was initially inspired by a result of Kor\'{a}ndi, Peled and Sudakov~\cite{triadic} which is essentially equivalent to the special case of Theorem~\ref{mainThmWeak} where $\mathcal{H}=\mathcal{H}_{K_n,K_3}$ and $\alpha=1/2$. As an application of Theorem~\ref{mainThmWeak}, we generalise their result to a wider class of graphs. 

For a graph $F$ with at least two edges, the  \emph{$2$-density} of $F$ is defined to be $d_2(F) := \frac{|E(F)|-1}{|V(F)|-2}$.
A graph $F$ with at least two edges is said to be \emph{$2$-balanced} if $d_2(F)\geq  d_2(F')$ for every proper subgraph $F'$ of $F$ with at least two edges. If the inequality is strict for every such $F'$, then we say that $F$ is \emph{strictly $2$-balanced}. Given a graph $F$, observe that $\mathcal{H}_{K_n,F}$ is $|E(F)|$-uniform and $d(n,F)$-regular for some integer $d(n,F)$ such that $d(n,F) = \Theta\left(n^{|V(F)|-2}\right)$ (the constant factor is related to the number of automorphisms of $F$ which fix an edge). We will derive the following result from Theorem~\ref{mainThmWeak}.

\begin{thm}
\label{graphThm}
Let $F$ be a strictly $2$-balanced graph with at least three edges and define $r:=|E(F)|$, $\mathcal{H}:=\mathcal{H}_{K_n,F}$ and $d_n:=d(n,F)$. If $q_n:=\alpha d_n^{-1/(r-1)}$ for some fixed $\alpha>0$, then
\[\lim_{n\to\infty}\left(p_c\left(\mathcal{H}_{q_n}\right)\cdot d_n^{1/(r-1)}\right) = \frac{(r-2)}{\alpha^{1/(r-2)}(r-1)^{(r-1)/(r-2)}}.\]
\end{thm}

We remark that, for $F$ strictly $2$-balanced, Theorem~\ref{graphThm}  can be viewed as a sharp threshold for a variant of the graph bootstrap process where we first ``activate'' each copy of $F$ in $K_n$ with probability $\Theta\left(n^{-1/d_2(F)}\right)$ independently of one another and then, given a random initial set of infected edges of $K_n$, at each step of the process a healthy edge becomes infected if it is the unique healthy edge in an active copy of $F$. 

\subsection{The Differential Equations Method}
\label{subsection:diffMethod}

The main tool in our proof of Theorem~\ref{hypmainThm} is the ``differential equations method'' which was developed to a large extent by Ruci\'{n}ski and Wormald~\cite{RucinskiWormald,RucinskiWormald2}; see also the surveys of Wormald~\cite{Wormald,worm2}. Roughly speaking, the method is described as follows. Suppose $X_0,X_1,\dots,X_N$ is a discrete stochastic process. For example, given a hypergraph $\mathcal{H}$, consider the \emph{random greedy independent set algorithm} in which $X_0:=\emptyset$ and, for $i\geq1$, the set $X_{i+1}$ is obtained from $X_i$ by adding one vertex chosen uniformly at random from all vertices $v$ of $V(\mathcal{H})\setminus X_i$ such that $X_i\cup\{v\}$ contains no hyperedge (as long as such a vertex exists). Suppose we wish to estimate a numerical parameter $\varphi\left(X_i\right)$: e.g.~the number of hyperedges of $\mathcal{H}$ intersecting $X_i$ on exactly three vertices. If we are able to obtain good bounds on the expected and maximum change of our parameter at each step, then we could apply a martingale concentration inequality to bound $\varphi\left(X_i\right)$. 

However, the change in $\varphi(X_i)$ often depends on other parameters which must themselves be controlled (i.e. concentrated or bounded) in order to obtain useful bounds on the change in $\varphi(X_i)$. The key to applying the method is to find a collection of random variables containing $\varphi(X_i)$ whose ``one step changes'' can be expressed in terms of other variables in the collection. These expressions form a system of
difference equations and the solution to the initial value problem for the corresponding system
of differential equations gives a natural guess for the ``expected trajectory" of the variables.\footnote{In this paper, we will not explicitly state or solve any actual differential equations; the expected trajectory will instead be inferred from some simple heuristics.} 
The last step is to apply a martingale concentration inequality and a union bound to prove that all of the variables in the collection  (including the variable $\varphi(X_i)$ that we care about) are concentrated around their expected trajectory.

Some recent applications of this method include the analysis of the \emph{$H$-free process}, which can be thought of as the random greedy independent set algorithm applied to the hypergraph $\mathcal{H}_{K_n,H}$. Bohman~\cite{TriangleFreeBohman} used the differential equations method to determine the size of the largest independent set in the graph produced by the $K_3$-free process up to a constant factor. Far more detailed analysis of the $K_3$-free process was famously achieved independently by Bohman and Keevash~\cite{trifreeBK} and Fiz Pontiferos, Griffiths and Morris~\cite{trifreeFGM} by further developing the method and exploiting the ``self-correcting'' nature of the process; this work yielded precise asymptotics of the independence number and the best known lower bound on the Ramsey number $R(3,k)$. 

Bohman and Keevash~\cite{HFreeEvolution} have also used the differential equations method to analyse the $H$-free process for strictly $2$-balanced graphs $H$. Even more generally, Bennett and Bohman~\cite{BennettBohman} used the method to show that the random greedy algorithm produces an independent set of size $\Omega\left(N\left(\log{N}/d\right)^{1/(r-1)}\right)$ when applied to any hypergraph $\mathcal{H}$ satisfying the hypotheses of Theorem~\ref{mainThmWeak}. Our result extends the theorem of Kor\'andi, Peled and Sudakov~\cite{triadic} in a way which is roughly analogous to the generalisation of~\cite{TriangleFreeBohman} in~\cite{BennettBohman}, except that, due to the nature of the bootstrap process, we are able to simultaneously obtain a more precise result and handle larger codegrees and neighbourhood intersections. For other recent applications of the method, see~\cite{BennettBohmanMatch,Squares}.

\subsection{Structure of the Paper}

The rest of the paper is organised as follows. In the next section, we give a detailed outline of the proof of Theorem~\ref{hypmainThm}. While it does not contain any proofs, this section is perhaps the most crucial to understanding the paper as it includes a description of the discrete processes to be analysed in later sections and most of the key definitions and statements to be proved. In Section~\ref{sec:prob} we state the main probabilistic tools (i.e. concentration inequalities) that we will apply. A few preliminary lemmas will be proved in Section~\ref{sec:prelims} before moving on to the main meat of the proof. The proof of Theorem~\ref{hypmainThm} is divided into four parts which are contained in Sections~\ref{sec:timeZero},~\ref{sec:diff},~\ref{sec:phaseTwoSub} and~\ref{sec:super}.  Finally, in Section~\ref{balancedGraphs} we use Theorem~\ref{mainThmWeak} to derive  Theorem~\ref{graphThm} and a generalisation of it to ``strictly $k$-balanced hypergraphs'' (defined in the section itself).

\section{Outline of the Proof}\label{sec:outline}

Rather than attempting to apply the differential equations method to the $\mathcal{H}_q$-bootstrap process directly (which is completely deterministic and, therefore, ill-suited to the method), we will analyse two different random processes in which the hypergraph $\mathcal{H}_q$ is revealed iteratively and the infection spreads in a way which depends on the structure of $\mathcal{H}_q$ unveiled so far. These processes, to be defined shortly, are equivalent to the  $\mathcal{H}_q$-bootstrap process in the sense that the final set of infected vertices is the same. The purpose of this section is to provide a fairly detailed outline of the proof of Theorem~\ref{hypmainThm}; in particular, we will describe the two random processes that we will analyse and will define (and motivate) the variables that we wish to track.\footnote{Throughout the paper, when we say that we \emph{track} a random variable, it will always mean one of two things: either (a) we show that it is concentrated or (b) we show that it satisfies a certain upper or bound with high probability (i.e.~with probability tending to 1 as $V(\mathcal{H})\rightarrow \infty$).}

The analysis is divided into two phases. The first phase involves an application of the differential equations method and is done in essentially the same way regardless of whether $c^{r-2}\alpha$ is smaller or larger than $\frac{(r-2)^{r-2}}{(r-1)^{r-1}}$. The analysis in the second phase differs depending on which of these cases we are in. Throughout both phases, we will track a family of variables which will allow us to determine, with high probability, whether or not the initial infection percolates. 

Before diving deeply into the details, we fix some parameters and notation that will be used throughout the paper. Let $r,c,\alpha,\beta>0$ be fixed, let $K$ be large with respect to $r,c$ and $\alpha$ and, for large $d$, let $\mathcal{H}$ be an $r$-uniform $\left(d,\log^{-K}(d),d^\beta\right)$-well behaved hypergraph (defined in Definition~\ref{hypwellBDef}). Set $N:= |V(\mathcal{H})|$.
Note that, by property~\ref{Nnottoobig} of Definition~\ref{hypwellBDef}, we have $N\leq d^\beta$. Since each set of size $r$ contains exactly $\binom{r}{\ell}$ sets of size $\ell$,  by the pigeonhole principle, for $1\leq \ell\leq r-1$ we have,
\[\Delta_\ell(\mathcal{H})\geq\frac{|E(\mathcal{H})|\binom{r}{\ell}}{\binom{N}{\ell}}\geq \frac{\delta(\mathcal{H}) N \binom{r}{\ell}}{r\binom{N}{\ell}}.\]
Rearranging this expression gives
$$N^{\ell - 1} \ge \frac{a \delta(\mathcal{H})}{\Delta_{\ell}(\mathcal{H})},$$
for some constant $a$.
 Now applying conditions~\ref{almostRegular} and~\ref{codegBound} of Definition~\ref{hypwellBDef}, we obtain
\[N^{\ell - 1} \ge \frac{a d(1 - \rho)}{d^{1 - \frac{\ell - 1}{r-1}} \rho} = ad^{\frac{\ell - 1}{r-1}} \frac{1 - \rho}{\rho}.\]
Setting $\rho:= \log^{-K}(d)$ and $\ell = 2$ gives
\begin{equation}\label{Nnotsmall}
N = \Omega\left( d^{1/(r-1)}\log^{K} (d)\right).
\end{equation}
In particular,
\begin{equation}\label{Nisd}
\log(N) = \Theta\left(\log(d)\right).
\end{equation} 
In what follows, we will write $p:= c\cdot d^{-1/(r-1)}$ and $q:= \alpha \cdot d^{-1/(r-1)}$.

\subsection{The First Phase}

Here we define the random hypergraph process that we will analyse during the first phase. At time $m$, we will have a hypergraph $\mathcal{H}(m)$ formed by the unsampled hyperedges of $\mathcal{H}$ and a set of infected vertices $I(m)$ (where $\mathcal{H}(m)$ and $I(m)$ will be formally defined below). Both $\mathcal{H}(m)$ and $I(m)$ depend on the outcomes of the process up to this point. 

Due to the nature of the $\mathcal{H}$-bootstrap process, it should come as no surprise that the most important variable for us to track is the number of hyperedges containing a unique healthy vertex; to this end, define
\begin{equation}\label{qdef}
Q(m):= \{e \in E(\mathcal{H}(m)): |e \setminus I(m)| = 1\}.
\end{equation}
We refer to the hyperedges in $Q(m)$ as \emph{open} hyperedges. In what follows, to \emph{sample} a hyperedge $e$ of $\mathcal{H}$ means to determine whether or not it is contained in $\mathcal{H}_q$. The sampling of the hyperedge $e$ is said to be \emph{successful} if $e$ is contained in $\mathcal{H}_q$.

\begin{process}
At time zero, we let $I(0):= V(\mathcal{H})_p$  be the set of initially infected vertices and let $\mathcal{H}(0):=\mathcal{H}$. Now, for $m\geq0$, given $I(m)$ and $\mathcal{H}(m)$, we obtain $I(m+1)$ and $\mathcal{H}(m+1)$ in the following way: if $Q(m)=\emptyset$, then set $\mathcal{H}(m+1):=\mathcal{H}(m)$ and $I(m+1):=I(m)$; otherwise,  choose an open hyperedge $e$ from $Q(m)$ uniformly at random and sample it. If the sampling is successful, then set $I(m+1):=I(m)\cup \{v\}$ where $v$ is the unique vertex of $e\setminus I(m)$ and, otherwise, set $I(m+1):=I(m)$. In either case, set $\mathcal{H}(m+1):=\mathcal{H}(m)\setminus\{e\}$.  
\end{process} 

As a slight abuse of notation, for a collection $X(m)$ of subhypergraphs or vertices of $\mathcal{H}(m)$ we will often write $|X(m)|$ simply as $X(m)$ (for example, we will write $Q(m)$ to mean $|Q(m)|$ and $I(m)$ to mean $|I(m)|$). In all cases, it should be clear from context whether we are referring to the collection $X(m)$ or its cardinality.

We will run the first phase process up to some time $M$ at which point, with high probability, $Q(M)$ will be either large enough or small enough to be able to determine if the infection is likely to percolate by other methods in the second phase (summarised in Subsection~\ref{secsum}). Our goal in the first phase will be to show that $Q(m)$ stays close to its expected trajectory with high probability. As was described in Subsection~\ref{subsection:diffMethod}, this will involve finding a suitable collection of random variables containing $Q(m)$ whose ``one step changes'' depend on other variables in the collection and to apply martingale concentration inequalities to get control over all of these variables simultaneously. 
 
It is sometimes more convenient to think of our random variables as depending on a continuous variable $t$ rather than the discrete variable $m$. The scaling that we will use when moving between discrete and continuous settings is 
 \begin{equation}\label{tdef}
 t= t_m := m/N,
 \end{equation}
 for $m\geq 0$. Throughout the paper we will alternate between the discrete and continuous settings without further comment. In the first phase, we will only consider values of $m$ up to 
 \begin{equation}
 \label{roughMaxTime}
 \begin{cases}O\left(N\right) & \text{if }c^{r-2}\alpha<\frac{(r-2)^{r-2}}{(r-1)^{r-1}},\\ 
 O\left(N\log(d)\right) & \text{if }c^{r-2}\alpha>\frac{(r-2)^{r-2}}{(r-1)^{r-1}}.
 \end{cases}
 \end{equation}
The fact that $m$ does not get too large during the first phase will be used in some of the heuristic discussions which follow. 
 
 At time zero, each vertex is infected with probability $p$ and, provided that $Q(m-1)\neq \emptyset$, at the $m$th step of the process a new vertex is infected with probability $q$. So if $Q(m-1)\neq\emptyset$ then we would expect
 $$I(m) \approx pN + mq.$$
Letting $M$ be the number of steps we run the first phase for, using the Chernoff bound (Theorem~\ref{Chernoff}) we will prove the following (see Lemma~\ref{Zprop} of Section~\ref{sec:diff}).
  
  \begin{prop}\label{Ibound}
  For $0 \le m \le M$, with probability at least $1 - N^{-\Omega\left(\sqrt{\log N}\right)}$,
  $$I(m) = O\left(\log N \cdot Nd^{-1/(r-1)}\right).$$
  \end{prop}

We will use the fact that $\mathcal{H}$ is well behaved to show that $I(m)$ behaves similarly to a random infection in which each vertex is infected independently with probability $p + qt_m$ (which we shall call a \emph{uniformly random infection of density} $p + qt_m$), in the sense that $Q(m)$ is close to the value that one would expect in this case. First, let us determine the value of $Q(m)$ which we would expect if $I(m)$ were a uniformly random infection of density $p + qt_m$. If this were the case, then a particular hyperedge of $\mathcal{H}$ would contain $r-1$ infected vertices with probability $r\cdot(p+qt_m)^{r-1}(1 - (p+qt_m))$ which is approximately $r\cdot(p+qt_m)^{r-1}$ since $p+qt_m = o(1)$. Since $\mathcal{H}$ is roughly $d$-regular, we have $|E(\mathcal{H})| \approx d \cdot N/r$. Also, recall that, at each step, we sample precisely one open hyperedge which is immediately discarded from the hypergraph. Thus, we would expect
\begin{equation}\label{Qexpect}
Q(m) \approx r\cdot(p+qt_m)^{r-1}|E(\mathcal{H})| - m \approx [(c + \alpha t_m)^{r-1} - t_m]\cdot N.
\end{equation}
The main point of the first phase is to show that, up to a small error, $Q(m)$ follows this trajectory (see Lemma~\ref{Qrough}). Before stating this more precisely, let us discuss the motivation behind the choice of $M$. 

Define
\begin{equation}\label{ft}
\gamma(t):= (c + \alpha t)^{r-1} - t.
\end{equation}
Observe that
$$\gamma'(t) = \alpha(r-1)(c + \alpha t)^{r-2} - 1.$$
Therefore, since $c, \alpha >0$, we have that $\gamma'(t)$ has exactly one real root if $r$ is odd and two real roots (one positive, one negative) if $r$ is even. The rightmost root of $\gamma'(t)$ is a local minimum for $\gamma(t)$ located at
\[t_{\min}:=\frac{1}{\alpha}\left(\left(\frac{1}{\alpha(r-1)}\right)^{1/(r-2)} - c\right).\]
Now, 
\begin{align*}
\gamma\left(t_{\min}\right) &= \left(\frac{1}{\alpha(r-1)}\right)^{(r-1)/(r-2)} - \left(\frac{1}{\alpha}\left(\frac{1}{\alpha(r-1)}\right)^{1/(r-2)} - \frac{c}{\alpha}\right)\\
& = \left(\frac{2-r}{(\alpha(r-1))^{(r-1)/(r-2)}}\right) + \frac{c}{\alpha}
\end{align*}
which is negative if and only if $c^{r-2}\alpha<\frac{(r-2)^{r-2}}{(r-1)^{r-1}}$. 

\begin{figure}
\centering
\begin{tikzpicture}
      \draw[->] (-1,0) -- (7,0) node[right] {$t$};
      \draw[->] (0,-1) -- (0,4) node[above] {$f(t)$};
      \draw[scale=2.0,domain=0:3.5,smooth,variable=\x,blue] plot ({\x},{(\x - 1)*(\x - 3)*0.3}); \draw[scale=2.0,domain=0:3.5,smooth,variable=\x,red] plot ({\x},{(\x - 1.2)*(\x - 3.2)*0.25 + 0.3});

\end{tikzpicture}
\caption{The red curve gives an example of the trajectory of $Q(m)$ when $r=3$ and $c^{r-2}\alpha > \frac{(r-2)^{r-2}}{(r-1)^{r-1}}$. The blue curve is an example of the trajectory of $Q(m)$ when $r=3$ and $c^{r-2}\alpha < \frac{(r-2)^{r-2}}{(r-1)^{r-1}}$.}
\label{pqdep}
\end{figure}
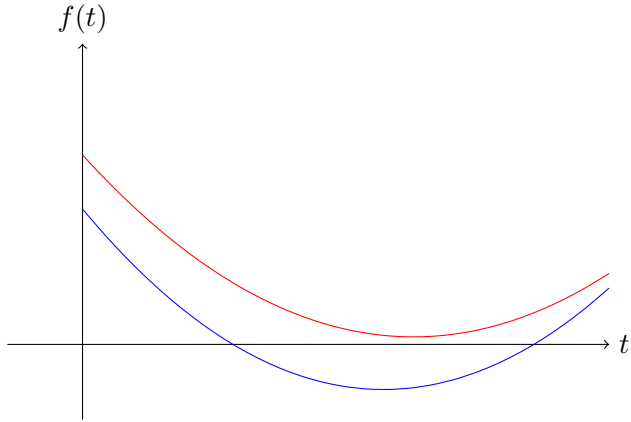

From this and the fact that $\gamma(0)>0$, we see that, if $c^{r-2}\alpha<\frac{(r-2)^{r-2}}{(r-1)^{r-1}}$, then $\gamma(t)$ has precisely two distinct positive real roots, say $T_0$ and $T_1$ where $0<T_0<t_{\min}<T_1$. Coming back to the random process, this tells us that if $c^{r-2}\alpha<\frac{(r-2)^{r-2}}{(r-1)^{r-1}}$, then we expect the number of open hyperedges to become very small as $t_m$ approaches $T_0$ from the left. What we will do in this case is track our variables until $\gamma(t)< \zeta$ where $\zeta$ is a constant chosen small with respect to $r,c$ and $\alpha$; the value of $\zeta$ is given in Definition~\ref{zetaDef}. At this point we initiate the second phase in which we prove that, with high probability, percolation does not occur. 

On the other hand, if $c^{r-2}\alpha>\frac{(r-2)^{r-2}}{(r-1)^{r-1}}$, then $\gamma\left(t_{\min}\right)>0$ and $\gamma(t)$ has no positive real roots. Since $p + qt_m = o(1)$ for all values of $m$ that we consider in the first phase (see \eqref{Mdef}), we expect our supply of open hyperedges not to run out (in fact, when  $t_m> t_{\min}$, we expect the number of open hyperedges to be typically increasing). What we will do in this case is track the above variables until step $N\left\lfloor \frac{\log(N)}{\alpha}\right\rfloor = O\left(N\log(d)\right)$, at which point the number of open hyperedges will be large enough that we can deduce that percolation occurs with high probability in the second phase. See Figure~\ref{pqdep} for examples of how the trajectory of $Q(m)$ depends on the relationship between $\alpha$, $c$ and $r$.

\begin{rem}
Let us briefly discuss why we have chosen to focus on values of $p$ and $q$ of order $d^{-1/(r-1)}$. As we have argued above, if the infection at time $m$ resembles a uniform infection of density $p+qt_m$, then we expect the variable $Q(m)$ to be roughly
\[N[d(p+qt_m)^{r-1} - t_m].\]
The nice thing about considering $p$ and $q$ of order $d^{-1/(r-1)}$ is that the expression inside the square brackets becomes a function of $t_m$ only. Thus, as long as $t_m$ is bounded by a constant, open hyperedges are both being created and discarded at a constant rate independent of $d$. It would of course be natural (and interesting) to consider more general values of $p$ and $q$, but one would likely require a different approach.\footnote{Actually, one can apply Theorem~\ref{hypmainThm} directly to get a result in the case that $d^{-1+o(1)}\leq q\ll d^{-1/(r-1)}$. Choose $q'$ and $q''$ so that $q'q'' = q$ and $q''= \left(q'd\right)^{-1/(r-1)}$. Given a $d$-regular hypergraph $\mathcal{H}$ satisfying some appropriate conditions, one can deduce that, with high probability, the random hypergraph $\mathcal{H}_{q'}$ satisfies the conditions of Theorem~\ref{hypmainThm} with $(1+o(1))q'd$ playing the role of $d$. Thus, we get a sharp threshold for bootstrap percolation in $\left(\mathcal{H}_{q'}\right)_{q''} = \mathcal{H}_q$. The case $q\gg d^{-1/(r-1)}$, on the other hand, is likely to require different ideas.}
\end{rem}

As we said above, the main aim of the first phase is to show that for $0 \le m \le M$, the value of $Q(m)$ is within a small error term of $\gamma(t_m)\cdot N$.  The following function describes the relative error that we will allow ourselves in these bounds,
\begin{equation}\label{errorterm}
\epsilon(t):=\Erbt.
\end{equation}
Note that $\epsilon(t)=o(1)$ when $t=O\left(\log(d)\right)$. In what follows, we write an interval of the form $[(1-\epsilon)g(t),(1+\epsilon)g(t)]$ as $(1\pm\epsilon)g(t)$ for brevity. To summarise, we track the process for $M$ steps, where
\begin{equation}
\label{Mdef}
M:=
\begin{cases}
\min\{m\geq0: (1 + 4\epsilon(t_m))\gamma(t)<\zeta\} & \text{if }c^{r-2}\alpha < \frac{(r-2)^{r-2}}{(r-1)^{r-1}}, \\N\left\lfloor \frac{\log(N)}{\alpha}\right\rfloor & \text{if }c^{r-2}\alpha > \frac{(r-2)^{r-2}}{(r-1)^{r-1}}.
\end{cases}
\end{equation} 
Define
\begin{equation}
\label{Tdef}
T:= M/N.
\end{equation}

\begin{rem}\label{awayfromzero}
Observe that, by definition of $M$, on the interval $[0,T]$ the function $\gamma(t)$ is always bounded away from zero by a function of $r,c$ and $\alpha$. 
\end{rem}

We are now ready to formally state the bounds we will prove on $Q(m)$ in the first phase. 
\begin{lem}\label{Qrough}
With high probability the following statement holds. For all $0 \le m \le M$,
$$Q(m) \in (1 \pm 4\epsilon(t_m))\gamma (t_m)\cdot N.$$
\end{lem}

One way to prove that $Q(m)$ is controlled in this way, or indeed to prove bounds for any of our variables, involves determining their expected and maximum change (conditioned on what has previously occurred during the process) at each time step and applying a martingale concentration inequality. Before thinking in more detail about the expected change of $Q(m)$ we introduce some notation that will be helpful. 

For each $v \in V(\mathcal{H})\setminus I(m)$ we write
\begin{equation}
\label{Qvdefn}
Q_v(m):= \{e \in Q(m): e \setminus I(m)=\{v\}\}.
\end{equation}
For $u \not= v \in V(\mathcal{H})$, the sets $Q_u(m)$ and $Q_v(m)$ are disjoint. Note that, $$Q(m) = \bigcup_{v \in V(\mathcal{H})\setminus I(m)}Q_v(m).$$

Let us now think about the expected change of $Q(m)$. Firstly, which open hyperedges from $Q(m)$ are not present in $Q(m+1)$? At each step, the hyperedge $e$ we sample from $\mathcal{H}(m)$ is deleted and is not present in $\mathcal{H}(m+1)$.  Also, with probability $q$, the unique healthy vertex $v$ of $e$ becomes infected and so all the hyperedges in $Q(m)$ whose unique healthy vertex is $v$ are no longer open. This results in a loss of $(Q_v(m) - 1)$ hyperedges (in addition to $e$). Now let us consider how we gain a new open hyperedge. This occurs when some hyperedge $e$ is successfully sampled, and the vertex $v$ of $e$ that becomes infected is contained in a hyperedge $e'$ with exactly $r-2$ infected vertices (the hyperedge $e'$ will now be open). 

Observe that for each vertex $v \in V(\mathcal{H})\setminus I(m)$, the probability an open hyperedge containing $v$ is sampled is $Q_v(m)/Q(m)$. Given the above discussion, we can express the expectation of $Q(m+1) - Q(m)$ conditioned on $\mathcal{H}(m)$ and $I(m)$ as
\begin{equation}
\label{roughexpq}
-1 - \sum_{v \in V(\mathcal{H})\setminus I(m)}q\cdot \frac{Q_v(m)}{Q(m)}(Q_v(m) - 1) + \sum_{v \in V(\mathcal{H})\setminus I(m)}q \cdot \frac{Q_v(m)}{Q(m)}Y_v^{r-2}(m),
\end{equation}
where
$$Y_v^{r-2}(m) := \{e \in \mathcal{H}(m): |e \cap I(m)| = r-2, v \in e\}.$$
So to be able to determine \eqref{roughexpq} we can see that we would need to have control over $Y_v^{r-2}(m)$. So let us consider how a new copy of $Y^{r-2}$ is created at a time step. One way a member of $Y^{r-2}_v(m+1)\setminus Y^{r-2}_v(m)$ can be created is from a pair of hyperedges $\{e_1,e_2\} \subseteq E(\mathcal{H})$ where: $v \in e_1 \setminus e_2$; $e_2$ is open; $e_1$ has exactly $r-3$ infected vertices and intersects $e_2$ on its unique healthy vertex; and $e_2$ is successfully sampled at time $m$. Thus to determine the expected change of $Y_v^{r-2}(m)$, we need to also have control over this family $Z$ of pairs. And similarly, to do this there are a number of other variables that we must keep track of.

To summarise this train of thought, to prove Lemma~\ref{Qrough} we must have control over a number of families of variables; in particular, variables of the two types described above. We briefly remark that, in our proof, we do not explicitly calculate the expected change of $Q(m)$ in the manner we have alluded to above. In fact, we show that having control over a more general family of variables will imply the required bounds on $Q(m)$ in a different way (see Lemma~\ref{QfromOthers}). However to prove bounds on our other variables, we do calculate their expected and maximum changes. The point of performing this thought exercise on $Q(m)$ was to illustrate its interdependence on a number of other variables and to motivate the following discussion.

In order to formally describe the families of variables that we wish to track, it is helpful to introduce a few definitions. Each variable that we wish to control counts the number of ``copies'' of some particular subhypergraph $\mathcal{F}\subseteq \mathcal{H}(m)$ such that these copies of $\mathcal{F}$ are ``rooted'' at a particular subset $S \subseteq V(\mathcal{H})$ (in the sense that these vertices are contained within the copy) and some particular vertices of these copies are infected (i.e. contained in $I(m)$). We begin by introducing some notation to describe the particular structures (which we call \emph{configurations}) we are interested in counting ``copies'' of.

\begin{defn}
A \emph{configuration} is a triple $X=(\mathcal{F},R,D)$, where $\mathcal{F}$ is an $r$-uniform hypergraph in which every vertex is contained in at least one hyperedge and $R$ and $D$ are disjoint subsets of $V(\mathcal{F})$. The vertices of $R$ are called the \emph{roots} of $X$, the vertices of $D$ are called the \emph{marked} vertices of $X$, and the vertices of $V(\mathcal{F})\setminus (D \cup R)$ are called the \emph{neutral} vertices of $X$. 
\end{defn}

Now that we have a good way to describe the things we are interested in counting, we will formally define what we mean by a \emph{copy} of a configuration.

\begin{defn}
\label{copyDef}
For $m\geq0$, given a configuration $X=(\mathcal{F},R,D)$ and a set $S\subseteq V(\mathcal{H})$, a \emph{copy} of $X$ in $\mathcal{H}(m)$ \emph{rooted} at $S$ is a subhypergraph $\mathcal{F}'$ of $\mathcal{H}(m)$ such that there exists an isomorphism $\phi:\mathcal{F}\to \mathcal{F}'$ with $\phi(R)=S$ and $\phi(D) \subseteq I(m)$. Also define $X_S(m)$ to be the collection of copies of $X$ in $\mathcal{H}(m)$ rooted at $S$. We denote $X_{\{v\}}(m)$ by $X_v(m)$ for $v\in V(\mathcal{H})$.
\end{defn}

Take note that a copy of a configuration $(\mathcal{F},R,D)$ in $\mathcal{H}(m)$ can contain elements of $I(m)$ apart from those in $\phi(D)$. In particular, it is even possible for the set $\phi(R)$ to contain elements of $I(m)$ (despite the fact that $R$ and $D$ are disjoint).

Before discussing specific families of configurations, let us discuss heuristically how many copies we expect there to be of some fixed configuration $X = (\mathcal{F},R,D)$ in $\mathcal{H}(m)$ rooted at $S \subseteq V(\mathcal{H})$. If $\hat{X}_S$ is the number of copies of $(\mathcal{F},R,\emptyset)$ rooted at $S$ in $\mathcal{H}$, then if $I(m)$ is a uniformly random infection of density $p + qt_m$, we would expect 
$$X_S(m) \approx \hat{X}_S \cdot(p + qt_m)^{|D|},$$
as each vertex is independently infected with probability $p + qt_m$. That is, each infected vertex contributes a factor of at most $\plog(d) \cdot d^{-\frac{1}{r-1}}$ (as $t_m = O\left(\log(d)\right)$).

Now, heuristically, how do we bound $\hat{X}_S$? In our proof, for some families of configurations we will only require an upper bound, but for some we need to be more careful and also need a lower bound. All the configurations $X=(\mathcal{F},R,D)$ that we are interested in tracking during the first phase will satisfy the following properties: $\mathcal{F}$ is connected and contains at most $r+1$ hyperedges, no vertex of $\mathcal{F}$ is contained in the intersection of more than two hyperedges, every root is contained in a unique hyperedge, and $|R| \ge 1$. 

So suppose $X$ satisfies these conditions. To find a bound on $\hat{X}_S$, we can break $\mathcal{F}$ up into its hyperedges $e_1,\ldots,e_k$, where $|e_1 \cap R|\ge 1$ and each $e_i$ intersects $\bigcup_{\ell <i}e_{\ell}$, and bound the number of choices for each hyperedge using properties~\ref{codegBound} and~\ref{neighSim} of Definition~\ref{hypwellBDef}. Let $S_1,\ldots,S_k$ be a fixed partition of $S$ such that $|S_i|= \left|R \cap e_i \setminus \bigcup_{\ell <i}e_{\ell}\right|$. We will bound the number of members of $\hat{X}_S$ in $\mathcal{H}$ such that $S_i \subseteq e_i \setminus \bigcup_{\ell < i}e_{\ell}$. As there are $O(1)$ such partitions of $S$, the total number of members of $\hat{X}_S$ will be a constant factor away from this. 

First consider the number of ways to choose $e_1$. By conditions~\ref{maxDeg},~\ref{almostRegular} and~\ref{codegBound} of Definition~\ref{hypwellBDef}, if $|R|= 1$, then the number of choices is within $(1 \pm \log^{-K}(d))\cdot d$ and, if $|R| \ge 1$, then it is at most $d^{1 - \frac{|R|-1}{r-1}}\cdot \log^{-K}(d)$. Similarly, we can then bound the number of ways to choose $e_2$. By our choice of hyperedge order, $e_2$ intersects $e_1$. Given a choice of $e_1$, there are $O(1)$ ways $e_2$ can intersect it. Defining $b:= |e_2 \cap (R \cup e_1)|$ (by assumption on hyperedge order $b \ge 1$), by conditions~\ref{maxDeg} and~\ref{codegBound} of Definition~\ref{hypwellBDef} there are at most 
\[O\left(\Delta_{b}(\mathcal{H})\right)  = O\left(d^{1 - \frac{b-1}{r-1}}\right) = O\left(d^{\frac{r-b}{r-1}}\right)\]
choices for $e_2$. When $2 \le b \le r-1$, using condition~\ref{codegBound} of Definition~\ref{hypwellBDef} gives a stronger bound of
\[O\left(\Delta_{b}(\mathcal{H})\right)  = O\left(d^{\frac{r-b}{r-1}}\log^{-K}(d)\right)\]
choices for $e_2$.

Given these bounds, the number of choices for $e_2$ can be thought of as being $O\left(d^{\frac{a}{r-1}}\right)$, where $a$ is the number of vertices of $e_2$ that are not in $e_1$ or $R$ (i.e.~the number of ``new'' vertices). We can bound the number of choices for $e_3,\ldots,e_k$ analogously. A more careful version of this argument will be applied later to give the bound in Lemma~\ref{Xcount}.

So, heuristically, for most configurations $X$, up to a $\plog(d)$ factor we generally expect there to be about $d^{\frac{|V(\mathcal{F})| - |R| - |D|}{r-1}}$ copies of $X$ rooted at $S$ in $\mathcal{H}(m)$. One way of thinking about this is to imagine each hyperedge contributes a factor of $d$, but for each vertex that is either in the intersection of two hyperedges or not neutral we lose a factor of $d^{\frac{1}{r-1}}$ (up to some powers of $\log(d)$). Alternately, (again up to some powers of $\log(d)$) we get a factor of $d^{\frac{1}{r-1}}$ for each neutral vertex in the configuration. It will be helpful to bear this rough heuristic in mind throughout the calculations which come later. 

We now introduce our most important family of configurations, the $Y$ configurations. These are a generalisation of the two variables $Y_v^{r-2}$ and $Z$ discussed above. The control we have over these more general variables in $\mathcal{H}(m)$ dictates the bounds we can prove on $Q(m)$ (see Lemma~\ref{QfromOthers}) and on the $Y$ configurations in $\mathcal{H}(m+1)$. The two further sets of variables we will discuss below (see Definitions~\ref{usefulDef} and~\ref{type1}) do affect how the $Y$ configurations behave, but due to the codegree conditions on $\mathcal{H}$ (see Definition~\ref{hypwellBDef}), we can ensure that they only contribute lower order terms.

In general, the $Y$ configurations consist of a hyperedge $e$ containing a root and a fixed number of marked vertices, with open hyperedges that are disjoint from one another and only intersect $e$ on their unique unmarked vertex. See Figure~\ref{Yexamp} for a visualisation of some of these configurations in the case $r=4$. See also Figure~\ref{Ycopies} for some examples of {copies} of $Y$ configurations in the case $r=6$.

\begin{figure}[htbp]
\centering
\includegraphics[width=1\textwidth]{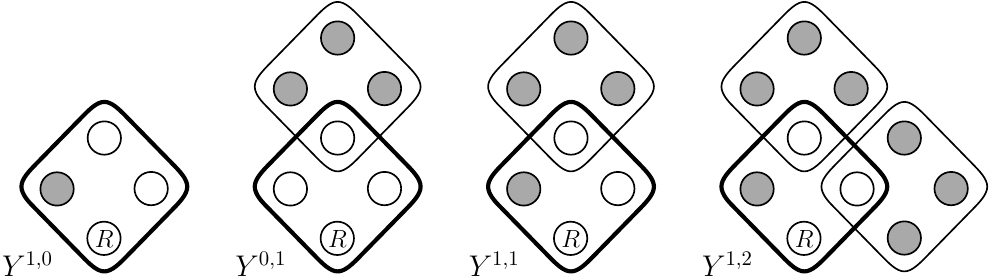}
\caption{Some examples of $Y$ configurations in the case $r=4$. The marked vertices are shaded, the neutral vertices are white, the root is labelled with an $R$ and the central hyperedge is drawn with a thick outline.}
\label{Yexamp}
\end{figure}

\begin{figure}[htbp]
\centering
\includegraphics[width=1\textwidth]{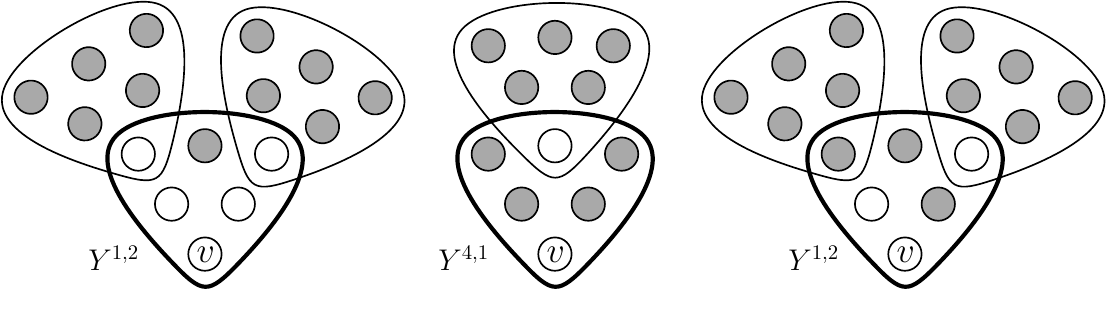}
\caption{Some examples of {copies} of $Y$ configurations rooted at $v$ in the case $r = 6$. The infected vertices are shaded, the healthy vertices are white and the central hyperedge is drawn with a thick outline. Each copy is labelled by which configuration it is a copy of. Notice that a copy of a configuration could have more infections than marked vertices in the configuration itself, as is demonstrated by the third example.}
\label{Ycopies}
\end{figure}
We now formally define the family of $Y$ configurations.

\begin{defn} 
\label{Ydef}
For non-negative integers $i$ and $j$ such that $i+j\leq r-1$, let $Y^{i,j}$ denote the configuration $(\mathcal{F},R,D)$ such that $\mathcal{F}$ is a hypergraph containing a hyperedge $e$, called the \emph{central} hyperedge, such that
\begin{enumerate}[(a)]
\item \label{Y:iInfected}$e$ contains exactly $i$ marked vertices,
\item \label{Y:uniqueRoot}there is a unique root and $e$ is the only hyperedge of $\mathcal{F}$ containing the root,
\item \label{Y:jNonCentral}$\mathcal{F}$ has exactly $j$ non-central hyperedges,
\item \label{Y:hangsOff}for each non-central hyperedge $e'$ we have $|e\cap e'|=1$ and the unique element of $e\cap e'$ is neutral,
\item \label{Y:hangingOpen} every vertex of $V(\mathcal{F})\setminus e$ is marked, and
\item \label{Y:nonCentralDisjoint}no two non-central hyperedges intersect one another.
\end{enumerate}
\end{defn}
Observe that for $v\in V(\mathcal{H})\setminus I(m)$, we have that $Y_v^{r-1,0}(m)$  is precisely the set of open hyperedges in which $v$ is the unique healthy vertex; i.e. $Y_v^{r-1,0}=Q_v(m)$. Given that $Q(m) \approx \gamma (t_m) \cdot N$, if the open hyperedges were distributed uniformly among the healthy vertices, then for each $v \in V(\mathcal{H})\setminus I(m)$ we would expect
\begin{equation}
\label{expectedQv}
Q_v(m) \approx \gamma (t_m),
\end{equation}
 since by Proposition~\ref{Ibound}, at time $t_m < T$, we have $I(m) = o(N)$ with high probability (so at time $t_m$ there are $(1-o(1))N$ healthy vertices). However, when $M = O(N)$, the quantity $\gamma(t_m)$ is constant, and so we cannot hope to prove that $Q_v(m)$ is concentrated around $\gamma (t_m)$. However, as we will see in a moment, we should be able to prove concentration for $Y_v^{i,j}(m)$ when $i+j \not= r-1$. This is why we need to track all of the configurations $Y_v^{i,j}(m)$ individually and cannot just bound $Y_v^{i,j}(m)$ by $Y_v^{i,0}(m)\cdot \binom{r-1-i}{j} \cdot Q^j$, where $Q$ is an upper bound on $Q_v(m)$ which holds for all $v \in V(\mathcal{H})\setminus I(m)$ with high probability. That is, we would not be able to get a good enough bound on  $Q$ to prove bounds as tight as we would like on $Y_v^{i,j}(m)$.

Now we discuss how we expect the variables $Y_v^{i,j}(m)$ to behave. Consider first the variable $Y_v^{i,0}(m)$ for $0\leq i\leq r-2$ and $v\in V(\mathcal{H})$. By properties~\ref{maxDeg} and~\ref{almostRegular} of Definition~\ref{hypwellBDef}, every vertex has degree between $\left(1-o(1)\right)d$ and $d$. Therefore if $I(m)$ is a uniformly random infection of density $p + qt_m$, then we would expect
\begin{equation}\label{Yi0rough}
Y_v^{i,0}(m) \approx \binom{r-1}{i}d(p + qt_m)^i = \binom{r-1}{i}(c + \alpha t_m)^i d^{1 - \frac{i}{r-1}}.
\end{equation}

Now, let us consider $Y_v^{i,j}(m)$ for $j\geq1$. One can express $Y_v^{i,j}(m)$ as the sum over all $\mathcal{F}'\in Y_v^{i,0}(m)$ and all subsets $U$ of $V(\mathcal{F}')\setminus  \{v\}$ with $|U|=j$ and $|\left(V(\mathcal{F}')\cap I(m)\right)\setminus U|\geq i$ of the number of ways to choose one open hyperedge rooted at each element of $U$ in such a way that (a) no two such hyperedges intersect and (b) each of them intersects $V(\mathcal{F}')$ on exactly one vertex. Given \eqref{Yi0rough}, most copies of $Y^{i,0}$ rooted at $v$ have precisely $i$ infected vertices. We will show by bounding other configurations (see Lemma~\ref{Xbound}) that the vast majority of choices of $U$ and open hyperedges rooted at vertices of $U$ satisfy (a) and (b). So, if $I(m)$ is a uniformly random infection of density $p + qt_m$, then using \eqref{expectedQv} we would expect that
\begin{align*}
Y_v^{i,j}(m) &\approx  Y_v^{i,0}(m) \binom{r-i-1}{j}\gamma(t_m)^j \\
&\approx \binom{r-1}{i}\binom{r-1-i}{j}(c + \alpha t_m)^i\gamma(t_m)^j d^{1 - \frac{i}{r-1}}.
\end{align*}

Before stating the bounds we wish to prove on the $Y$ configurations, it is helpful to introduce the following notation which will be used throughout the paper. Define
\begin{equation}
\label{yi0def}
y_{i,0}(t):=\binom{r-1}{i}(c+\alpha t)^i,
\end{equation}
and
\begin{equation}\label{yijdef}
y_{i,j}(t):=\binom{r-1-i}{j}y_{i,0}(t)\gamma(t)^j.
\end{equation}
We will prove the following.
\begin{lem}
\label{Yijrough}
With high probability the following statement holds. For all $0 \le m \le M$, for all $v \in V(\mathcal{H})$, for all $0\leq i\leq r-2$ and all $0 \le j \le r-1-i$,
$$Y_v^{i,j}(m) \in \left(1 \pm \epsilon(t_m)\right)y_{i,j}(t_m) d^{1 - \frac{i}{r-1}}.$$
\end{lem}

To track the $Y$ configurations and thus prove Lemma~\ref{Yijrough} we will bound the expectation of $Y_v^{i,j}(m+1) - Y_v^{i,j}(m)$ given $\mathcal{H}(m)$ and $I(m)$. Let us consider how a new copy of $Y^{i,j}$ rooted at $v$ is created. Let $Y^{i,j} = (\mathcal{F},R,D)$. For each new copy $Y$ of $Y^{i,j}$ rooted at $v$, there exists some $u \in D$ such that $Y$ is created from a copy of a configuration $X = (\mathcal{F},R\cup \{u\}, D \setminus \{u\})$, such that the image of $u$ in $\mathcal{H}(m)$ is the unique healthy vertex of a hyperedge that gets successfully sampled. So in fact it is created from a copy of configuration $X'= (\mathcal{F}',R, D')$, where $\mathcal{F}'$ is the union of $\mathcal{F}$ and one new hyperedge $e$ whose intersection with $\mathcal{F}$ contains some $u \in D$, and $D'$ is the union of $D\setminus \{u\}$ and $e\setminus \{u\}$. See Figure~\ref{Ymaking} for some examples of this.
\begin{figure} [htbp]
\centering
\includegraphics[width=0.9\textwidth]{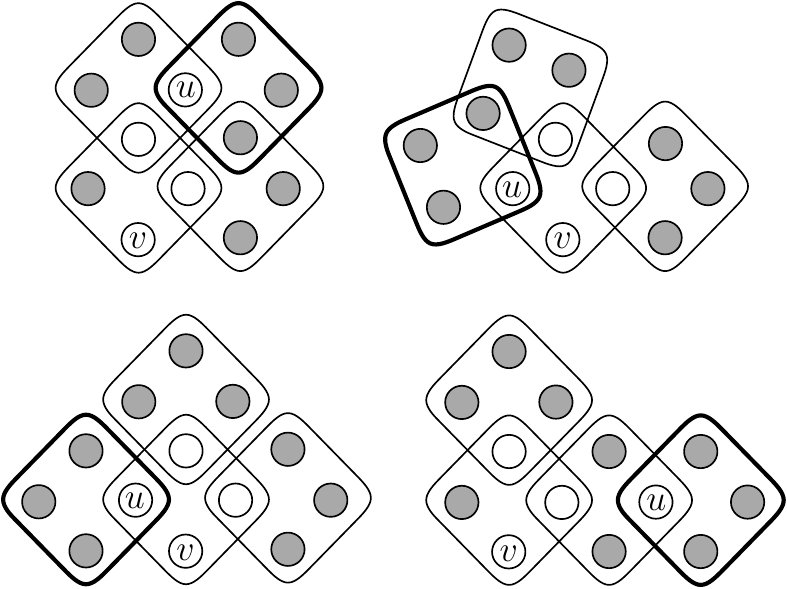}
\caption{Four examples of ways that a copy of $Y^{1,2}$ rooted at $v$ can be created when $r=4$. Such a copy is created if the open hyperedge containing $u$ (drawn with a thick outline) is successfully sampled.}\label{Ymaking}
\end{figure}

Either $e$ intersects $\mathcal{F}$ on precisely one vertex, or $e$ intersects $\mathcal{F}$ on several vertices, including vertices of $D$. In the first case, as we will see in Section~\ref{sec:diff}, $X'$ can be expressed as a combination of $Y$ configurations. However, the family of $Y$ configurations does not contain the type of intersections we get in the second case, so we need to control another family of variables which includes those with such intersections. In Section~\ref{yijsec} the calculation of the expectation of $Y_v^{i,j}(m+1) - Y_v^{i,j}(m)$ given $\mathcal{H}(m)$ and $I(m)$ will be presented in full detail; for now we have motivated the introduction of our next family, the \emph{secondary} configurations. These configurations are so called because they will not usually contribute to the main order term in our calculations, but still need to be controlled. See Figure~\ref{usefig} for some examples of these configurations.

\begin{figure}[htbp]
\centering
\includegraphics[width=0.8\textwidth]{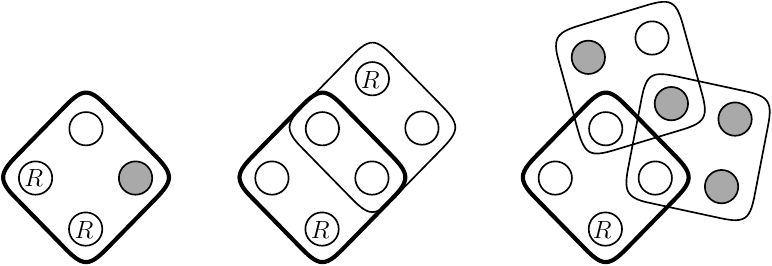}
\caption{Three examples of secondary configurations $X$ in the case $r=4$.  The marked vertices are shaded, the neutral vertices are white, the root is labelled with an $R$ and the central hyperedge is drawn with a thick outline.}
\label{usefig}
\end{figure}

The family of secondary configurations contains configurations with multiple roots and more complicated intersections of hyperedges than the $Y$ configurations. Due to the codegree conditions on $\mathcal{H}$ (see Definition~\ref{hypwellBDef}), we are able to prove stronger upper bounds on secondary configurations than on $Y$ configurations. Therefore we do not need to prove that the secondary configurations are concentrated. It will suffice to prove an upper bound that shows that they do not affect the main order term in our calculations for the $Y$ configurations.

This will ensure that any way of creating $Y$ configurations using secondary configurations (the second case above) is a lower order term than those terms given by the first case above. So the dominant behaviour of each $Y$ configuration is dictated purely by other $Y$ configurations. This is one reason why it is important for us to have codegree conditions on $\mathcal{H}$.

\begin{defn}
\label{usefulDef}
Say that a configuration $X=\left(\mathcal{F},R,D\right)$ is \emph{secondary} if $1\leq |E(\mathcal{F})|\leq3$ and $\mathcal{F}$ contains a hyperedge $e$ called the \emph{central} hyperedge such that 
\begin{enumerate}[(a)]
\item $e$ contains at least one root and at least one neutral vertex,
\item for every non-central hyperedge $e'$ we have that $e\cap e'$ contains at least one neutral vertex,
\item every vertex of $\mathcal{F}$ is contained in at most two hyperedges of $\mathcal{F}$ and
\item at least one of the following holds:
\begin{itemize}
\item $|R|\geq2$,
\item $|E(\mathcal{F})|=2$ and the two hyperedges of $\mathcal{F}$ intersect on more than one vertex, or
\item $|E(\mathcal{F})|=3$, each non-central hyperedge intersects $e$ on only one vertex and the two non-central hyperedges intersect one another. 
\end{itemize}
\end{enumerate}
\end{defn}

\begin{rem}
\label{rootTransfer}
If $X=(\mathcal{F},R,D)$ is a secondary configuration and $u\in D$, then both $(\mathcal{F},R,D\setminus\{u\})$ and $(\mathcal{F},R\cup \{u\},D\setminus\{u\})$ are secondary configurations.
\end{rem}

In order to determine how we expect $X(m)$ to behave, we require the following lemma giving an upper bound on $X_S(m)$ when $X$ is a secondary configuration with no marked vertices. This lemma will be proved in the next section and applied throughout the rest of the paper.

\begin{lem}\label{Xcount}
Let $X = (\mathcal{F},R,D)$ be a secondary configuration with $D = \emptyset$. Then, for any set $S\subseteq V(\mathcal{H})$ of cardinality $|R|$, the number of copies of $X$ in $\mathcal{H}(m)$ rooted at $S$ is
$$O\left(d^{\frac{|V(\mathcal{F})| - |R|}{r-1}}\log^{-K}(d)\right).$$
\end{lem}

We now consider what the expected number of copies of any secondary configuration $X=(\mathcal{F},R,D)$ in $\mathcal{H}(m)$ rooted at $S$ would be, if $I(m)$ were a uniformly random infection with density $p + qt_m$. The case $D = \emptyset$ is covered by Lemma~\ref{Xcount}, so now we consider the case that $D\neq\emptyset$. By Remark~\ref{rootTransfer}, the configuration $\tilde{X}:=(\mathcal{F},R,\emptyset)$ is also a secondary configuration and so, by Lemma~\ref{Xcount}, the expected number of copies of $\tilde{X}$ is $O\left(d^{\frac{|V(\mathcal{F})| - |R|}{r-1}}\log^{-K}(d)\right)$. For any such copy, if $I(m)$ were a uniformly random infection with density $p + qt_m$, then the probability that every vertex of the image of $D$ in this copy is infected would be precisely $(p + qt_m)^{|D|}$. Putting this together we get that, if $I(m)$ were uniform, then we would expect 
\[X_S(m) = O\left((p + qt_m)^{|D|} d^{\frac{|V(\mathcal{F})| - |R| }{r-1}}\log^{-K}(d)\right).\]

Now we formally describe the upper bound that we will prove on $X_S(m)$ for a general secondary configuration $X=(\mathcal{F},R,D)$ with central hyperedge $e$.

\begin{lem}
\label{Xbound}
With high probability the following statement holds. For all $0 \le m \le M$, for all secondary configurations $X = (\mathcal{F},R,D)$ and all $S \subseteq V(\mathcal{H})$ of cardinality $|R|$:
$$X_S(m) \leq \log^{2|D|r^4}(d) d^{\frac{|V(\mathcal{F})| - |R| - |D|}{r-1}}\log^{-3K/5}(d).$$
\end{lem}

By Definition~\ref{usefulDef}, every secondary configuration contains at least one neutral vertex. To bound the maximum change of our variables we will also need to have control over configurations that consist of a single hyperedge with no neutral vertices. So now we will define the final set of variables we wish to control.

 \begin{defn}
 \label{type1}
 For $1\leq i\leq r$, let $W^i$ denote the configuration $(\mathcal{F},R,D)$ where $\mathcal{F}$ is a hypergraph consisting of $r$ vertices contained in a single hyperedge, $|R|=i$ and $D=V(\mathcal{F})\setminus R$. 
 \end{defn}
 
 See Figure~\ref{wfig} for an illustration of some of these configurations in the case $r=6$. 
 \begin{figure}[htbp]
 \centering
 \includegraphics[width=0.4\textwidth]{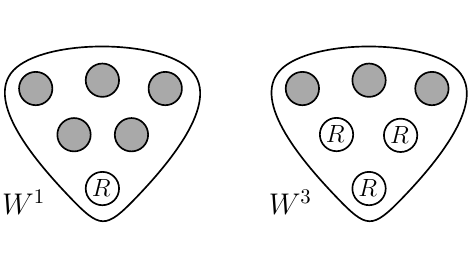}
 \caption{Some examples of $W$ configurations in the case $r = 6$. The marked vertices are shaded and the roots are labelled with an $R$.}
 \label{wfig}
 \end{figure}
 
 Of course, the variable $W_v^1(m)$ is the same as $Y^{r-1,0}_v(m)$, which is the same as $Q_v(m)$ if $v\notin I(m)$. For $i=1$ and $S=\{v\}$, the degree of $v$ in $\mathcal{H}$ is at most $d$ and so, if $I(m)$ were a uniformly random infection with density $p + qt_m$, then we would expect $W_S^1(m)$ to be at most $d\cdot \left(p + q t_m\right)^{r-1}$. For $i\geq2$, by condition~\ref{codegBound} of Definition~\ref{hypwellBDef}, the number of hyperedges of $\mathcal{H}(m)$ containing $S$ is at most $d^{1-\frac{i-1}{r-1}}\log^{-K}(d)$. So, if $I(m)$ were a uniformly random infection with density $p + qt_m$, then we would expect $W_S^i(m)$ to be at most
 \[\left(p +qt_m\right)^{r-i}\cdot d^{1-\frac{i-1}{r-1}}\log^{-K}(d) = \left(p + qt_m\right)^{r-i}d^{\frac{r-i}{r-1}}\log^{-K}(d)= (\alpha + ct_m)\log^{-K}(d)=o(1).\]
 Here, we use the fact $K$ is large and that we are only considering values of $t$ up to $O\left(\log(d)\right)$. 
 
 Thus we cannot hope to prove tight concentration bounds for these variables. However, we will prove the following.
 \begin{lem}
  \label{Wbound}
  With high probability the following statement holds. For all $1 \le i \le r$, for all $S \subseteq V(\mathcal{H})$ of cardinality $i$ and all $0 \le m \le M$, 
 $$W_S^i(m) \leq\log^{r^3(r-i)}(d).$$
 \end{lem}
 
The parts of the paper concerning the first phase are structured as follows. In Section~\ref{sec:prelims} we show (in Lemma~\ref{QfromOthers}) that Lemma~\ref{Qrough} can be deduced from Proposition~\ref{Ibound} and Lemmas~\ref{Yijrough} and~\ref{Wbound}. Then we prove the case $m=0$ of Lemmas~\ref{Yijrough},~\ref{Xbound} and~\ref{Wbound} in Section~\ref{sec:timeZero} using a version (Corollary~\ref{vunew}) of the Kim--Vu Inequality (Theorem~\ref{vu}). The case $1\leq m\leq M$ of Lemmas~\ref{Yijrough},~\ref{Xbound} and~\ref{Wbound} are proved in Section~\ref{sec:diff} using Freedman's Inequality (Theorem~\ref{Freed}) and the differential equations method. This concludes our discussion of the first phase of the proof of Theorem~\ref{hypmainThm}.

\subsection{The Second Phase}\label{secsum}
In the second phase of the proof of Theorem~\ref{hypmainThm} we define a different process which involves sampling a large set of open hyperedges in each round, rather than sampling one hyperedge at a time like we did in the first phase. As a slight abuse of notation, in the second phase we let $I(0)$ denote $I(M)$ and $\mathcal{H}(0)$ denote $\mathcal{H}(M)$; that is, after the first phase, we ``restart the clock'' from zero before running the second phase process. The second phase process will be defined slightly differently depending on whether we are in the subcritical or supercritical case. 

For $m \ge 1$, in round $m$ we will sample a set of open hyperedges and use the outcomes to define $I(m+1)$ and $\mathcal{H}(m+1)$. Again we will let $Q(m)$ be the set of hyperedges $e$ of $\mathcal{H}(m)$ such that $|e \setminus I(m)|= 1$ (the set of open hyperedges). Analogous to the first phase, for each configuration $X = (\mathcal{F},R,D)$, $S \subseteq V(\mathcal{H})$ with $|S|= |R|$ and integer $m\geq0$, we let $X_S(m)$ denote the set of copies of $X$ in $\mathcal{H}(m)$ rooted at $S$. As before, we still write $X_S(m)$ when referring to $\left|X_S(m)\right|$ and $Q(m)$ when referring to $|Q(m)|$.

We now define the number of steps for which we will run the processes in the second phase. For $\lambda < 1/8$ depending on only $r,c$ and $\alpha$ ($\lambda$ is defined in \eqref{T0q'}):
\begin{equation}
\label{M2def}
M_2:=\begin{cases} 2\left\lceil\log_{1/(1-\lambda)}(N)\right\rceil & \text{if }c^{r-2}\alpha < \frac{(r-2)^{r-2}}{(r-1)^{r-1}},\\
\min\left\{m:\left(\log N\right)^{\left(\frac{3}{2}\right)^m}>d^{\frac{1}{r-1} + \frac{1}{10}}\right\} & \text{if }c^{r-2}\alpha > \frac{(r-2)^{r-2}}{(r-1)^{r-1}}.\end{cases}
\end{equation}

\subsection{The Subcritical Case}\label{subcritrough}
First, let us consider the ``subcritical case''; i.e.~when $c^{r-2}\alpha<\frac{(r-2)^{r-2}}{(r-1)^{r-1}}$. Recall that, in this case, we track the first phase process until the number of open hyperedges is bounded above by $\zeta N$ for some constant $\zeta$ chosen small with respect to $r,c$ and $\alpha$.

\begin{process2}
 We obtain $I(m+1)$ and $\mathcal{H}(m+1)$ in the following way. For $m \ge 0$, sample every hyperedge in $Q(m)$. We let $I(m+1)$ be the union of $I(m)$ and all of the vertices in hyperedges which were successfully sampled and let $\mathcal{H}(m+1):=\mathcal{H}(m)\setminus Q(m)$.
\end{process2}

Our main result in the subcritical case is the following lemma, which immediately implies Theorem~\ref{hypmainThm} in the case $c^{r-2}\alpha<\frac{(r-2)^{r-2}}{(r-1)^{r-1}}$.

\begin{lem}\label{submain}
If $c^{r-2}\alpha<\frac{(r-2)^{r-2}}{(r-1)^{r-1}}$, then with high probability, 
\begin{enumerate}[(i)]
\item \label{QM2=0}$Q(M_2) = 0$, and
\item \label{IM2=o(N)}$|I(M_2)|= \frac{N\cdot\plog(N)}{d^{1/(r-1)}} = o(N)$.
\end{enumerate}
\end{lem}

The key ingredient in the proof Lemma~\ref{submain} is the following bound on $\mathbb{E}(Q(m))$:
\begin{equation}\label{Qsub}
\mathbb{E}(Q(m)) \le (1 - \lambda)^m 2 \zeta \cdot N,
\end{equation}
for some $\lambda < 1/8$ depending only on $r, c$ and $\alpha$ ($\lambda$ is defined in \eqref{T0q'}).
For $m=M_2$, \eqref{Qsub} implies that $\mathbb{E}(Q(M_2)) = o(1)$, and so \ref{QM2=0} follows from  Markov's Inequality (Theorem~\ref{Markov}). 

To deduce \eqref{Qsub}, consider how a member of $Q(m+1)$ is created. (As every open hyperedge is sampled at each time step, $Q(m+1) \cap Q(m) = \emptyset$.) A member of $Q(m+1)$ is created from a hyperedge $e \in \mathcal{H}(m)$ such that $0 \le |e \cap I(m)| \le r-2$, but $|e \cap I(m+1)|= r-1$. So in particular, in round $m$ each healthy vertex of $e$ but one is contained in an open hyperedge that is successfully sampled. Since $Q(m+1) \cap Q(m) = \emptyset$, $e$ contains at least one healthy vertex that is contained in a successfully sampled hyperedge of $Q(m)$.

So a member of $Q(m+1)$ is created when we have a hypergraph consisting of:
\begin{enumerate}
\item[(1)] a hyperedge $e \notin Q(m)$ such that $e \nsubseteq I(m)$,
\item[(2)] a hyperedge $e_u \in Q_u(m)$, for each healthy vertex $u$ in $e$ except one, 
\end{enumerate}
and each open hyperedge chosen in (2) is successfully sampled at time $m$.

Given such a hypergraph, picking a hyperedge $e' \not= e$ and deleting it gives a hypergraph that may be a copy of a $Y$ configuration (rooted at the healthy vertex of $e \cap e'$), or may not be (because of additional overlaps between the non-central hyperedges). In order to track $Q(m)$ in the second phase, we wish to track these sorts of configurations. The reason that we consider configurations of this type (where the edge $e'$ is not present rather than with it also included), is that we wish to express $Q(m+1)$ in terms of $Q(m)$ and breaking up the configuration this way allows us to do this (see for example \eqref{expQsub} below). This motivates the introduction of another family of configurations: the $Z$ configurations.

\begin{figure}[htbp]
\centering
\includegraphics[width=1\textwidth]{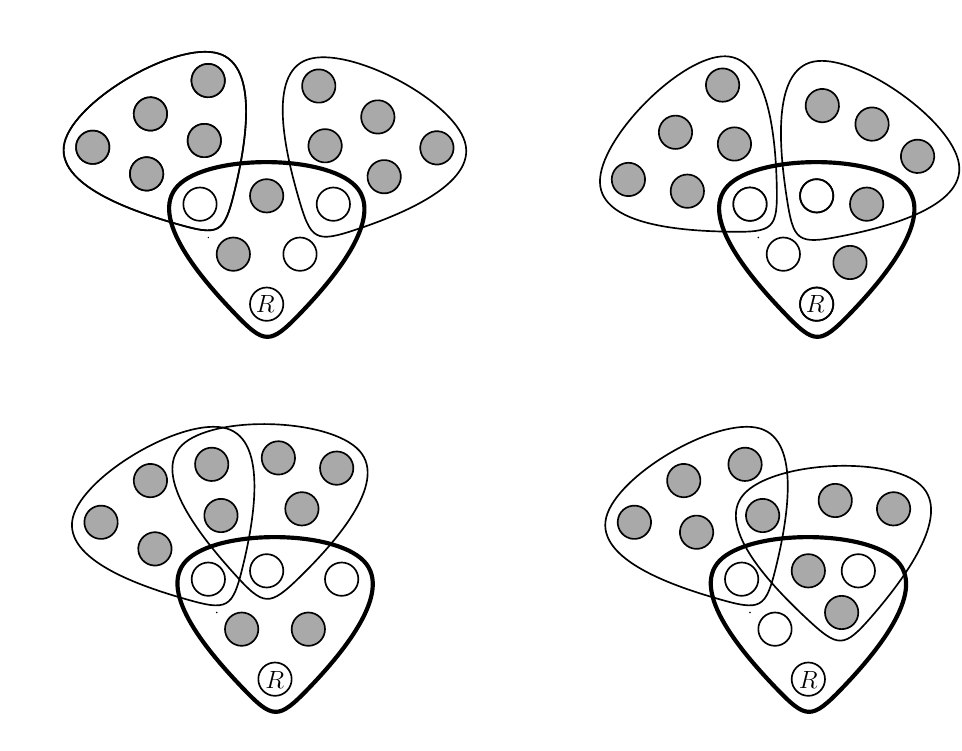}
\caption{Some examples of possible members of $Z_v^{2,2}(m)$ in the case $r=6$. The root is labelled by $R$, the infected vertices are shaded, the healthy vertices are white and the central hyperedge is drawn with a thick outline.}
\label{Zconfig}
\end{figure}

\begin{defn}\label{Zij}
Given $v\in V(\mathcal{H})$, $0\leq i\leq r-2$, $0\leq j\leq r-1-i$ and $0\leq m\leq M_2$ let $Z_v^{i,j}(m)$ be the union of $Z_v(m)$ over all configurations $Z=(\mathcal{F},R,D)$ such that $\mathcal{F}$ is a hypergraph containing a hyperedge $e$, called the \emph{central} hyperedge, such that 
\begin{enumerate}[(a)]
\item $e$ contains exactly $i$ marked vertices, 
\item there is a unique root and $e$ is the only hyperedge of $\mathcal{F}$ containing the root,
\item $\mathcal{F}$ has exactly $j$ non-central hyperedges, 
\item every non-central hyperedge contains exactly $r-1$ marked vertices, 
\item for each non-central hyperedge $e'$, the unique neutral vertex of $e'$ is contained in $e$, and
\item no neutral vertex is contained in two non-central hyperedges. 
\end{enumerate}
\end{defn}
See Figure~\ref{Zconfig} for some examples in the case $r=6$. These configurations can be thought of as a more general version of the $Y$ configurations, in fact the main order term of $Z_v^{i,j}(m)$ comes from $Y_v^{i,j}(m)$. The following simple observation provides a helpful way of thinking about the $Z$ configurations.

\begin{obs}\label{Zsecondary}
Let $Z$ be a member of $Z_v^{i,j}(m)$. If each of the non-central hyperedges intersects the central hyperedge $e$ on only one vertex and no pair of them intersect one another then $Z$ is a member of $Y_v^{i,j}(m)$. Also observe that $Y_v^{i,j}(m) \subseteq Z_v^{i,j}(m)$, as every member of $Y_v^{i,j}(m)$ is a member of $Z_v^{i,j}(m)$ of this type. If one of the non-central hyperedges intersects $e$ on more than one vertex or two of them intersect one another, then $Z$ consists of a copy $\mathcal{F}'$ of a secondary configuration with one root, $r-1-i$ neutral vertices and a set of at most $j-1$ copies of $W^1$ rooted at vertices of $\mathcal{F}'$. 
\end{obs}

We will see in Lemma~\ref{subtrack} that the members of $Z_v^{i,j}(m)$ that come from secondary configurations just contribute a lower order term. Given the above discussion, conditioned on $\mathcal{H}(m)$ and $I(m)$ the expected value of $Q(m+1)$ is at most
\begin{equation}\label{expQsub}
\sum_{Q \in Q(m)} \sum_{w \in Q \setminus I(m)} \sum_{j = 0}^{r-2}Z_w^{r-2-j,j}(m) q^{j+1}.
\end{equation}
The proof of Lemma~\ref{submain} comes down to proving that $Z_w^{r-2-j,j}(m)$ satisfies strong enough upper bounds (with high probability) so that evaluating this expression gives the bound in \eqref{Qsub}.

So to prove Lemma~\ref{submain}, it suffices to control $Z_v^{i,j}(m)$. To do this, we must also keep control over $Y_v^{i,j}(m)$, $X_S(m)$ and $W_S^i(m)$, as before. This is achieved via multiple applications of a version of the Kim--Vu Inequality (Corollary~\ref{vunew}). Full details will be given in Section~\ref{sec:phaseTwoSub}.

\subsection{The Supercritical Case}\label{subsup}
The strategy in the ``supercritical case'', i.e.~when $c^{r-2}\alpha>\frac{(r-2)^{r-2}}{(r-1)^{r-1}}$, is somewhat similar, but the details are a little different. 

\begin{process3}
 We obtain $I(m+1)$ and $\mathcal{H}(m+1)$ in the following way. Each round contains two steps. In the first step we choose some $Q'(m) \subseteq Q(m)$ and sample every open hyperedge of $Q'(m)$. We define $Q'(0)=Q(0)$. For $m>0$ we choose $Q_v'(m)$ to be a subset of $Q_v(m)$ with cardinality $\min\left\{Q_v(m),(\log N)^{\left( \frac{3}{2}\right)^{m}} \right\}$ and define 
 $$Q'(m):= \bigcup_{v \in V(\mathcal{H})\setminus I(m)} Q_v'(m).$$ We let $I_0(m+1)$ be the union of $I(m)$ and all of the hyperedges of $Q'(m)$ that were successfully sampled. We let $\mathcal{H}_0(m+1):= \mathcal{H}(m) \setminus Q'(m)$. 
 
 For $i \ge 0$ and $v \in V(\mathcal{H}) \setminus \mathcal{I}_{i}(m+1)$, define $Q_v^i(m+1)$ to be the set of open hyperedges containing $v$ in $\mathcal{H}_i(m+1)$ (for $i \ge 1$, $\mathcal{I}_{i}(m+1)$ and $\mathcal{H}_i(m+1)$ will be defined shortly). Say that $Q_v^i(m+1)$ is \emph{large} if it has cardinality at least $d^{\frac{1}{r-1} + \frac{1}{10}}$.
 
  In the second step, for $i \ge 0$, let $L_i(m+1)$ be the set of vertices $v$ such that $Q_v^i(m+1)$ is large. We sample every open hyperedge contained in $Q''_i(m+1):=\cup_{v \in L_i(m+1)}Q_v^i(m+1)$. Define $\mathcal{H}_{i+1}(m+1):= \mathcal{H}_{i}(m+1) \setminus Q''_i(m+1)$ and let $I_{i+1}(m+1)$ be the union of $I_i(m)$ and the vertices of all of the hyperedges of $Q''_i(m+1)$ that were successfully sampled. The second step ends when we reach $j$ such that there is no healthy vertex $v \in \mathcal{H}_j(m+1)$ such that $Q_v^j(m+1)$ is large. We define $\mathcal{H}(m+1):= \mathcal{H}_j(m+1)$, and $I(m+1):= I_j(m+1)$.
\end{process3}

We will see that when $Q_v^i(m)$ is large, with high probability, $v$ will become infected when we sample every hyperedge in $Q_v^i(m)$.

The proof of Theorem~\ref{hypmainThm} relies on the following bound:
\begin{equation}
\label{rough2}
Q_v(m) \ge (\log N)^{\left(\frac{3}{2}\right)^m},
\end{equation}
for each $0 \le m \le M_2$ and all $ v \in V(\mathcal{H}) \setminus I(m)$. To prove this bound, we use a version of Janson's Inequality for the lower tail (Theorem~\ref{JansonLower}). Given \eqref{rough2}, after at most $O\left(\log\log N\right)$ rounds, every healthy vertex has at least $d^{1/(r-1) + 1/10}$ open hyperedges containing it. Now, by the Chernoff Bound (Theorem~\ref{Chernoff}) and the fact that $N=d^{O(1)}$ (by property~\ref{almostRegular} of Definition~\ref{hypwellBDef}),  with high probability every vertex is infected after only one additional round. Therefore, percolation occurs with high probability. See Section~\ref{sec:super} for full details. This concludes our outline of the proof. 

\section{Probabilistic Tools}\label{sec:prob}
Here, for convenience, we collect together the probabilistic tools we apply throughout the paper. We will also formally define the probability space that we are working in.

\subsection{Standard Concentration Inequalities}

The following two theorems will be repeatedly applied in the rest of the paper. The first is Markov's Inequality, which is perhaps the simplest  concentration inequality in probability theory.  

\begin{thm}[Markov's Inequality]
\label{Markov}
If $X$ is a non-negative random variable and $a>0$, then 
\[\mathbb{P}(X\geq a)\leq \frac{\mathbb{E}(X)}{a}.\]
\end{thm}

The second is a version of the Chernoff Bound, which can be found in~\cite[Theorem~1.1]{ConcentrationBook}. 

\begin{thm}[The Chernoff Bound]
\label{Chernoff}
Let $X_1,\dots,X_n$ be a sequence of independent $[0,1]$-valued random variables and let $X=\sum_{i=1}^nX_i$. Then, for $0<\varepsilon<1$, 
\[\mathbb{P}\left(X<(1-\varepsilon)\mathbb{E}(X)\right)\leq e^{-\frac{\varepsilon^2\mathbb{E}(X)}{2}},\]
\[\mathbb{P}\left(X>(1+\varepsilon)\mathbb{E}(X)\right)\leq e^{-\frac{\varepsilon^2\mathbb{E}(X)}{3}}.\]
Moreover, if $t>2e\mathbb{E}(X)$, then
\[\mathbb{P}(X>t) < 2^{-t}.\]
\end{thm}

\subsection{Kim--Vu and Janson Inequalities}

A central theme in the study of large deviation inequalities is that if a random variable $X$ depends on a sequence of independent trials in which, for any outcome of the trials, changing the result of a small set of the trials does not influence the value of $X$ too much, then $X$ is often concentrated (see, e.g.,~\cite{boundedDifferences,Lutz,ConcentrationBook} for further discussion). In our case, it is clear that the value of any of the variables that we track at time zero depends on the $N$ independent random trials which determine whether or not each vertex of $V(\mathcal{H})$ is contained in $I(0)$. However, as it turns out, our codegree and neighbourhood similarity conditions (conditions~\ref{codegBound} and~\ref{neighSim} of Definition~\ref{hypwellBDef}) are not strong enough to obtain good control over the worst case influence of changing a small set of trials. Fortunately for us, there exist a number of different tools for obtaining strong concentration when the worst case influence is rather large, but the typical influence is small. 

The tool that we use in Section~\ref{sec:timeZero} to prove bounds on our variables when $m=0$ is a version of the Kim--Vu Inequality~\cite{KimVu} due to Vu~\cite{Vu} which is particularly well suited to our situation. Other such tools include large deviation versions of Janson's Inequality (Theorem~\ref{JansonLower}), which we apply in  Section~\ref{sec:super}, and the ``method of typical bounded differences'' developed by Warnke~\cite{Lutz}.  Before stating the Kim--Vu inequality, we require some definitions.

\begin{defn}
Let $V$ be a finite index set and let $f$ be a multivariate polynomial in variables $\{x_v: v\in V\}$. Given a multiset $A$ of indices from $V$, let $\partial_A f$ denote the partial derivative of $f$ with respect to the variables $\{x_v:v\in A\}$ (with multiplicity).
\end{defn}

\begin{defn}
\label{defnEj}
Let $X$ be a random variable of the form $f\left(\xi_v: v\in V\right)$ where $V$ is an index set with $N$ elements, $f$ is a multivariate polynomial of degree $k$ in variables $\{x_v: v\in V\}$ with coefficients in $[0,1]$ and $\{\xi_v: v\in V\}$ is a collection of independent random variables taking values in $\{0,1\}$. For $0\leq j\leq k$, define
\[\mathbb{E}_j(X):=\max_{|A|\geq j}\mathbb{E}\left(\partial_A f\left(\xi_v: v\in V\right)\right)\]
where the maximum is taken over all multisets of indices from $V$ with cardinality at least $j$.
\end{defn}

\begin{thm}[Vu~\cite{Vu}]
\label{vu}
There exist positive constants $c_k$ and $C_k$ such that if $X$ is a random variable as in Definition~\ref{defnEj} and $\mathscr{E}_0 > \mathscr{E}_1 > \cdots > \mathscr{E}_k=1$ and $\ell$ are positive numbers such that $\mathbb{E}_j(X)\leq \mathscr{E}_j$ for $0 \le j \le k$, and $\mathscr{E}_j / \mathscr{E}_{j+1} \ge \ell + j \log N$ for $0 \le j \le k-1$, then
$$\mathbb{P}\left(\left|X - \mathbb{E}(X)\right| \ge \sqrt{\ell \mathscr{E}_0 \mathscr{E}_1}\right) \le C_k e^{-c_k \ell}.$$
\end{thm}

In order to apply Theorem~\ref{vu} to some random variable of the form $f\left(\xi_v: v\in V\right)$, we require that the coefficients of $f$ are in $[0,1]$. Therefore, in practice, to apply this theorem to most of our variables we first need to rescale them, then apply the theorem, then scale them back to get the bounds we want on the original variable. For this reason, we prove a corollary to Theorem~\ref{vu} which applies this theorem in precisely the form we will use it throughout the paper. This should make the later calculations easier to follow.

\begin{cor}\label{vunew}
Let $X =  f\left(\xi_v: v\in V\right)$ be a random variable where $V$ is an index set with $N$ elements, $f$ is a multivariate polynomial of degree $k$ in variables $\{x_v: v\in V\}$ with non-negative coefficients and no variable in $f$ has an exponent greater than 1. Let $\tau$ and $\mathcal{E}_0 \ge \log^{2k +1} N$ be positive numbers such that
\begin{enumerate}[(i)]
\item\label{Ezero} $ \mathbb{E}(X) \le  \tau \cdot \mathcal{E}_0$, and
\item\label{Ej}  $ \mathbb{E}_j(X) \le  \tau$, for $1 \le j \le k$.
\end{enumerate}
Then for $N$ sufficiently large, with probability at least $1 - N^{-20\sqrt{\log N}}$,
$$X \in \mathbb{E}(X) \pm O\left(\tau\log^{k}(N) \sqrt{\mathcal{E}_0}\right).$$
\end{cor}

\begin{proof}
Define
$$Z = g\left(\xi_v: v\in V\right) := \tau^{-1} \cdot  f\left(\xi_v: v\in V\right).$$
We show that $Z$ is close to its expectation with high probability and use this to obtain bounds on $X$ which hold with high probability. Using the definition of $Z$, from \ref{Ezero} we obtain that
\begin{equation}\label{vu1}
\mathbb{E}(Z) = \tau^{-1} \cdot \mathbb{E}(X) \le \mathcal{E}_0
\end{equation}
and from \ref{Ej} we obtain that for $1 \le j \le k$,
\begin{equation}\label{vu2}
\mathbb{E}_j(Z) =  \tau^{-1} \cdot \mathbb{E}_j(X) \le 1.
\end{equation}

 In particular this implies that, when $N$ is sufficiently large, every term of $g$ has a coefficient which is at most one. Indeed, for a monomial $\prod_{u \in A}x_u$ appearing with a non-zero coefficient in $g$, its coefficient is precisely $\mathbb{E}\left(\partial_A g\right)\leq \mathbb{E}_{|A|}\left(Z\right)$.
 
 Set $\ell:= \log^2 (N)$, $\mathcal{E}_k := 1$ and $\mathcal{E}_j := (\ell + j \log N)\mathcal{E}_{j+1}$ for $1 \le j < k$. By hypothesis, $\mathcal{E}_0 \ge \log^{2k +1} (N)$, which is at least $(\ell + j \log N)\mathcal{E}_1$ for $N$ sufficiently large, and so $\mathcal{E}_j/\mathcal{E}_{j+1} \ge \ell + j \log N$ for $0 \le j \le k-1$.
 
 By \eqref{vu1} and \eqref{vu2}, we have $\mathbb{E}_j(Z) \le \mathcal{E}_j$ for all $0 \le j \le k$. Therefore, we may apply Theorem~\ref{vu} to obtain that
 $$\mathbb{P}\left(|Z - \mathbb{E}(Z)| \ge \sqrt{\ell \mathcal{E}_0 \mathcal{E}_1}\right) \le e^{-\Omega(\log^2 N)} \le N^{-20\sqrt{\log N}}.$$

Since by definition $\ell \cdot \mathcal{E}_1 = O\left(\log^{2k}(N)\right)$, we have $\sqrt{\ell \mathcal{E}_0 \mathcal{E}_1} = O\left(\log ^k (N) \sqrt{\mathcal{E}_0}\right)$.
Therefore with probability at least $1 - N^{-20\sqrt{\log N}}$,
$$Z \in \mathbb{E}(Z) \pm O\left(\log ^k (N) \sqrt{\mathcal{E}_0}\right).$$
Now rescaling by $\tau$ gives that with probability at least $1 - N^{-20\sqrt{\log N}}$,
$$X \in \mathbb{E}(X) \pm O\left(\tau\log ^k (N) \sqrt{\mathcal{E}_0}\right),$$
as required.
\end{proof}

We will also use Corollary~\ref{vunew} to prove bounds on our variables for the ``subcritical'' case during the second phase. For the ``supercritical'' case, we find it more convenient to apply the following lower tail version of Janson's Inequality. 

\begin{thm}[Janson's Inequality for the Lower Tail~\cite{Janson2}] 
\label{JansonLower}
Let $\mathcal{G}$ be a hypergraph and, for $p\in (0,1)$ and $e\in E(\mathcal{G})$, let $I_e$ be the indicator variable for the event $e\subseteq V(\mathcal{G})_p$. Set
\[X:=\sum_{e\in E(\mathcal{G})} I_e,\]
\[\mu:= \mathbb{E}(X), \text{ and}\]
\[\delta:=\sum_{\substack{(e,e')\in E(\mathcal{G})^2 \\e \not= e' \\ e\cap e'\neq\emptyset}}\mathbb{E}\left(I_eI_{e'}\right),\]
where the final sum is over all ordered pairs, (so each pair is counted twice). 
Then, for any $\varepsilon\in [0,1]$,
\[\mathbb{P}\left(X\leq (1-\varepsilon)\mathbb{E}(X)\right)\leq \exp\left(-\varphi(-\varepsilon)\mu^2/(\mu+\delta)\right),\]
where $\varphi(x) = (1+x)\log(1+x) - x$. 
\end{thm}

\subsection{Martingales and Concentration}
In Section~\ref{sec:diff} we will use standard martingale concentration inequalities to prove bounds on our variables throughout the first phase for $m>0$. Recall that a sequence $0=B(0),B(1),\dots$ of random variables is said to be a \emph{supermartingale} with respect to a filtration $\mathcal{F}(0), \mathcal{F}(1),\dots$ if, for all $m\geq0$, we have that $B(m)$ is $\mathcal{F}(m)$-measurable and
\[\mathbb{E}\left(B(m+1)\mid \mathcal{F}(m)\right)\leq B(m).\]
In what follows, when we say that $B(0),B(1),\dots$ is a supermartingale, it is always with respect to the natural filtration corresponding to our process (which will be formally defined in the next subsection). A sequence $B(0),B(1),\dots$ is a \emph{submartingale} if the sequence $-B(0),-B(1),\dots$ is a supermartingale. Also, a sequence $B(0),B(1),\dots$ is said to be \emph{$\eta$-bounded} if, for all $m\geq0$,
\[-\eta\leq B(m+1) - B(m)\leq \eta.\]
Our main tool in Section~\ref{sec:diff} is the following concentration inequality of Freedman~\cite{Freedman}.

\begin{thm}[Freedman~\cite{Freedman}]
\label{Freed}
Let $0:=B(0),B(1),\dots$ be a $\eta$-bounded supermartingale and let 
\[V(m):=\sum_{\ell=0}^{m-1} \Var(B(\ell+1)-B(\ell)\mid \mathcal{F}_\ell).\]
Then, for all $a,\nu>0$, 
\[\mathbb{P}(B(m)\geq a\text{ and } V(m)\leq \nu\text{ for some $m$})\leq \exp\left(-\frac{a^2}{2(\nu+a\eta)}\right).\]
\end{thm}

\subsection{The Probability Space}\label{sec:probspace}

A natural candidate for the probability space on which to view our process is 
\[\Omega:=  \{0,1\}^N \times E(\mathcal{H})^{|E(\mathcal{H})|}\times \mathcal{P}(E(\mathcal{H}))^{|E(\mathcal{H})|}\times \{0,1\}^{|E(\mathcal{H})|}\]
where $\mathcal{P}(E(\mathcal{H}))$ is the collection of all subsets of $E(\mathcal{H})$. For any point in $\Omega$, the first $N$ coordinates determine the infection at time zero, the next $|E(\mathcal{H})|$ coordinates  list the hyperedges sampled during the first phase process (although, note that the first phase stops before $|E(\mathcal{H})|$ hyperedges have been sampled), the next $|E(\mathcal{H})|$ coordinates list the sets of hyperedges sampled during the second phase process and the last $|E(\mathcal{H})|$ coordinates determine which hyperedges of $\mathcal{H}$ are contained in $\mathcal{H}_q$. 

One should notice that $\Omega$ contains a large number of infeasible points (i.e.~points of measure zero); for example, it contains points corresponding to evolutions of the processes in which some hyperedges are sampled more than once, or the $m$th hyperedge sampled in the first phase is not even chosen from $Q(m-1)$, etc. We let $\Omega'$ be the subspace of $\Omega$ consisting of only those points which have positive measure.

For $m\geq0$, let $\mathcal{F}_m$ be the $\sigma$-algebra generated by the partitioning of $\Omega'$ in which two points are in the same class if they  correspond to evolutions of the processes which have the same initial infection and which are indistinguishable after $\ell$ steps of the first phase process for every $\ell$ in the range $1\leq \ell\leq m$. For example, any two points of $\Omega'$ corresponding to evolutions in which the first phase process runs for fewer than $m$ steps are in the same class if and only if they are indistinguishable at every step of the first phase. Similarly, for $m\geq0$, let $\mathcal{F}_m'$ be the $\sigma$-algebra generated by the partitioning of $\Omega'$ in which two points are in the same class if they are indistinguishable at every step of the first phase and, for every $1\leq \ell\leq m$, they are indistinguishable after the $\ell$th step of the second phase process. We will work in this probability space throughout the proof without further comment.

\section{Preliminaries}\label{sec:prelims}

In this section, we will prove four preliminary results. First we prove Proposition~\ref{Ibound}, which gives a bound on the number of infected vertices at each step of the first phase. Then we deduce that, in order to track $Q(m)$, it is enough to have control over the $Y$ and $W$ configurations and the number of infected vertices at time $m$. After that, we will prove Lemma~\ref{Xcount}, which bounds the number of copies of any secondary configuration $(\mathcal{F},R,\emptyset)$ in $\mathcal{H}(m)$ rooted at $S$ (where $|S|= |R|$), for any $0 \le m \le M$. At the end of the section, we will prove an analytic lemma that will be used in the application of the differential equations method in Section~\ref{sec:diff}.

We restate here Proposition~\ref{Ibound} from Section 2, to aid the reader.

\begin{Iprop}[Restated]
  For $0 \le m \le M$, with probability at least $1 - N^{-\Omega\left(\sqrt{\log N}\right)}$,
  $$I(m) = O\left(\log N \cdot Nd^{-1/(r-1)}\right).$$
\end{Iprop}
\begin{proof}
The expected number of vertices which are infected at time zero is $cNd^{-1/(r-1)}$. By the Chernoff bound (Theorem~\ref{Chernoff}) with $\epsilon = 1/2$, we have that, with probability at least $1-e^{-\Omega\left(N d^{-1/(r-1)}\right)}$, there are at most $\frac{3c}{2}Nd^{-1/(r-1)}$ vertices infected at time zero.

At each step, one hyperedge is sampled and becomes infected with probability $q$. As the total number of hyperedges sampled is at most $M = O\left(\log N \cdot N\right)$, the expected number of vertices infected by successfully sampling an open hyperedge is $O\left(\frac{N\cdot\log(N)}{d^{1/(r-1)}}\right)$.

Applying the Chernoff bound with $\epsilon = 1/2$, we get that with probability at least $1-N^{-\Omega\left(\sqrt{\log N}\right)}$ there are at most $O\left(\frac{N\cdot\log(N)}{d^{1/(r-1)}}\right)$ vertices infected by successfully sampling an open hyperedge.  By choosing $K$ large enough and \eqref{Nnotsmall}, we have $1-e^{-\Omega\left(N d^{-1/(r-1)}\right)} \geq 1-N^{-10\sqrt{\log N}}$. So with probability at least $1 - N^{-10\sqrt{\log N}}$, 
$$I(M) = \frac{3c}{2}Nd^{-1/(r-1)} + O\left(\frac{N\cdot\log(N)}{d^{1/(r-1)}}\right)= O\left(\frac{N\cdot\log(N)}{d^{1/(r-1)}}\right).$$ 
As  $I(\ell) \le I(M)$ for any $\ell \le M$, this completes the proof.
\end{proof}

As mentioned in Section~\ref{sec:outline}, to prove Lemma~\ref{Qrough}, it is sufficient to prove Lemmas~\ref{Yijrough} and~\ref{Wbound} and Proposition~\ref{Ibound}. More formally:

\begin{lem}
\label{QfromOthers}
If for every choice of $0 \le m \le M$, $v \in V(\mathcal{H})$, $0\leq i\leq r-2$ and $0 \le j \le r-1$ we have:
\begin{enumerate}[(i)]
\item\label{QYbound} $Y_v^{i,j}(m) \in \left(1 \pm \epsilon(t_m)\right)y_{i,j}(t_m)d^{1 - \frac{i}{r-1}},$
\item\label{QWbound} $W_v^{1}(m) \le \log^{r^4}(d)$, \text{ and}
\item\label{QIbound} $I(m)= O\left(\log N \cdot Nd^{-1/(r-1)}\right)$,
\end{enumerate}
then for every $0 \le m \le M$,
$$Q(m) \in (1 \pm 4\epsilon(t_m))\gamma(t_m) \cdot N.$$
\end{lem}

\begin{proof}
The sum 
\begin{equation}\label{Sdef}
S:= \sum_{w\in V(\mathcal{H})\setminus I(m)}(r-1)Q_w(m)\left(Y_w^{0,0}(m)-Y_w^{1,0}(m)\right)
\end{equation} 
counts the number of ways of choosing
\begin{enumerate}[(1)]
\item a vertex $w\in V(\mathcal{H})\setminus I(m)$,  
\item a hyperedge $e'\in Q_w(m)$, 
\item a hyperedge $e \in \mathcal{H}(m)$ disjoint from $I(m)$ containing $w$, and 
\item a vertex $v\in e\setminus\{w\}$. 
\end{enumerate}

Viewing $v$ as the root, such a configuration is a copy of $Y^{0,1}$ rooted at $v$ but not a copy of $Y^{1,1}$. Define $S_2:= \sum_{v\in V(\mathcal{H})\setminus I(m)}\left(Y_v^{0,1}(m) -Y_v^{1,1}(m)\right)$. Then $S_2$ counts the number of ways of choosing 
\begin{enumerate}[(1)]
\item a vertex $v\in V(\mathcal{H})\setminus I(m)$,  
\item a hyperedge $e \in \mathcal{H}(m)$ containing $v$ and at most one infected vertex,
\item a vertex $w\in e\setminus \{v\}$ such that if $e \cap I(m) \not= \emptyset$ then $w \in I(m)$, and 
\item a hyperedge $e'\in W^1_w(m)$,
\end{enumerate}
which contains everything that $S$ counts. So $S \le S_2$. The configurations counted by $S_2$ but not $S$ are those given by choosing
\begin{enumerate}[(1)]
\item a vertex $v\in V(\mathcal{H})\setminus I(m)$,  
\item a hyperedge $e \in \mathcal{H}(m)$ containing $v$ such that $e$ has a unique infected vertex $w$, 
\item a hyperedge $e' \in W^{1}_w(m)$.
\end{enumerate}
See Figure~\ref{S1v2} for an illustration of the difference between what is counted in $S$ and $S_2$ when $r=6$.

\begin{figure}[htbp]
\centering
\includegraphics[width=0.3\textwidth]{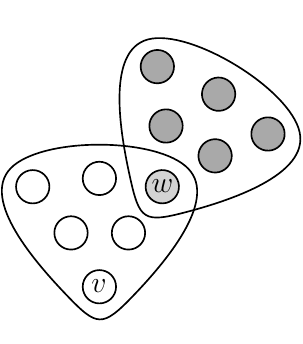}
\caption{Here $r=6$. Copies of this hypergraph are counted by $S$ only when $w$ is healthy. Copies are counted by $S_2$ whether $w$ is infected or not. The (unlabelled) vertices shaded dark grey are infected, the unshaded vertices are healthy.}
\label{S1v2}
\end{figure}

Applying hypotheses \ref{QYbound} and \ref{QWbound} to bound $Y_v^{1,0}(m)$ and $W_w^1(m)$ for all $v\in V(\mathcal{H})\setminus I(m)$ and $w\in I(m)$, we see that $S_2-S$ is bounded above by
\[\sum_{v\in V(\mathcal{H})\setminus I(m)} \left(Y_v^{1,0}(m)-Y_v^{2,0}(m)\right)\log^{r^4}(d) =N\cdot d^{\frac{r-2}{r-1}}\log^{O(1)}(d).\] 
Thus,
\[S\in S_2 \pm N\cdot d^{1-\frac{1}{r-1}}\log^{O(1)}(d).\]

Using the definition of $S_2$ and applying the bounds on $Y_v^{1,1}(m)$ given by the hypotheses of the lemma, we have
\begin{equation}
\label{Sput2}
S \in \sum_{v\in V(\mathcal{H})\setminus I(m)}Y_v^{0,1}(m) \pm N\cdot d^{1-\frac{1}{r-1}}\log^{O(1)}(d).
\end{equation}

Using the bounds on $Y_v^{1,0}(m)$ and (for $w \in V(\mathcal{H})\setminus I(m)$) the bounds on $Q_w(m) = W^1_w(m)$ given by the hypotheses of the lemma, from \eqref{Sdef} we get
\begin{equation}
\label{Sput1}
S \in \sum_{w\in V(\mathcal{H})\setminus I(m)}(r-1)Q_w(m)Y_w^{0,0}(m) \pm N\cdot d^{1-\frac{1}{r-1}}\log^{O(1)}(d).
\end{equation}

Combining \eqref{Sput1} and \eqref{Sput2} gives
$$\sum_{w\in V(\mathcal{H})\setminus I(m)}(r-1)Q_w(m)Y_w^{0,0}(m) \in \sum_{v\in V(\mathcal{H})\setminus I(m)}Y_v^{0,1}(m) \pm N\cdot d^{1-\frac{1}{r-1}}\log^{O(1)}(d).$$

We also have by hypothesis that
\[\sum_{w\in V(\mathcal{H})\setminus I(m)}(r-1)Q_w(m)Y_w^{0,0}(m) \in \sum_{w\in V(\mathcal{H})\setminus I(m)}(r-1)Q_w(m)\left(1\pm \epsilon(t_m)\right)d\]
\[\subseteq Q(m)(r-1)\left(1\pm \epsilon(t_m)\right)d\]
and
\[\sum_{v\in V(\mathcal{H})\setminus I(m)}Y_v^{0,1}(m) \in  |V(\mathcal{H})\setminus I(m)|\left(1\pm \epsilon(t_m)\right)(r-1)\gamma(t_m)\cdot d.\]
Putting all this together gives  
\[Q(m)\in \frac{|V(\mathcal{H})\setminus I(m)|\left(1 \pm \epsilon(t_m)\right)(r-1)\gamma(t_m)\pm  d^{-\frac{1}{r-1}}N\plog{d}}{(r-1)\left(1\mp \epsilon(t_m)\right)}\]
\[\subseteq N(1 \pm 4\epsilon(t_m))\gamma(t_m)\]
since $\frac{1\pm x}{1 \pm x}\in 1 \pm 3x$ for $x$ sufficiently small and $I(m) = O\left(\log N \cdot Nd^{-1/(r-1)}\right)$ by hypothesis \ref{QIbound}. The result follows. 
\end{proof}

We will now present the proof of Lemma~\ref{Xcount}. It may be helpful to first recall the definition of a secondary configuration from Definition~\ref{usefulDef}. We restate here the result from Section 2  to aid the reader.

\begin{usefullem}[Restated]
Let $X = (\mathcal{F},R,D)$ be a secondary configuration with $D = \emptyset$. Then, for any set $S\subseteq V(\mathcal{H})$ of cardinality $|R|$, the number of copies of $X$ in $\mathcal{H}(m)$ rooted at $S$ is
$$O\left(d^{\frac{|V(\mathcal{F})| - |R|}{r-1}}\log^{-K}(d)\right).$$
\end{usefullem}

\begin{proof}
Let $X=(\mathcal{F},R,D)$ be a secondary configuration. First we see that, for any ordering $e_1,\dots,e_{|E(\mathcal{F})|}$ of the hyperedges of $\mathcal{F}$, we have that 
\begin{equation}\label{V-R}
|V(\mathcal{F})|= |R| + \sum_{i=1}^{|E(\mathcal{F})|} \left(r - \left|e_i\cap \left(R\cup \bigcup_{j=1}^{i-1} e_j\right)\right|\right).\end{equation}
To see this, count the number of vertices by first counting the vertices of $R$ and then, for each $i$ in turn, count the vertices of $e_i$ which have not yet been counted. This will be used several times in the calculations below. 

Our goal is to bound the number of copies of $X$ in $\mathcal{H}(m)$ rooted at a set $S\subseteq V(\mathcal{H})$ of cardinality $|R|$. By construction, $\mathcal{H}(m)$ is a subhypergraph of $\mathcal{H}(0)=\mathcal{H}$. Therefore, it suffices to upper bound the number of copies of $X$ in $\mathcal{H}(0)=\mathcal{H}$ rooted at $S$. We do this in the way we described in the previous section, by breaking $\mathcal{F}$ up into individual hyperedges and bounding the number of ways to choose each one individually, given the previous choices. We will consider a number of different cases.

First, suppose there exists a hyperedge $e_1\in E(\mathcal{F})$ such that $|e_1\cap R|\geq2$. By definition of a secondary configuration, every hyperedge of a secondary configuration contains a neutral vertex, so $e_1\nsubseteq R$ and  $|e_1\cap R|$ is between $2$ and $r-1$. Thus, by condition~\ref{codegBound} of Definition~\ref{hypwellBDef} the number of hyperedges $f_1$ in $\mathcal{H}$ intersecting $S$ in exactly $|e_1\cap R|$ vertices is 
\[\binom{|S|}{|e_1\cap R|} \cdot \Delta_{|e_1\cap R|}(\mathcal{H}) = O\left(d^{\frac{r- |e_1\cap R|}{r-1}}\log^{-K}(d)\right).\]
Note that this bound is already enough to complete the proof in the case  $|E(\mathcal{F})|=1$ since, by definition of a secondary configuration, the unique hyperedge of $\mathcal{F}$ contains at least two roots. So, in what follows, we may assume that $|E(\mathcal{F})|\geq2$.

Let $e_2$ be a hyperedge which intersects $e_1$ (which exists by definition of a secondary configuration) and, if $|E(\mathcal{F})|\geq3$, then let $e_3$ be the remaining hyperedge. The number of copies of $X$ rooted at $S$ is at most the number of ways to choose a hyperedge $f_1$ of $\mathcal{H}$ intersecting $S$ on exactly $|e_1\cap R|$ vertices, a hyperedge $f_2$ intersecting $S\cup f_1$ on exactly $|e_2\cap (R\cup e_1)|$ vertices and, if $|E(\mathcal{F})|\geq 3$, a hyperedge $f_3$ intersecting $S\cup f_1\cup f_2$ on exactly $|e_3\cap (R\cup e_1\cup e_2)|$ vertices. Using the bound that we have already proven for the number of ways of choosing $f_1$, we get that this is
\[O\left(\prod_{i=1}^{|E(\mathcal{F})|} \Delta_{\left|e_i\cap \left(R\cup \bigcup_{j=1}^{i-1}e_j\right)\right|}(\mathcal{H})\right)=O\left(\log^{-K}(d) \cdot \prod_{i=1}^{|E(\mathcal{F})|} d^{\frac{r-\left|e_i\cap \left(R\cup \bigcup_{j=1}^{i-1} e_j\right)\right|}{r-1}}\right)\]
By \eqref{V-R}, the exponent of $d$ in the above expression is precisely $\frac{|V(\mathcal{F})|-|R|}{r-1}$, and so we are done when there exists some hyperedge $e_1$ such that $|e_1\cap R|\geq2$.

So from now on we assume that every hyperedge of $\mathcal{F}$ contains at most one root. In particular, by definition of a secondary configuration, the central hyperedge has exactly one root.

Now, let $e_1$ denote the central hyperedge and suppose that there is a non-central hyperedge $e_2$ such that $e_2\subseteq R\cup e_1$. Then, since $e_2$ contains at most one root, we must have that $|e_1\cap e_2|=r-1$ and that the unique vertex of $e_2\setminus e_1$ is a root. The vertex of $e_1\setminus e_2$ is also a root because, by definition of a secondary configuration,  $e_1$ contains a root and this root cannot be contained in $e_2$ (as every hyperedge contains at most one root). By condition~\ref{neighSim} of Definition~\ref{hypwellBDef}, the number of ways to choose two vertices $x,y$ of $S$ and two hyperedges $f_1$ and $f_2$ of $\mathcal{H}$ such that $f_1\triangle f_2=\{x,y\}$ is 
\[O\left(d\cdot \log^{-K}(d)\right).\]
If $e_1$ and $e_2$ are the only two hyperedges of $\mathcal{F}$, then $|V(\mathcal{F})|-|R|=r-1$ and so this bound is what we wanted to prove. If $|E(\mathcal{F})|=3$, then there are 
\[O\left(\Delta_{\left|e_3\cap (R\cup e_1\cup e_2)\right|}(\mathcal{H})\right) = O\left(d^{\frac{r-\left|e_3\cap (R\cup e_1\cup e_2)\right|}{r-1}}\right)\]
ways to choose a third hyperedge to form a copy of $\mathcal{F}$. Combining this with the bound on the number of ways to choose the first two hyperedges and applying \eqref{V-R} gives the desired bound.

So every non-central hyperedge contains at least one non-root vertex which is not contained in the central hyperedge. We can now conclude the proof in the case $|E(\mathcal{F})|=2$. Indeed, let $e_1$ be the central hyperedge and $e_2$ be the non-central hyperedge. We can bound the number of copies of $X$ by 
\[O\left(\Delta(\mathcal{H})\cdot \Delta_{|e_2\cap (e_1\cup R)|}(\mathcal{H})\right).\]
By definition of a secondary configuration and the fact that $e_1$ has at most one root, we know that $|e_2\cap (e_1\cup R)|$ must be at least two. Also, it is at most $r-1$ by the result of the previous paragraph. So, by condition~\ref{codegBound} of Definition~\ref{hypwellBDef}, we get an upper bound of 
\[O\left(d\cdot d^{\frac{r-|e_2\cap (e_1\cup R)|}{r-1}} \cdot \log^{-K}(d)\right)\]
which, by \eqref{V-R} and the fact that $|e_1\cap R|=1$, is the desired bound. 

It remains to consider the case $|E(\mathcal{F})|=3$. Let $e_1$ be the central hyperedge and let $e_2$ and $e_3$ be the other two hyperedges. Suppose that $|R|\geq 2$. Then, since $e_1$ contains exactly one root, there must be a non-central hyperedge, say $e_2$, such that $e_2\setminus e_1$ contains a root. We can now bound the number of copies of $X$ by
\[O\left(\Delta(\mathcal{H})\cdot \Delta_{|e_2\cap (R\cup e_1)|}(\mathcal{H}) \cdot \Delta_{|e_3\cap (R\cup e_1\cup e_2)|}(\mathcal{H})\right).\]
Since $|e_2\cap (R\cup e_1)|$ is at least two and at most $r-1$, this is bounded above by 
\[O\left(d \cdot d^{\frac{r-|e_2\cap (R\cup e_1)|}{r-1}} \cdot d^{\frac{r-|e_3\cap (R\cup e_1\cup e_2)|}{r-1}}\cdot \log^{-K}(d)\right)\]
and so, in this case, we are again done by \eqref{V-R} and the fact that $|e_1\cap R|=1$. 

Thus, there is exactly one root and it is contained in $e_1$. By definition of a secondary configuration, this implies that $e_1\cap e_2\cap e_3=\emptyset$ and that $e_2$ intersects $e_3$. In particular, it implies that $|e_3\cap (e_1\cup e_2)|\geq2$. We assume that $e_2$ was chosen to be the non-central hyperedge such that $|e_1\cap e_2|$ is maximised. As above, the number of copies of $X$ is bounded above by 
\[O\left(\Delta(\mathcal{H})\cdot \Delta_{|e_2\cap e_1|}(\mathcal{H}) \cdot \Delta_{|e_3\cap (e_1\cup e_2)|}(\mathcal{H})\right)\]
which gives the desired bound by condition~\ref{codegBound} of Definition~\ref{hypwellBDef} unless $|e_2\cap e_1|=1$ and $|e_3\cap (e_1\cup e_2)|=r$ (since we already know that $|e_3\cap (e_1\cup e_2)|\geq2$). So, we assume that $\mathcal{F}$ satisfies these conditions. By our choice of $e_2$, we also get that $|e_3\cap e_1|=1$ as well. The last case to consider is therefore when $e_2$ and $e_3$ each intersect $e_1$ on on a single vertex (where these two vertices are distinct) and $|e_2\cap e_3|=r-1$. In this case, the number of copies of $X$ is bounded above by the number of ways to choose a hyperedge $f_1$ containing $S$, choose two vertices $x,y$ of $f_1\setminus S$ and then choose two hyperedges $f_2,f_3$ such that $f_2\triangle f_3=\{x,y\}$. By condition~\ref{neighSim} of Definition~\ref{hypwellBDef}, this is bounded above by
\[O\left(d^2\log^{-K}(d)\right),\]
which is what we needed since $|V(\mathcal{F})|-|R|=2(r-1)$ in this case. This completes the proof. 
\end{proof}

In our application of the differential equations method in Section~\ref{sec:diff}, it is often useful for us to approximate certain sums by a related integral. For this, we use the following simple lemma. We remark that a very similar statement is derived in~\cite[Claim~3.5]{triadic} using the same proof.

\begin{lem}
\label{sumToInt}
For $T>0$, let $s(t)$ be a function which is differentiable and has bounded derivative on $[0,T]$. Then, for non-negative integers $N$ and $m$ such that $m\leq TN$, we have
\[\left|\int_0^{m/N}s(t)dt - \frac{1}{N}\sum_{i=0}^{m-1}s(i/N)\right|\leq \frac{m\cdot \sup_{t\in[0,T]}|s'(t)|}{2N^2}.\]
\end{lem}

\begin{proof}
Let $a,b\in [0,T]$ with $a\leq b$. As for all $t \in [a,b]$, $s(t) \le s(a) + (t-a)\sup_{t \in [a,b]}|s'(t)|$, we have that 
\[\int_a^b s(t) dt \leq \int_a^b\left(s(a) + (t-a)\sup_{x\in[a,b]}|s'(x)|\right)dt \]
\[= (b-a)s(a) +\left(\frac{(b-a)^2}{2}\right)\sup_{t\in[a,b]}|s'(t)|\]
and, similarly, 
\[\int_{a}^bs(t)dt\geq (b-a)s(a) -\left(\frac{(b-a)^2}{2}\right)\sup_{t\in[a,b]}|s'(t)|.\]
So, setting $a=i/N$ and $b=(i+1)/N$ for $0\leq i\leq m-1$, we obtain
\[\left|\int_{i/N}^{(i+1)/N}s(t)dt - \frac{s(i/N)}{N}\right|\leq \left(\frac{1}{2N^2}\right)\sup_{t\in[0,T]}|s'(t)|.\]
Summing up these expressions and applying the triangle inequality, we have
\[\left|\int_{0}^{m/N}s(t)dt - \sum_{i=0}^{m-1}\frac{s(i/N)}{N}\right|\leq \left(\frac{m}{2N^2}\right)\sup_{t\in[0,T]}|s'(t)|\]
as desired.
\end{proof}

\section{Concentration at Time Zero}
\label{sec:timeZero}

Our goal in this section is to prove Lemmas~\ref{Yijrough},~\ref{Xbound} and~\ref{Wbound}  in the case $m=0$. Lemma~\ref{Qrough} will follow from Lemmas~\ref{Yijrough} and~\ref{Wbound} and Proposition~\ref{Ibound} via Lemma~\ref{QfromOthers}. In fact, we will actually prove the following stronger bounds in order to give ourselves some extra room in the next section.

\begin{lem}
\label{timeZeroLemmaW}
With probability at least $1-N^{-10\sqrt{\log(N)}}$ the following statement holds. For each $1\leq i\leq r$ and set $S\subseteq V(\mathcal{H})$, we have
\[W_S^i(0)\leq \log^{2r}(d).\]
\end{lem}

\begin{lem}
\label{timeZeroLemmaUseful}
With probability at least $1-N^{-10\sqrt{\log(N)}}$ the following statement holds. For every secondary configuration $X=(\mathcal{F},R,D)$  and set $S\subseteq V(\mathcal{H})$, we have
\[X_S(0)\leq d^{\frac{|V(\mathcal{F})|-|R|-|D|}{r-1}}\log^{-4K/5}(d)\]
\end{lem}

Recall the definitions of $y_{i,0}(t)$ and $y_{i,j}(t)$ from \eqref{yi0def} and \eqref{yijdef}. We will prove the following.
\begin{lem}
\label{YTIMEZERO}
With probability at least $1-N^{-10\sqrt{\log(N)}}$ the following statement holds. For every pair of non-negative integers $i$ and $j$ such that $i\leq r-2$ and $i+j\leq r-1$ and any vertex $v\in V(\mathcal{H})$, we have
\[Y_v^{i,j}(0)\in \left(1\pm \log^{-3K/10}(d)\right)y_{i,j}(0)d^{1-\frac{i}{r-1}}.\]
\end{lem}

 Note that we get the following concentration result for $Q(0)$ from Lemmas~\ref{timeZeroLemmaW},~\ref{YTIMEZERO} and Proposition~\ref{Ibound} via Lemma~\ref{QfromOthers}.

\begin{lem}
\label{timeZeroQ}
With probability at least $1-N^{-9\sqrt{\log(N)}}$, we have
$$Q(0) \in (1 \pm 4\epsilon(0))\gamma(0) \cdot N.$$
\end{lem}

We will prove Lemmas~\ref{timeZeroLemmaW},~\ref{timeZeroLemmaUseful} and~\ref{YTIMEZERO} by applying Corollary~\ref{vunew}. Although the $Y$ configurations are arguably the most important, we save proving Lemma~\ref{YTIMEZERO} until last; the proofs of the first two lemmas involve more simple applications of Corollary~\ref{vunew} and hence provide a more gentle introduction for the reader to the style of arguments we will be using throughout the section. We remark that in the proof of Lemma~\ref{timeZeroLemmaW}, we technically do not need to rescale our random variable, and so could apply Theorem~\ref{vu} directly. However it is marginally simpler to apply Corollary~\ref{vunew}, so this is what we shall do. 

 We will use the following random variables throughout the rest of the section. Given $w\in V(\mathcal{H})$, let $\xi_w$ be the Bernoulli random variable which is equal to one if and only if $w \in I(0)$. Without further ado we present the proofs of Lemmas~\ref{timeZeroLemmaW},~\ref{timeZeroLemmaUseful} and~\ref{YTIMEZERO}. 

\begin{proof}[Proof of Lemma~\ref{timeZeroLemmaW}]
Let $S\subseteq V(\mathcal{H})$ be a set of cardinality $i$. If $i=r$, then clearly $W^i_S(0)\leq 1$ and so we may assume that $1\leq i\leq r-1$. Observe that $W^i_S(0)$ can be written as $f\left(\xi_v: v\in V(\mathcal{H})\right)$ where 
\[f\left(x_v: v\in V(\mathcal{H})\right):=\sum_{\substack{e\in E\left(\mathcal{H}\right)\\ S\subseteq e}}\left(\prod_{v\in e\setminus S}x_v\right).\]
Observe that no variable in $f$ has an exponent greater than 1 and the degree of $f$ is $r-i$. We wish to apply Corollary~\ref{vunew} to obtain an upper bound on $W_S^i(0)$ which holds with high probability. In order to do this, we must bound $\mathbb{E}_j(W_S^i(0))$ for $0 \le j \le r-i$. 

Let $A$ be a set of at most $r-i$ vertices of $\mathcal{H}$ disjoint from $S$. We have
\[\partial_A f\left(x_v: v\in V(\mathcal{H})\right) = \sum_{\substack{e\in E(\mathcal{H}) \\ A\cup S\subseteq e}}\left(\prod_{v\in e\setminus(S\cup A)}x_v\right).\]
Therefore, by linearity of expectation and independence we have
$$\mathbb{E}\left(\partial_A f\left(\xi_v: v\in V(\mathcal{H})\right)\right) = \sum_{\substack{e\in E(\mathcal{H}) \\ A\cup S\subseteq e}}p^{r-i-|A|}.$$
In the case that $|A|=r-i$, the above expression is simply equal to $0$ or $1$ (depending on whether $A\cup S$ is a hyperedge of $\mathcal{H}$ or not). Otherwise, 
\[\mathbb{E}\left(\partial_A f\left(\xi_v: v\in V(\mathcal{H})\right)\right) = \sum_{\substack{e\in E(\mathcal{H}) \\ A\cup S\subseteq e}}p^{r-i-|A|}\leq \Delta_{|A|+i}\left(\mathcal{H}\right)\cdot p^{r-i-|A|}\]
\[=\Delta_{|A|+i}\left(\mathcal{H}\right) c^{r-i-|A|}d^{\frac{|A|+i-1}{r-1}-1}.\]
By conditions~\ref{maxDeg} and~\ref{codegBound} of Definition~\ref{hypwellBDef}, this expression is  $o(1)$ if $|A|+i\geq2$ and is at most $c^{r-1}$ otherwise (i.e.~if $A=\emptyset$ and $i=1$).

This analysis gives
\[\mathbb{E}(W_S^i(0)) \le c^{r-1}\]
and for $1 \le j \le r-i$,
\[
\mathbb{E}_j(W_S^i(0)) \le 1.
\]

Set $\tau:= 1$ and $\mathcal{E}_0:= \log^{2r}N$. As $\mathbb{E}(W_s(0)) = O(1)$ and
$$\tau \log^{r-i} N \sqrt{\mathcal{E}_o} = o\left(\log^{2r}(d)\right),$$
applying Corollary~\ref{vunew} gives that with probability at least $1 - N^{-20\sqrt{\log N}}$,
$$W_S^i(0) \le c^{r-1} + o\left(\log^{2r}(d)\right) < \log^{2r}(d).$$

The result now follows by taking a union bound over all values of $i$ and all subsets of $V(\mathcal{H})$ of cardinality $i$. 
\end{proof}

Before proving Lemmas~\ref{timeZeroLemmaUseful} and~\ref{YTIMEZERO} it is helpful to introduce the following definition and simple claim.

\begin{defn}
Let $X = (\mathcal{F},R,D)$ be a configuration and $S$ be a subset of $V(\mathcal{H})$. Let $\mathcal{T}_{X,S}$ be the collection of all pairs $(\mathcal{F}',D')$, with $\mathcal{F}'$ a subhypergraph of $\mathcal{H}$ and $D' \subseteq V(\mathcal{F}')$, such that there exists an isomorphism $\phi$ from $\mathcal{F}$ to $\mathcal{F}'$ such that $\phi(R)=S$ and $\phi(D)=D'$.  
\end{defn}

\begin{claim}\label{bigT}
Let $X = (\mathcal{F},R,D)$ be a configuration and $S \subseteq V(\mathcal{H})$. Then $X_S(0)$ is bounded above by $f\left(\xi_v: v\in V(\mathcal{H})\right)$, where
\[f\left(x_v: v\in V(\mathcal{H})\right) := \sum_{(\mathcal{F}',D')\in \mathcal{T}_{X,S}} \left(\prod_{v\in D'}x_v\right).\]
\end{claim}

\begin{proof}
 Let $\mathcal{F}' \subseteq \mathcal{H}(0)$. By Definition~\ref{copyDef}, $\mathcal{F}'$ is a copy of $X$ rooted at $S$ in $\mathcal{H}(0)$ only if there exists an isomorphism $\phi: \mathcal{F} \rightarrow \mathcal{F}'$ with $\phi(R) = S$ and $D' := \phi(D) \subseteq I(0)$. For such a $\phi$ and $D'$, say that $(\mathcal{F}',\phi, D')$ is a \emph{witness triple} for $\mathcal{F}'$.
 
 The number of copies of $X$ rooted at $S$ in $\mathcal{H}(0)$ is at most the number of $\mathcal{F'} \subseteq \mathcal{H}(0)$ such that there exists some $\phi$ and $D'$ where $(\mathcal{F}',\phi, D')$ is a witness triple for $\mathcal{F}'$. This is at most the number of pairs $(\mathcal{F}', D')$ in $\mathcal{T}_{X,S}$ such that $D' \subseteq I(0)$. 
 
 Therefore $X_S(0)$ is bounded above by $f\left(\xi_v: v\in V(\mathcal{H})\right)$, where
 \[f\left(x_v: v\in V(\mathcal{H})\right) := \sum_{(\mathcal{F}',D')\in \mathcal{T}_{X,S}} \left(\prod_{v\in D'}x_v\right),\]
as required. 
  
\end{proof}

\begin{obs}\label{counttoomuch}
Let $\mathcal{F}'$ be a copy of $X = (\mathcal{F},R,D)$ rooted at $S$ in $\mathcal{H}(0)$. Let us consider how $\mathcal{F}'$ may be counted multiple times by  $f\left(\xi_v: v\in V(\mathcal{H})\right)$. This will happen precisely when there exist two witness triples (defined in the proof above) for $\mathcal{F}'$ of the form $(\mathcal{F}',\phi_1,D_1)$ and $(\mathcal{F}', \phi_2,D_2)$ such that $D_1 \not = D_2$ (and so $\phi_1 \not= \phi_2$).

When $|I(0) \cap \mathcal{F}'| = |D|$, there is only one choice for the set $D'$ in a witness triple and so no such pair $(\mathcal{F}',\phi_1,D_1)$ and $(\mathcal{F}', \phi_2,D_2)$ exists. However, if $|I(0) \cap \mathcal{F}'| > |D|$, then there may exist subsets $D_1\not=D_2$ of $\mathcal{F}'$ (and isomorphisms $\phi_1$ and $\phi_2$) such that $(\mathcal{F}',\phi_1,D_1)$ and $(\mathcal{F}', \phi_2,D_2)$ are both witness triples for $X$. In this case both $(\mathcal{F}',D_1)$ and $(\mathcal{F}',D_2)$ are in $\mathcal{T}_{X,S}$ and $\mathcal{F}'$ is counted multiple times by  $f\left(\xi_v: v\in V(\mathcal{H})\right)$.

So the difference between $X_S(0)$ and $f\left(\xi_v: v\in V(\mathcal{H})\right)$ is at most $O(1)$ times the number of copies of configurations $X' = (\mathcal{F}, R, D')$, where $D' := D \cup \{u\}$, for some $u \in V(\mathcal{F})\setminus (D \cup R)$.
\end{obs}

We will now return to proving Lemmas~\ref{timeZeroLemmaUseful} and~\ref{YTIMEZERO}.

\begin{proof}[Proof of Lemma~\ref{timeZeroLemmaUseful}]
Let $X=(\mathcal{F},R,D)$ be a secondary configuration and let $S\subseteq V(\mathcal{H})$ be a set of cardinality $|R|$. If $D=\emptyset$, $X_S(0)$ is simply bounded above by the number of copies of $X$ in $\mathcal{H}$ rooted at $S$. By Lemma~\ref{Xcount}, this is $O\left(d^{\frac{|V(\mathcal{F}|-|R|}{r-1}}\log^{-K}(d)\right)$, and this bound is actually stronger than we need. So, from now on, we assume that $D\neq\emptyset$.

By Claim~\ref{bigT}, letting $\mathcal{T}:= \mathcal{T}_{X,S}$, we have that the variable $X_S(0)$ is bounded above by $f\left(\xi_v: v\in V(\mathcal{H})\right)$ where
\[f\left(x_v: v\in V(\mathcal{H})\right) := \sum_{(\mathcal{F}',D')\in \mathcal{T}} \left(\prod_{v\in D'}x_v\right).\]
Observe that the degree of $f$ is $|D|$ and no variable in $f$ has an exponent greater than 1. 

We wish to apply Corollary~\ref{vunew} with
$$\tau:= d^{\frac{|V(\mathcal{F})| - |R|-|D|}{r-1}}\log^{-9K/10}(d)$$
  and $\mathcal{E}_0:= \log^{2|D| + 1}(N)$ to obtain an upper bound for $X_S(0)$ which holds with high probability. As above, in order to apply Corollary~\ref{vunew} we must bound $\mathbb{E}_j(X_S(0))$ for $0 \le j \le |D|$. 

If $A$ contains an element of $S$, then $\partial_A f=0$. On the other hand, if $A$ is a subset of $V(\mathcal{H})\setminus S$ of cardinality at most $|D|$, then
\[\partial_A f(x_v:v\in V(\mathcal{H})) = \sum_{\substack{(\mathcal{F}',D')\in \mathcal{T}\\ A\subseteq D'}}\left(\prod_{v\in D'\setminus A}x_v\right)\]
and so, by linearity of expectation and independence,
\begin{equation}
\label{Xpartialexp}
\mathbb{E}\left(\partial_A f(\xi_v:v\in V(\mathcal{H}))\right) = \sum_{\substack{(\mathcal{F}',D')\in \mathcal{T}\\ A\subseteq D'}}p^{|D'\setminus A|}.
\end{equation}
Recalling Remark~\ref{rootTransfer}, we see that the number of $(\mathcal{F}',D')\in\mathcal{T}$ with $A\subseteq D'$ is at most the sum of $X'_{S\cup A}(0)$ over all secondary configurations $X'$ with $|V(\mathcal{F})|$ vertices, $|R|+|A|$ roots and zero marked vertices multiplied by a constant factor (as there is a choice for which $|A|$ roots are in $D'$). So,  by Lemma~\ref{Xcount}, we get that the right side of \eqref{Xpartialexp} is bounded above by
\[O\left(d^{\frac{|V(\mathcal{F})|-|R|-|A|}{r-1}}\log^{-K}(d) p^{|D|-|A|}\right)\]
\[= O\left(d^{\frac{|V(\mathcal{F})| - |R|-|D|}{r-1}}\log^{-K}(d)\right)= O\left(\tau \log^{-K/10}(d)\right)=o(\tau).\]
So for $0 \le j \le |D|$,
$$\mathbb{E}_j(X_S(0)) = o(\tau).$$
 
As $X$ is secondary, $|D| \le 3r - 1$. So for $K$ large with respect to $r$,
\begin{align*}
\tau \log^{|D|}(N) \sqrt{\mathcal{E}_0} &\le  d^{\frac{|V(\mathcal{F})| - |D| - |R|}{r-1}} \log^{6r}(d) \log^{-9K/10}(d)\\
& = o\left(d^{\frac{|V(\mathcal{F})| - |D|-|R|}{r-1}}\log^{-4K/5}(d)\right).
\end{align*} 

Using this and the fact that $\mathbb{E}(X_S(0)) = o(\tau)$, applying Corollary~\ref{vunew} gives that
$$X_S(0) =  o\left(d^{\frac{|V(\mathcal{F})| - |D|-|R|}{r-1}}\log^{-4K/5}(d)\right).$$
with probability at least $1-N^{-20\sqrt{\log(N)}}$. The result follows by taking a union bound over all secondary configurations and choices of $S$.
\end{proof}

\subsection{Proof of Lemma~\ref{YTIMEZERO}}
First, note that it suffices to consider the case that $i$ and $j$ are not both zero, since $Y^{0,0}_v(0)=\deg(v)$  and so the bounds hold for $Y^{0,0}_v(0)$ by conditions~\ref{maxDeg} and~\ref{almostRegular} of Definition~\ref{hypwellBDef}. Thus, from now on, we assume $i+j\geq1$.

Write the configuration $Y^{i,j}$ as $(\mathcal{F},R,D)$. By Claim~\ref{bigT}, setting $\mathcal{T}:= \mathcal{T}_{Y^{i,j},\{v\}}$ gives that the variable $Y_v^{i,j}(0)$ is bounded above by $\tilde{Y}^{i,j}_v:=f(\xi_u: u\in V(\mathcal{H}))$ where 
\[f\left(x_u: u\in V(\mathcal{H})\right) := \sum_{(\mathcal{F}',D')\in \mathcal{T}} \left(\prod_{u\in D'}x_u\right).\]
Note that by definition of $Y^{i,j}$, $f$ has degree $i + j(r-1)$. Observe that no variable in $f$ has an exponent greater than 1.

We will prove the following.
\begin{prop}\label{Ytildebound}
For each $v \in V(\mathcal{H})$, $0 \le i \le r-2$ and $0 \le j \le r-1-i$ such that $i + j \ge 1$, with probability $1 - N^{-20\sqrt{\log(N)}}$ we have
$$\tilde{Y}^{i,j}_v \in \left(1\pm o\left(\log^{-3K/10}(d)\right)\right)y_{i,j}(0)d^{1-\frac{i}{r-1}}.$$
\end{prop}
We now show that $\tilde{Y}^{i,j}_v$ is a good approximation for $Y_v(0)$, and hence it suffices to prove Proposition~\ref{Ytildebound}.
\begin{prop}
\label{Ysuffclaim}
If for all $v \in V(\mathcal{H})$, $0 \le i \le r-2$ and $0 \le j \le r-1-i$ such that $i + j \ge 1$,
$$\tilde{Y}^{i,j}_v \in \left(1\pm o\left(\log^{-3K/10}(d)\right)\right)y_{i,j}(0)d^{1-\frac{i}{r-1}},$$
and for all $w \in V(\mathcal{H})$,
$$W^1_w(0) \le \log^{2r}(d),$$
then
$$Y_v^{i,j}(0) \in \left(1\pm \log^{-3K/10}(d)\right)y_{i,j}(0)d^{1-\frac{i}{r-1}}.$$
\end{prop}
The proof of Lemma~\ref{YTIMEZERO} follows from Propositions~\ref{Ytildebound} and~\ref{Ysuffclaim} and Lemma~\ref{timeZeroLemmaW} by applying the union bound. 
\begin{proof}[Proof of Propostion~\ref{Ysuffclaim}]
 Fix $v \in V(\mathcal{H})$. By Observation~\ref{counttoomuch}, it may be the case that $\tilde{Y}^{i,j}_v$ counts an element $\mathcal{F}'$ of $Y^{i,j}_v(0)$ more than once if it contains more than $i+(r-1)j$ elements of $I(0)$. However, we have
\begin{equation}
\label{tired}
\overbar{Y}^{i,j}_v := \tilde{Y}^{i,j}_v-Y_v^{i,j}(0) = O\left( \sum_{\mathcal{F}'' \in Y_v^{i+1,0}(0)}\sum_{\substack{U\subseteq V(\mathcal{F}'')\setminus\{v\}\\ |U| = j}}\prod_{u \in U}W^1_u(0)\right).
\end{equation}

 We now prove that
 $$\overbar{Y}^{i,j}_v = o\left(d^{1 - \frac{i}{r-1}} \log^{-3K/10}(d)\right),$$
 from which the claim follows. 
  If $i=r-2$, then $Y_v^{i+1,0}(0) = W^1_v(0)$, so by hypothesis (using the fact that $j<r$) we have $\overbar{Y}^{i,j}_v = O\left( \log^{2r^2}(d)\right)$. Otherwise, for $i < r-2$ we have by hypothesis that
  $$Y_v^{i+1,0}(0) \le \tilde{Y}^{i+1,0}_v  = O\left(d^{1- \frac{i+1}{r-1}}\right).$$ 
Using this, as $j < r$, from \eqref{tired} we get 
 \begin{align*}
  \overbar{Y}^{i,j}_v  &= O\left(\sum_{\mathcal{F}'' \in Y_v^{i+1,0}(0)}\sum_{\substack{U\subseteq V(\mathcal{F}'')\setminus\{v\}\\ |U| = j}}\prod_{u \in U}W^1_u(0)\right)\\
  & = O\left( Y_v^{i+1,0}(0) \log^{2rj}(d)\right)\\
  &= O\left(d^{1- \frac{i+1}{r-1}}\log^{2r^2}(d)\right)\\
  &= o\left(d^{1 - \frac{i}{r-1}} \log^{-3K/10}(d)\right),
 \end{align*}
 as required. The claim follows.
\end{proof}

It remains to prove Proposition~\ref{Ytildebound}. 
\begin{proof}[Proof of Proposition~\ref{Ytildebound}]
Fix $v \in V(\mathcal{H})$. We wish to apply Corollary~\ref{vunew} with
$$\tau:= d^{1-\frac{i}{r-1}}\log^{-9K/10}(d)$$
and $\mathcal{E}_0:= \log^K(d)$ to obtain bounds on $\tilde{Y}^{i,j}_v$ that hold with high probability.  As before, in order to apply Corollary~\ref{vunew}, we must bound $\mathbb{E}\left(\tilde{Y}^{i,j}_v\right)$ for $0 \le \ell \le i + (r-1)j$.

\begin{claim}
\label{Y0claim}
We have
$$\mathbb{E}\left(\tilde{Y}^{i,j}_v\right) =  \left(1\pm 3r\log^{-K}(d)\right)y_{i,j}(0)d^{1 - \frac{i}{r-1}},$$
and for $1 \le \ell \le i + (r-1)j$ we have
$$\mathbb{E}_{\ell}\left(\tilde{Y}^{i,j}_v\right) = o(\tau).$$
\end{claim}

\begin{proof}
If $A$ contains $v$, then $\partial_A f=0$. On the other hand, if $A$ is a subset of $V(\mathcal{H})\setminus \{v\}$ of cardinality at most $i+(r-1)j$, then
\[\partial_A f\left(x_u: u\in V(\mathcal{H})\right) = \sum_{\substack{(\mathcal{F}',D')\in \mathcal{T}\\ A\subseteq D'}}  \left(\prod_{u\in D'\setminus A}x_u\right)\]
and so, by linearity of expectation and independence,
\begin{equation}
\label{expectA}
\mathbb{E}\left(\partial_A f\left(\xi_u: u\in V(\mathcal{H})\right)\right)= \sum_{\substack{(\mathcal{F}',D')\in \mathcal{T}\\ A\subseteq D'}}  p^{|D'\setminus A|}.\end{equation}
Now let us bound the number of $(\mathcal{F}',D') \in \mathcal{T}$ with $A = \emptyset$. This is the number of ways to:
\begin{itemize}
\item[(1)] choose a hyperedge $e \in \mathcal{H}$ such that $v \in e$,
\item[(2)] choose vertices $\{v_1,\ldots,v_i\}$ in $e \setminus \{v\}$,
\item[(3)] choose vertices $\{w_1,\ldots,w_j\}$ in $e \setminus \{v,v_1,\ldots,v_j\}$,
\item[(4)] choose hyperedges $e_1,\ldots,e_j$, sequentially, such that $e_k \cap \left(e \cup \bigcup_{a < k} e_a\right) = \{w_k\}$. 
\end{itemize}

By condition~\ref{maxDeg} of Definition~\ref{hypwellBDef}, there are are $(1 \pm \log^{-K}(d))d$ choices of $e$ for (1). Given the choice of $e$, by conditions~\ref{almostRegular} and~\ref{codegBound} of Definition~\ref{hypwellBDef}, there are $(1 \pm \log^{-K}(d))d$ choices for each of $e_1,\ldots,e_j$. Therefore 
\[|\mathcal{T}|\in \left(1\pm \log^{-K}(d)\right)^{j+1} \binom{r-1}{i} \binom{r-i-1}{j} d^{j+1}\]
\[\subseteq \left(1\pm 3r\log^{-K}(d)\right)  \binom{r-1}{i} \binom{r-i-1}{j} d^{j+1}\]
Therefore, applying \eqref{expectA} to the case $A=\emptyset$, we have
\[\mathbb{E}\left(\tilde{Y}^{i,j}_v\right)\in \left(1\pm 3r\log^{-K}(d)\right)  \binom{r-1}{i} \binom{r-i-1}{j} d^{j+1}p^{i+j(r-1)}\] 
\[=\left(1\pm 3r\log^{-K}(d)\right)y_{i,j}(0)d^{1 - \frac{i}{r-1}}.\] 

Now, if $A\neq\emptyset$, then the number of $(\mathcal{F}',D')\in \mathcal{T}$ with $A\subseteq D'$ is at most the number of ways to partition $A$ into $j+1$  sets $A_0,\dots,A_{j}$ (some of which may be empty) and do the following:
\begin{enumerate}[(1)]
\item choose a hyperedge $e$ such that $e\cap (A\cup\{v\})=A_0\cup \{v\}$, 
\item choose a subset $W=\{w_1,\dots,w_j\}$ of $e\setminus\left(A_0\cup \{v\}\right)$,
\item for $k=1,\dots,j$, choose a hyperedge $e_k$ of $\mathcal{H}$ containing $A_k\cup\{w_k\}$. 
\end{enumerate}
The number of such partitions of $A$ is $O(1)$. For any such partition $A_0,\dots,A_j$, by condition~\ref{codegBound} of Definition~\ref{hypwellBDef} and the fact that $|A|\geq1$, the number of elements of $\mathcal{T}$ generated by the above procedure is at most
\[O\left(\prod_{k=0}^j\Delta_{|A_k|+1}(\mathcal{H})\right) = O\left(d^{j+1 - \frac{|A|}{r-1}}\log^{-K}(d). \right)\]
 Combining this with \eqref{expectA}, we get that
\[\mathbb{E}\left(\partial_A f\left(\xi_u: u\in V(\mathcal{H})\right)\right)= O\left(d^{j+1 - \frac{|A|}{r-1}}\log^{-K}(d)  p^{i+j(r-1)-|A|}\right)\]
\[= O\left(d^{1-\frac{i}{r-1}}\log^{-K}(d)\right)=O\left(\tau\log^{-K/10}(d)\right)=o(\tau).\]
This completes the proof of Claim~\ref{Y0claim}.
\end{proof}
The degree of $f$ can be crudely bounded by $r^2$. So as for $K$ sufficiently large with respect to $r$,
$$\tau \log^{r^2}N \sqrt{\mathcal{E}_0} = o\left(d^{1- \frac{i}{r-1}} \log^{-3K/10}\right),$$
using the value of $\mathbb{E}\left(\tilde{Y}^{i,j}_v\right)$ given by Claim~\ref{Y0claim} and applying Corollary~\ref{vunew} gives that with probability at least $1 - N^{-20\sqrt{\log N}}$ we have
$$\tilde{Y}^{i,j}_v(0) \in \left(1\pm o\left(\log^{-3K/10}(d)\right)\right)y_{i,j}(0)d^{1-\frac{i}{r-1}}.$$
This completes the proof of Proposition~\ref{Ytildebound}.
\end{proof}
Proposition~\ref{Ytildebound} was the final piece for the proof of Lemma~\ref{YTIMEZERO}. Thus this completes the proof of Lemma~\ref{YTIMEZERO} and hence our analysis of $m=0$ is now complete. 

\section{The First Phase After Time Zero}
\label{sec:diff}

In this section, we will use the differential equations method to prove Lemmas~\ref{Yijrough},~\ref{Xbound} and~\ref{Wbound} for $1 \le m \le M$, where $M$ is defined in \eqref{Mdef}. 
\begin{defn}\label{bmdef}
For $0\leq m\leq M$, let $\mathcal{B}_m$ be the event (in $\Omega'$, which was defined in Subsection~\ref{sec:probspace}) that there exists $0\leq \ell\leq m$ such that, for some $v \in V(\mathcal{H})$ or $S \subseteq V(\mathcal{H})$, one of the following four statements fails to hold.
\begin{enumerate}
\renewcommand{\theenumi}{B.\arabic{enumi}}
\renewcommand{\labelenumi}{(\theenumi)}
\item\label{Ystate}  For all $0 \le i \le r-2$ and $0 \le j \le r-1-i$:
$$Y_v^{i,j}(\ell) \in \left(1 \pm \epsilon(t_{\ell})\right)y_{i,j}(t_\ell) \cdot d^{1 - \frac{i}{r-1}}.$$
\item\label{Xstate} For any secondary configuration $X = (\mathcal{F},R,D)$,
$$X_S(\ell) \leq \log^{2|D|r^4}(d) \cdot d^{\frac{|V(\mathcal{F})| - |R| - |D|}{r-1}}\log^{-3K/5}(d).$$
\item\label{Wstate}
For all $1 \le i \le r$,
$$W_S^i(\ell) \leq\log^{r^3(r-i)}(d).$$
\item\label{Istate}
$I(m) = O\left(\log N \cdot Nd^{-1/(r-1)}\right).$
\end{enumerate}
\end{defn}
Note that these are precisely the bounds we wish to prove for Lemmas~\ref{Wbound},~\ref{Xbound} and~\ref{Yijrough} and Proposition~\ref{Ibound}.
One should think of $\mathcal{B}_m$ as the event that one of the variables that we are tracking has strayed far from its expected trajectory at or before the $m$th step. 

In this section, we prove the following lemma.
\begin{lem}\label{phaseoneevent}
$$\mathbb{P}\left(\mathcal{B}_M\right)\leq N^{-2\sqrt{\log N}}.$$
\end{lem} 
Observe that Lemmas~\ref{Yijrough},~\ref{Xbound} and~\ref{Wbound} all follow immediately from Lemma~\ref{phaseoneevent}. By Lemma~\ref{QfromOthers}, when $\omega \notin \mathcal{B}_M$, we obtain
\begin{equation}
\label{Qisok}
Q(m) \in (1 \pm 4\epsilon(t_m))\gamma(t_m) \cdot N.
\end{equation} 
so  Lemma~\ref{Qrough} is implied as well.

Now, given a point $\omega\in \mathcal{B}_M$, let 
\begin{equation}
\label{jcal}
\mathcal{J} = \mathcal{J}(\omega):= \min \{i: \omega\in \mathcal{B}_{i}\}.
\end{equation}
That is, $\omega$ corresponds to a trajectory of the process in which at least one of the variables strays far from its expectation at step $\mathcal{J}$ but not before. Define  $\mathcal{Z}$ to be the set of all $\omega\in \mathcal{B}_M$ such that (\ref{Istate}) is violated at time $\mathcal{J}(\omega)$ for some set $S\subseteq V(\mathcal{H})$. Similarly, define $\mathcal{Y}$, $\mathcal{X}$ and $\mathcal{W}$ to be the events that (\ref{Ystate}), (\ref{Xstate}) and (\ref{Wstate}) are respectively violated at time $\mathcal{J}(\omega)$. By definition,
\[\mathcal{B}_M = \mathcal{W}\cup\mathcal{X}\cup \mathcal{Y} \cup \mathcal{Z}.\]
Our goal is to show that the probability of each of the events $\mathcal{W}$, $\mathcal{X}$, $\mathcal{Y}$ and $\mathcal{Z}$ is small, from which Lemma~\ref{phaseoneevent} will follow.

Getting a sufficient bound on $\mathbb{P}(\mathcal{Z})$ follows directly from Proposition~\ref{Ibound}. We obtain the following.

\begin{lem}
\label{Zprop}
$$\mathbb{P}(\mathcal{Z}) \le N^{-5\sqrt{\log N}}.$$
\end{lem}

\begin{proof}
The event $\mathcal{Z}$ is contained within the event that there exists some $0 \le \ell \le M$, such that the bound
$$I(\ell) = O\left(\log N \cdot Nd^{-1/(r-1)}\right)$$
fails to hold. By Proposition~\ref{Ibound}, the result follows.
\end{proof}

We devote the rest of the section to proving that each of $\mathcal{W}$, $\mathcal{X}$ and $\mathcal{Y}$ occurs with probability at most $N^{-\Omega(\sqrt{\log N})}$. To prove this, we will apply the differential equations method and Theorem~\ref{Freed}. The sequences of variables that we track are not themselves supermartingales or submartingales and so we cannot apply Theorem~\ref{Freed} to them directly. What we do is show that the difference between each variable in the sequence and its expected trajectory, plus or minus some appropriate (growing) error function, is bounded above by an $\eta$-bounded supermartingale and below by an $\eta$-bounded submartingale (actually we only need to bound $W^i_S(m)$ and $X_S(m)$ from above). As in many applications of the differential equations method, the trick to verifying that these sequences are indeed $\eta$-bounded sub- or supermartingales is to define them in such a way that, if none of our sequences have strayed far from their expected trajectory, then we can use the fact that they have not strayed to prove that the properties hold and, otherwise, the properties hold for trivial reasons. 

As in Section~\ref{sec:timeZero}, despite the $Y$ configurations being the most important, we first consider the $W$ configurations, then the $X$ configurations, then the $Y$. This is because the proofs of their respective lemmas increase in complexity and it is helpful for the reader to first see a more simple application of Theorem~\ref{Freed}, before diving into the proof of Lemma~\ref{Yprop}. In the proof of Lemma~\ref{Yprop}, we need to be more careful than for Lemmas~\ref{Xprop} and~\ref{Wprop}. This is because we are proving that the $Y$ configurations are tightly concentrated, whereas we just prove weak upper bounds on the $X$ and $W$ configurations. 

\subsection{Tracking the \texorpdfstring{$\boldsymbol{W}$}{W} Configurations}

Roughly speaking, our goal is to determine the probability that (\ref{Wstate}) is the first bound to be violated. More precisely, to determine the probability that (\ref{Wstate}) is violated at the first time any of (\ref{Ystate}), (\ref{Xstate}), (\ref{Wstate}) and (\ref{Istate}) are violated. We will prove the following.

\begin{lem}
\label{Wprop}
$$\mathbb{P}(\mathcal{W}) \le N^{-4\sqrt{\log N}}.$$
\end{lem}

When $i=r$, (\ref{Wstate}) cannot be violated, so we assume $i<r$. For $0 \le i \le r-1$ and $S \subseteq V(\mathcal{H})$ with $|S|= i$, define $\mathcal{W}(S)$ to be the set of all $\omega \in \mathcal{B}_M$ such that the bound
\begin{equation}\label{WSi}
W_S^i(m) \le \log^{r^3(r-i)}(d)
\end{equation}
is violated at time $m = \mathcal{J}(\omega)$, where $\mathcal{J}(\omega)$ is defined in \eqref{jcal}. Observe that the bound \eqref{WSi} is precisely the bound (\ref{Wstate}) for our fixed choices of $i$ and $S$. It follows that
 $$\mathcal{W} = \bigcup_{\substack{S \subseteq V(\mathcal{H})\\ 1 \le |S| \le r-1}}\mathcal{W}(S).$$

Therefore the following proposition will imply the lemma, via an application of the union bound over all choices of $i$ and $S$.
\begin{prop}\label{Wpropp}
For all $0 \le i \le r-1$ and $S \subseteq V(\mathcal{H})$ with $|S|= i$,
$$\mathbb{P}(\mathcal{W}(S)) \le N^{-6\sqrt{\log N}}.$$
\end{prop}

\begin{proof}
For $\ell \ge 0$ define $\mathcal{E}_{\ell}$ to be the event that $W_S^i(\ell+1)\geq W_S^i(\ell)$. We remark that it \emph{is} possible for $W_S^{i}(\ell)$ to decrease in a step. This will happen if a copy of $W^{i}$ rooted at $S$ is an open hyperedge which is successfully sampled. However, our choice of martingale will reflect the fact that we are only concerned with proving an upper bound on $W_S^i(\ell)$. 

Given an event $E$, we let $\mathbbm{1}_E$ denote the indicator function of $E$. For $\ell\geq0$, define
$$a_{\ell}:= \frac{\alpha y_{r-2,1}(t_\ell)\left(1+ \epsilon(t_\ell)\right) + \alpha\log^{-K/2}(d)}{N\gamma(t_\ell)(1- 4\epsilon(t_\ell))}$$
and
\[A_S(\ell):=\begin{cases} \mathbbm{1}_{\mathcal{E}_{\ell}} \left(W_S^i(\ell+1)-W_S^i(\ell)\right) -  a_{\ell}& \text{if }\omega\notin \mathcal{B}_{\ell},\\
									0	& \text{otherwise},\end{cases}\]
where $y_{r-2,1}(t)$ is defined in \eqref{yijdef} and $\gamma(t)$ is defined in \eqref{ft}.
Also, set
\[B_S(m):=\sum_{\ell=0}^{m-1}A_S(\ell).\]
Note that, by definition, $B_S(0)=0$. Also, if $\omega\notin \mathcal{B}_{m-1}$, then
\begin{equation}
\label{Wsplit}
W_S^i(m) \leq B_S(m) + W_S^i(0) + \sum_{\ell=0}^{m-1}a_{\ell}.
\end{equation}
Therefore, to obtain an upper bound on $W_S^i(m)$, it suffices to bound the three quantities on the right side of this expression. Since $\epsilon(t_\ell)=o(1)$ for all $\ell\leq M$ and $\gamma(t)$ is bounded away from zero (by Remark~\ref{awayfromzero}) the sum can be bounded above in the following way:
\begin{align*}
\sum_{\ell=0}^{m-1}a_{\ell} &\leq \sum_{\ell=0}^{m-1}\frac{2\alpha y_{r-2,1}(t_\ell)}{N\gamma(t_\ell)} \\
&= 2\alpha(r-1)\left(\frac{1}{N}\sum_{\ell=0}^{m-1}\left(c+\alpha t_\ell \right)^{r-2}\right)\\
&= o\left(\log^{r^3(r-i)}(d)\right)
\end{align*}
as $m\leq M \leq \frac{N\log(N)}{\alpha}=O\left(N\log(d)\right)$ by \eqref{Nisd}. Still assuming $\omega \notin \mathcal{B}_{m-1}$, by the above analysis and \eqref{Wsplit}, if $W_S^i(0) \le \log^{2r}(d)=o\left(\log^{r^3(r-i)}(d)\right)$ we have
$$W_S^i(m) \le B_s(m) +  o\left(\log^{r^3(r-i)}(d)\right).$$
It follows that the event $\mathcal{W}(S)$ is contained within the event that either $W_S^i(0)> \log^{2r}(d)$ or that $B_S(m) > \frac{1}{2}\log^{r^3(r-i)}(d)$ for some $0 \le m \le M$.
 
By Lemma~\ref{timeZeroLemmaW}, 
$$\mathbb{P}\left(W^i_S(0) > \log^{2r}(d)\right) \le N^{-10\sqrt{\log N}},$$
so to prove Proposition~\ref{Wpropp} it suffices to show that $B_S(m)$ is unlikely to be large. We will show that 
\begin{equation}\label{BWsmall}
\mathbb{P}\left(B_S(m) > \frac{1}{2}\log^{r^3(r-i)}(d) \text{ for some } 0 \le m \le M\right)\le  N^{-6\sqrt{\log N}}.
\end{equation}
We wish to apply Theorem~\ref{Freed} to the sequence $B_S(0),\dots, B_S(M)$. In order to do this, we must show that $B_S(0),\dots, B_S(M)$ is an $\eta$-bounded supermartingale and we must also  bound the sum $\sum_{\ell=0}^{m-1}\Var(A_s(\ell)\given \mathcal{F}_{\ell})$.

\begin{claim}\label{Wmart}
$B_S(0),\dots, B_S(M)$ is a supermartingale.
\end{claim}
\begin{proof}
This is equivalent to showing that, for $0\leq \ell\leq M-1$, the expectation of $A_S(\ell)$ given $\mathcal{F}_\ell$ is non-positive. For $\omega\in\mathcal{B}_\ell$ we have $A_S(\ell)=0$, and so it suffices to consider $\omega\notin\mathcal{B}_\ell$. That is, we can assume that none of the variables that we track has strayed at or before time $\ell$. The only hyperedges $e$ which can be counted by $W_S^i(\ell+1)$ but not by $W_S^i(\ell)$ are those which contain $S$ and have the property that there is a unique vertex $x\in e\setminus S$ such that $x\notin I(\ell)$. Also, such a hyperedge $e$ contributes to $W_S^i(\ell+1)- W_S^i(\ell)$ if and only if an open hyperedge $e^*\not= e$ containing $x$ is successfully sampled at the $(\ell+1)$th step. Let $\mathcal{T}$ be the set of all such pairs $(e,e^*)$. As the probability that a particular open hyperedge is successfully sampled is $\frac{q}{Q(\ell)}$ and as $\omega \notin \mathcal{B}_{\ell}$, using \eqref{Qisok} we have
\begin{align}\label{squi}
\mathbb{E}\left[\mathbbm{1}_{\mathcal{E}_{\ell}} \left(W_S^i(\ell+1)-W_S^i(\ell)\right)\given[\Big] \mathcal{F}_{\ell}\right] &\le \sum_{(e,e^*) \in \mathcal{T}}\frac{q}{Q(\ell)}\nonumber \\ \le \frac{\alpha|\mathcal{T}|}{N  \gamma(t_\ell)(1- 4\epsilon(t_\ell))d^{1/(r-1)}},
\end{align}
 where $\gamma(t)$ is defined in \eqref{ft}. Therefore it suffices to bound $|\mathcal{T}|$. 

If $i\geq2$, then $e$ is a copy of a secondary configuration $(\mathcal{F}, R, D)$ where $\mathcal{F}$ is a single hyperedge, $|R|= i$ and $|D|= r-i-1$. Then since $\omega\notin\mathcal{B}_\ell$ we can use (\ref{Xstate}) to bound the number of such $e$ and (\ref{Wstate}) to bound the number of choices for $e^*$ to get 
\begin{equation}
\label{squi2}
|\mathcal{T}| \le \log^{2r^4(r-1)}(d) \cdot d^{\frac{1}{r-1}}\log^{-3K/5}(d)\cdot \log^{r^3(r-1)}(d) \leq d^{\frac{1}{r-1}}\log^{-K/2}(d),
\end{equation}
for $K$ chosen large with respect to $r$. 

Now, suppose that $i=1$ and let $v$ be the unique element of $S$. The number of such pairs $(e,e^*)$ with $|e\cap e^*|=1$ is precisely $Y_v^{r-2,1}(\ell)$ which, since $\omega\notin\mathcal{B}_\ell$, is at most $(1+\epsilon(t_\ell))y_{r-2,1}(t_\ell)d^{1/(r-1)}$. When $|e\cap e^*|\geq2$, we have that $e \cup e^*$ is a copy of a secondary configuration with a single neutral vertex. So, as $\omega\notin \mathcal{B}_\ell$, by (\ref{Xstate}) the number of pairs $(e,e^*)$ with $|e\cap e^*|\geq2$ is at most $d^{\frac{1}{r-1}}\log^{-K/2}(d)$, as above. Therefore when $i=1$, we have
\begin{equation}
\label{squi3}
|\mathcal{T}| \le (1+\epsilon(t_\ell))y_{r-2,1}(t_\ell)d^{1/(r-1)} + d^{\frac{1}{r-1}}\log^{-K/2}(d).
\end{equation}

Putting together \eqref{squi}, \eqref{squi2} and \eqref{squi3} gives 
\begin{equation}\label{Wexp}
\mathbb{E}\left[\mathbbm{1}_{\mathcal{E}_{\ell}} \left(W_S^i(\ell+1)-W_S^i(\ell)\right)\given[\Big] \mathcal{F}_{\ell}\right] \le a_{\ell},
\end{equation}
which implies that the expectation of $A_S(\ell)$ given $\mathcal{F}_\ell$ is non-positive, as desired. This completes the proof of the claim.
\end{proof}
\begin{claim}
\label{Wetabound}
$B_S(0),\ldots, B_S(M)$ is $\eta$-bounded for 
\[\eta:=\log^{r^3(r-i-1)}(d).\]
\end{claim}
\begin{proof}
First we bound the maximum value of $|A_S(\ell)|$. Again, we can assume that $\omega\notin \mathcal{B}_\ell$ as, otherwise, $|A_S(\ell)|$ is simply equal to zero. By definition of $A_S(\ell)$, the minimum possible value of $A_S(\ell)$ is 
\[-a_{\ell}= -o(1).\]
 Now we bound the maximum possible value of $W_S^i(\ell+1)- W_S^i(\ell)$. The only way that this quantity can be positive is if some vertex, say $x$, becomes infected in the $(\ell+1)$th step. Given that $x$ becomes infected, the maximum value that $W_S^i(\ell+1)- W_S^i(\ell)$ can achieve is precisely $W_{S\cup\{x\}}^{i+1}(\ell)$. This is at most $\log^{r^3(r-i-1)}(d)$ by (\ref{Wstate}) since $\omega\notin\mathcal{B}_\ell$. So $|A_S(\ell)|\leq \eta$ for $0\leq \ell\leq M$ and $B_S(0),\ldots, B_S(M)$ is $\eta$-bounded, as required. 
\end{proof}

\begin{claim}
\label{varW}
\[\sum_{\ell=0}^{m-1}\Var\left(A_S(\ell)\mid\mathcal{F}_\ell\right)  =O\left(\log^{r^3(r-i-1)+(r-1)}(d)\right)\]
\end{claim}
\begin{proof}
When $\omega \in \mathcal{B}_{\ell}$, we have that $\Var\left(A_S(\ell)\mid\mathcal{F}_\ell\right) = 0$. So now consider when $\omega \notin \mathcal{B}_{\ell}$. Since for a constant $c$ and any random variable $X$ we have $\Var(X - c) = \Var(X) \leq \mathbb{E}\left(X^2\right)$, by definition of $A_s(\ell)$ we have 
$$ \Var(A_S(\ell)\mid \mathcal{F}_{\ell}) \leq  \mathbb{E}\left(\mathbbm{1}_{\mathcal{E}_{\ell}}\left(W_S^i(\ell+1)-W_S^i(\ell)\right)^2\mid  \mathcal{F}_\ell\right).$$

Now,
\begin{equation}\label{expwchange}
\mathbb{E}\left(\mathbbm{1}_{\mathcal{E}_{\ell}}\left(W_S^i(\ell+1)-W_S^i(\ell)\right)^2\mid\mathcal{F}_\ell\right) =\sum_{k=1}^\infty k^2\mathbb{P}\left(W_S^i(\ell+1)-W_S^i(\ell)=k\mid\mathcal{F}_\ell\right).
\end{equation}
By Claim~\ref{Wetabound}, $W_S^i(\ell+1)- W_S^i(\ell)\leq \eta$ for $\omega\notin\mathcal{B}_\ell$. So, the right hand side of \eqref{expwchange} can be rewritten as 
\[\sum_{k=1}^\eta k^2\mathbb{P}\left(W_S^i(\ell+1)-W_S^i(\ell)=k\mid\mathcal{F}_\ell\right).\]
The sum 
\[\sum_{k=1}^\eta k\mathbb{P}\left(W_S^i(\ell+1)-W_S^i(\ell)=k\mid\mathcal{F}_\ell\right)\]
is precisely the expected value of $\mathbbm{1}_{\mathcal{E}_{\ell}}\left(W_S^i(\ell+1)-W_S^i(\ell)\right)$ given $\mathcal{F}_\ell$. So by \eqref{Wexp}, definition of $\eta$, $\gamma(t)$ and $y_{r-2,1}(t)$, and the fact that $t=O\left(\log(d)\right)$ we have, 
\[\sum_{k=1}^\eta k^2\mathbb{P}\left(W_S^i(\ell+1)-W_S^i(\ell)=k\mid\mathcal{F}_\ell\right) \leq \eta\cdot a_{\ell} \]
\[=O\left(\frac{\log^{r^3(r-i-1)}(d)y_{r-2,1}(t_\ell)}{N\gamma(t_{\ell})}\right)=O\left(\frac{\log^{r^3(r-i-1)}(d)\left(c+\alpha t_\ell\right)^{r-2}}{N}\right)\]
\[=O\left(\frac{\log^{r^3(r-i-1)}(d)\log^{r-2}(d)}{N}\right).\]
So by the above analysis, 
\[\sum_{\ell=0}^{m-1}\Var\left(A_S(\ell)\mid\mathcal{F}_\ell\right) =O\left( \frac{M\log^{r^3(r-i-1)+(r-2)}(d)}{N}\right) =O\left(\log^{r^3(r-i-1)+(r-1)}(d)\right),\]
since $M=O\left(N\log(d)\right)$. This completes the proof of the claim.
\end{proof}

Set $\nu:=\log^{r^3(r-i-1)+r}(d)$ and $a:=\frac{1}{2}\log^{r^3(r-i)}(d)$. Using Claims~\ref{Wmart},~\ref{Wetabound} and~\ref{varW}, we can apply Theorem~\ref{Freed} to show that
\[\mathbb{P}\left(B_S(m)\geq a \text{ for some } 0 \le m \le M\right)\leq\exp\left(-\frac{a^2}{2(\nu+a\eta)}\right)\ll N^{-6\sqrt{\log(N)}},\]
as required for \eqref{BWsmall}. This completes the proof of Proposition~\ref{Wpropp}.
\end{proof}
Hence Lemma~\ref{Wprop} is proved.

\subsection{Tracking the \texorpdfstring{$\boldsymbol{X}$}{X} Configurations}
Now we determine the probability that (\ref{Xstate}) is violated at the first time any of (\ref{Ystate}), (\ref{Xstate}), (\ref{Wstate}) and (\ref{Istate}) are violated. We will prove the following.

\begin{lem}
\label{Xprop}
$$\mathbb{P}(\mathcal{X}) \le N^{-4\sqrt{\log N}}.$$
\end{lem}

For a secondary configuration $X= (\mathcal{F},R,D)$ and $S \subseteq V(\mathcal{H})$ with $|S|= |R|$, define $\mathcal{X}(X,S)$ to be the set of all $\omega \in \mathcal{B}_M$ such that the bound
\begin{equation}\label{XSi}
X_S(m) \le \log^{2|D|r^4}(d)\cdot d^{\frac{|V(\mathcal{F})| - |R|-|D|}{r-1}}\log^{-3K/5}(d)
\end{equation}
is violated at time $m = \mathcal{J}(\omega)$, where $\mathcal{J}(\omega)$ is defined in \eqref{jcal}. Observe that the bound \eqref{XSi} is precisely the bound (\ref{Xstate}) for our fixed choices of $X$ and $S$. Therefore,
 $$\mathcal{X} = \bigcup_{X,S}\mathcal{X}(X,S).$$

So the following proposition will imply Lemma~\ref{Xprop}, via an application of the union bound over all choices of $X$ and $S$.
\begin{prop}\label{Xpropp}
For all secondary configurations $X = (\mathcal{F},R,D)$ and $S \subseteq V(\mathcal{H})$ such that $|S|= |R|$,
$$\mathbb{P}(\mathcal{X}(X,S)) \le N^{-6\sqrt{\log N}}.$$
\end{prop}

\begin{proof}
Define $\mathcal{E}_{\ell}$ to be the event that $X_S(\ell+1)\geq X_S(\ell)$. For $\ell\geq0$, define
$$a_{\ell}:= \frac{\log^{r^4(2|D|-1)}(d)\cdot d^{\frac{|V(\mathcal{F})|-|R|-|D|}{r-1}}}{N\log^{3K/5}(d)}$$
and
\[A_S(\ell):=\begin{cases}\mathbbm{1}_{\mathcal{E}_{\ell}}\left(X_S(\ell+1)-X_S(\ell)\right) - a_{\ell}, &\text{ if } \omega \notin \mathcal{B}_{\ell},\\
0 &\text{otherwise}\end{cases}\]
and let
\[B_S(m):=\sum_{\ell=0}^{m-1}A_S(\ell).\]
If $\omega\notin \mathcal{B}_{m-1}$, then
\[X_S(m)\leq B_S(m)+X_S(0) + m\cdot a_{\ell}.\]
So if $X_S(0) \le d^{\frac{|V(\mathcal{F})|-|R|-|D|}{r-1}} \log^{-4K/5}(d) = o\left(\log^{2r^4|D|}(d)\cdot d^{\frac{|V(\mathcal{F})|-|R|-|D|}{r-1}} \log^{-3K/5}(d)\right)$, then
\begin{equation}
\label{Xterms}
X_S(m) \le B_S(m) + o\left(\log^{2r^4|D|}(d)\cdot d^{\frac{|V(\mathcal{F})|-|R|-|D|}{r-1}} \log^{-3K/5}(d)\right)
\end{equation}
by the fact that $m=O\left(N\log(d)\right)$. 

It follows that the event $\mathcal{X}(X,S)$ is contained within the event that either $$X_S(0) > d^{\frac{|V(\mathcal{F})|-|R|-|D|}{r-1}} \log^{-4K/5}(d)$$ or $$B_S(m) > \frac{1}{2}\log^{2|D|r^4}(d)\cdot d^{\frac{|V(\mathcal{F})| - |R|-|D|}{r-1}}\log^{-3K/5}(d)$$ for some $0 \le m \le M$.
 
By Lemma~\ref{timeZeroLemmaUseful},
$$\mathbb{P}\left(X_S(0) > d^{\frac{|V(\mathcal{F})|-|R|-|D|}{r-1}} \log^{-4K/5}(d)\right) \le N^{-10\sqrt{\log N}},$$
so to prove Proposition~\ref{Wpropp} it suffices to show that $B_S(m)$ is unlikely to be large. We will show that
\begin{equation}
\label{XBbound}
\mathbb{P}\left(B_S(m)> \frac{1}{2}\log^{2|D|r^4 -3K/5}(d)\cdot d^{\frac{|V(\mathcal{F})| - |R|-|D|}{r-1}} \text{ for some } 0 \le m \le M \right) \le N^{-6\sqrt{\log N}}.
\end{equation}

We will apply Theorem~\ref{Freed}. In order to apply Theorem~\ref{Freed} we must show that the sequence $B_S(0),\dots, B_S(M)$ is an $\eta$-bounded supermartingale and we also need to bound the sum $\sum_{\ell=0}^{m-1}\Var(A_s(\ell)\given \mathcal{F}_{\ell})$.

\begin{claim}\label{Xmart}
$B_S(0),\dots, B_S(M)$ is a supermartingale.
\end{claim}

\begin{proof}
This is equivalent to showing that, for $0\leq \ell\leq M-1$, the expectation of $A_S(\ell)$ given $\mathcal{F}_\ell$ is non-positive. For $\omega\in\mathcal{B}_\ell$ we have $A_S(\ell)=0$, and so it suffices to consider $\omega\notin\mathcal{B}_\ell$. For each $u\in D$, let $X^u$ denote the configuration $(\mathcal{F},R,D\setminus\{u\})$. By Remark~\ref{rootTransfer}, $X^u$ is a secondary configuration. Every element of $X_S(\ell+1)\setminus X_S(\ell)$ comes from an element of $X_S^u(\ell)$, for some $u\in D$, and an open hyperedge $e$ containing the image of $u$ such that $e$ is successfully sampled. As $\omega\notin\mathcal{B}_\ell$, using \eqref{Qisok} and the fact that $\gamma(t_{\ell})$ is bounded below by a function of $r,c$ and $\alpha$ (see Remark~\ref{awayfromzero}) gives that the probability that any particular open hyperedge is successfully sampled is
\[\frac{q}{Q(\ell)} \leq \frac{\alpha}{\gamma(t_{\ell})(1-4\epsilon(t_\ell))d^{1/(r-1)}N}=O\left(\frac{1}{d^{1/(r-1)}N}\right).\]
Also, since $\omega\notin\mathcal{B}_\ell$, we can apply (\ref{Xstate}) to $X_S^u(\ell)$ and (\ref{Wstate}) to the number of open hyperedges containing the image of $u$ to get that the expected number of copies of $X$ created in the $(\ell+1)$th step is
$$O\left(\frac{1}{d^{1/(r-1)}N} \cdot \frac{\log^{2r^4(|D|-1)}(d) d^{\frac{|V(\mathcal{F})| - |R| - |D| + 1}{r-1}}}{\log^{3K/5}(d)} \cdot \log^{r^3(r-1)}(d)\right)$$
\begin{equation}\label{expvX}
=O\left(a_{\ell} \cdot \frac{ \log^{r^3(r-1)}(d)}{\log^{r^4}(d)}\right) = o\left(a_{\ell}\right)
\end{equation}
and therefore the expectation of $A_S(\ell)$ given $\mathcal{F}_\ell$ is negative. This proves that the sequence $B_S(0),\dots,B_S(M)$ is a supermartingale. 
\end{proof}
\begin{claim}
\label{Xetabound}
$B_S(0),\ldots, B_S(M)$ is $\eta$-bounded for 
\[\eta:=\log^{r^4(2|D|-1)  + 1}(d)\cdot d^{\frac{|V(\mathcal{F})| - |R|-|D|}{r-1}}\log^{-3K/5}(d).\]
\end{claim}
\begin{proof}
We first bound the maximum value of $|A_S(\ell)|$. Again, assume $\omega\notin\mathcal{B}_\ell$. By \eqref{Nnotsmall} and the definition of $A_S(\ell)$, we have
\[A_S(\ell)\geq - a_{\ell} =  o\left(d^{\frac{|V(\mathcal{F})|-|R|-|D|-1}{r-1}}\right).\]
 For $u\in D$, define $\tilde{X}^u$ to be the configuration $(\mathcal{F},R\cup\{u\},D\setminus\{u\})$. By Remark~\ref{rootTransfer} this configuration is secondary. The value of $A_S(\ell)$ can only be positive if some vertex, say $x$, becomes infected in the $(\ell+1)$th step. Given that $x$ becomes infected, the number of copies of $X$ rooted at $S$ created is at most $\sum_{u\in D}\tilde{X}^u_{S\cup\{x\}}(\ell)$.  Since $\omega\notin\mathcal{B}_\ell$, by (\ref{Xstate}) this is
\[O\left(\log^{2r^4(|D|-1)}(d)\cdot d^{\frac{|V(\mathcal{F})| - |R|-|D|}{r-1}}\log^{-3K/5}(d)\right).\]
 So we have $|A_S(\ell)|\leq \eta$ for $0\leq \ell\leq M$, as required. 
\end{proof}
\begin{claim}
\label{varX}
\[\sum_{\ell=0}^{m-1}\Var\left(A_S(\ell)\mid\mathcal{F}_\ell\right)  \le \frac{\log^{2r^4(2|D|-1) + 1}(d)\cdot d^{\frac{2(|V(\mathcal{F})|-|R|-|D|)}{r-1}}}{\log^{6K/5}(d)}.\]
\end{claim}
\begin{proof}
When $\omega \in \mathcal{B}_{\ell}$, we have that $\Var\left(A_S(\ell)\mid\mathcal{F}_\ell\right) = 0$. So now consider when $\omega \notin \mathcal{B}_{\ell}$. We have
$$\Var(A_S(\ell)\mid\mathcal{F}_\ell)\leq \mathbb{E}(A_S(\ell)^2\mid\mathcal{F}_\ell)\le \mathbb{E}\left(\mathbbm{1}_{\mathcal{E}_{\ell}}\left(X_S(\ell+1)-X_S(\ell)\right)^2\mid\mathcal{F}_\ell\right).$$

As in the proof of Claim~\ref{varW}, we have
\[\mathbb{E}\left(\mathbbm{1}_{\mathcal{E}_{\ell}}\left(X_S(\ell+1)-X_S(\ell)\right)^2\mid\mathcal{F}_\ell\right) = \sum_{k=1}^\eta k^2\mathbb{P}\left(X_S(\ell+1)-X_S(\ell)=k\mid\mathcal{F}_\ell\right),\]
and
\[\mathbb{E}\left(\mathbbm{1}_{\mathcal{E}_{\ell}}\left(X_S(\ell+1)-X_S(\ell)\right)\mid\mathcal{F}_\ell\right) =  \sum_{k=1}^\eta k\mathbb{P}\left(X_S(\ell+1)-X_S(\ell)=k\mid\mathcal{F}_\ell\right).\]
 So, by \eqref{expvX} and definition of $\eta$, we get that $\Var(A_S(\ell)\mid\mathcal{F}_\ell)$ is at most
$$o\left(\eta \cdot a_{\ell} \right) =o\left(\frac{\log^{2r^4(2|D|-1)}(d)\cdot d^{\frac{2(|V(\mathcal{F})|-|R|-|D|)}{r-1}}}{N\log^{6K/5}(d)}\right).$$
By the above analysis, $\sum_{\ell=0}^{m-1}\Var\left(A_S(\ell)\mid\mathcal{F}_\ell\right) $ is therefore at most
\begin{equation}\label{nuX}
\nu:=\frac{\log^{2r^4(2|D|-1) + 1}(d)\cdot d^{\frac{2(|V(\mathcal{F})|-|R|-|D|)}{r-1}}}{\log^{6K/5}(d)},
\end{equation}
as $M = O\left(N\log(d)\right)$.
\end{proof}
Set $\nu$ as in \eqref{nuX} and $a:=\frac{1}{2}\log^{2|D|r^4}(d)\cdot d^{\frac{|V(\mathcal{F})| - |R|-|D|}{r-1}}\log^{-3K/5}(d)$. Applying Theorem~\ref{Freed} shows that
\[\mathbb{P}\left(B_S(m)\geq a \text{ for some } 0 \le m \le M\right)\leq \exp\left(-\frac{a^2}{2(\nu+a\eta)}\right)\ll N^{-6\sqrt{\log N}},\]
as required for \eqref{XBbound}. This completes the proof of Proposition~\ref{Xpropp}.
\end{proof}
Therefore, the proof of Lemma~\ref{Xprop} is concluded.

\subsection{Tracking the \texorpdfstring{$\boldsymbol{Y}$}{Y} Configurations}\label{yijsec}

Much of the analysis in this subsection is similar to the previous two; however, we must be more careful since we are aiming at tight concentration bounds (not just crude upper bounds). As we require both upper and lower bounds, we will be dealing with both a supermartingale and a submartingale.

We wish to bound the probability that (\ref{Ystate}) is violated at the first time any of (\ref{Ystate}), (\ref{Xstate}), (\ref{Wstate}) and (\ref{Istate}) are violated. We will prove the following.

\begin{lem}
\label{Yprop}
$$\mathbb{P}(\mathcal{Y}) \le \frac{1}{2}N^{-2\sqrt{\log N}}.$$
\end{lem}

For $0 \le i \le r-2$, $0 \le j \le r-1-i$ and $v \in V(\mathcal{H})$, define $\mathcal{Y}(i,j,v)$ to be the set of all $\omega \in \mathcal{B}_M$ such that the bound
\begin{equation}\label{YSi}
Y_v^{i,j}(m)  \in \left(1 \pm \epsilon(t)\right)d^{1 - \frac{i}{r-1}}y_{i,j}(t_m)
\end{equation}
is violated at time $m = \mathcal{J}(\omega)$, where $\mathcal{J}(\omega)$ is defined in \eqref{jcal} and $y_{i,j}(t)$ is defined in \eqref{yijdef}. Observe that the bound \eqref{YSi} is precisely the bound (\ref{Ystate}) for our fixed choices of $i,j$ and $v$. 
Note that
 $$\mathcal{Y} = \bigcup_{i,j,v}\mathcal{Y}(i,j,v).$$
Thus the following proposition will imply the lemma, via an application of the union bound over all choices of $i,j $ and $v$.
\begin{prop}\label{Ypropp}
For $0 \le i \le r-2$, $0 \le j \le r-1-i$ and $v \in V(\mathcal{H})$,
$$\mathbb{P}(\mathcal{Y}(i,j,v)) \le N^{-3\sqrt{\log N}}.$$
\end{prop}
The proof of the proposition relies on two claims (Claim~\ref{Yterms} and Claim~\ref{Yclaim}), which will be stated where they are needed once the relevant variables have been defined. They will be proved later after completing the proof of the proposition assuming the claims. 
\begin{proof}[Proof of Proposition~\ref{Ypropp}]
For $\ell\geq0$, define
$$a^{\pm}_{\ell}:= d^{1-\frac{i}{r-1}}\left(\frac{y_{i,j}'(t_\ell)}{N} \mp \frac{20\cdot \left|y_{i,j}'(t_\ell)\right|\cdot \epsilon(t_\ell)}{N}\right)  $$
and
\[A_v^\pm(\ell):=\begin{cases}\left(Y_v^{i,j}(\ell+1)-Y_v^{i,j}(\ell)\right) - a^{\pm}_{\ell}& \text{if }\omega\notin\mathcal{B}_\ell,\\
0 &\text{otherwise},\end{cases}\]
where $y_{i,j}(t)$ is defined in \eqref{yijdef}.
We clarify that we are defining two different random variables $A_v^+(\ell)$ and $A_v^-(\ell)$. The superscript denotes whether we expect the variable to be typically positive or typically negative. Given this definition, we define the following pair of random variables:
\[B_v^{\pm}(m):=\sum_{\ell=0}^{m-1}A_v^\pm(\ell).\]

Thus if $\omega \notin \mathcal{B}_{m-1}$, by definition,
\begin{equation}\label{YandB}
Y_v^{i,j}(m) = B_v^{\pm}(m) + Y_v^{i,j}(0) +d^{1-\frac{i}{r-1}}\left(\sum_{\ell=0}^{m-1}\frac{y_{i,j}'(t_\ell)}{N} \mp 20\sum_{\ell=0}^{m-1}\frac{\left|y_{i,j}'(t_\ell)\right|\cdot \epsilon(t_\ell)}{N}\right).
\end{equation}
Our strategy is to obtain a concentration result for $Y_v^{i,j}(m)$ by analysing each of the terms on the right side of this expression.

The proposition will follow from the next two claims.
\begin{claim}
\label{Yterms}
If $\omega \notin \mathcal{B}_{m-1}$ and
$$Y_v^{i,j}(0)\in \left(1\pm \log^{-3K/10}(d)\right)y_{i,j}(0)d^{1-\frac{i}{r-1}},$$
then the following bounds hold: 
\[Y_v^{i,j}(m) \leq  B_v^-(m) + d^{1-\frac{i}{r-1}}\left(1+ \frac{\epsilon(t_m)}{2}\right)y_{i,j}(t_m),\]
\[Y_v^{i,j}(m) \geq B_v^+(m) + d^{1-\frac{i}{r-1}}\left(1- \frac{\epsilon(t_m)}{2}\right)y_{i,j}(t_m).\]
\end{claim}

\begin{claim}
\label{Yclaim}
With probability at least $1 - N^{-4\sqrt{\log(N)}}$, both the following bounds hold for all $0 \le m \le M$: 
\[B_v^-(m)\le \frac{1}{2}\epsilon(t_m)d^{1 - \frac{i}{r-1}}y_{i,j}(t_m),\]
\[B_v^+(m)\ge -\frac{1}{2}\epsilon(t_m)d^{1 - \frac{i}{r-1}}y_{i,j}(t_m).\]
\end{claim}

Indeed, by Claim~\ref{Yterms}, the event $\mathcal{Y}(i,j,v)$ is contained within the event that either
\[
Y_v^{i,j}(0)\not\in \left(1\pm \log^{-3K/10}(d)\right)y_{i,j}(0)d^{1-\frac{i}{r-1}},
\]
or one of the following bounds hold:
\[
\begin{gathered}
B_v^-(m)> \frac{1}{2}\epsilon(t_m)d^{1 - \frac{i}{r-1}}y_{i,j}(t_m)\\
B_v^+(m)< -\frac{1}{2}\epsilon(t_m)d^{1 - \frac{i}{r-1}}y_{i,j}(t_m)
\end{gathered}
\]

By Lemma~\ref{YTIMEZERO},
$$\mathbb{P}\left(Y_v^{i,j}(0)\not\in \left(1\pm \log^{-3K/10}(d)\right)y_{i,j}(0)d^{1-\frac{i}{r-1}}\right) \le N^{-10\sqrt{\log N}}.$$
Combining this with Claim~\ref{Yclaim} completes the proof of the proposition.
\end{proof}

We now prove Claims~\ref{Yterms} and~\ref{Yclaim}.
\begin{proof}[Proof of Claim~\ref{Yterms}]
We will analyse each term in the right hand side of \eqref{YandB}. By Lemma~\ref{sumToInt} with $s(t) = y'_{i,j}(t)$, we have
\begin{equation}\label{y2}
\begin{gathered}
\sum_{\ell=0}^{m-1}\frac{y_{i,j}'(t_\ell)}{N} \in \int_0^{t_m}y'_{i,j}(t)dt \pm \frac{m\cdot\sup_{t\in [0,T]}\left|y_{i,j}''(t)\right|}{2N^2} \\
\in y_{i,j}(t_m) - y_{i,j}(0) \pm \frac{\log^{r^2}(d)}{N},
\end{gathered}
\end{equation}
as $m = O\left(N \log(d)\right)$.

Next, we analyse the second summation in the parentheses of \eqref{YandB}. First, by definition of $\epsilon(t)$ (given in \eqref{errorterm}), 
\begin{equation}
\label{ysumstuff}
20\sum_{\ell=0}^{m-1}\frac{\left|y_{i,j}'(t_\ell)\right|\cdot \epsilon(t_\ell)}{N} = \left(\frac{20}{\log^{K/5}(d)}\right) \sum_{\ell=0}^{m-1}\frac{\left|y_{i,j}'(t_\ell)\right|\cdot (t_\ell+1)^{K/10}}{N}.
\end{equation}
The function $y_{i,j}(t)$ is a polynomial with positive leading coefficient and, by Remark~\ref{awayfromzero}, $y_{i,j}(t)$ is bounded away from zero by a function of $r,c$ and $\alpha$ for all $t\in [0,T]$. Therefore, there exists a positive constant $C=C(r,c,\alpha)$ such that
\[\left|y_{i,j}'(t_\ell)\right|\cdot (t_\ell+1)^{K/10}\leq C\cdot y_{i,j}(t_\ell)\cdot (t_\ell+1)^{(K/10)-1}.\]
Combining this with \eqref{ysumstuff} gives
\[20\sum_{\ell=0}^{m-1}\frac{\left|y_{i,j}'(t_\ell)\right|\cdot \epsilon(t_\ell)}{N} \leq \left(\frac{20\cdot C}{\log^{K/5}(d)}\right) \sum_{\ell=0}^{m-1}\frac{y_{i,j}(t_\ell)\cdot (t_\ell+1)^{(K/10)-1}}{N}.\]
Now, applying Lemma~\ref{sumToInt} and choosing $K$ sufficiently large with respect to $r$,
\[\sum_{\ell=0}^{m-1}\frac{y_{i,j}(t_\ell)\cdot (t_\ell+1)^{(K/10)-1}}{N} \leq \int_0^{t_m}y_{i,j}(t)\cdot (t+1)^{(K/10)-1}dt + \frac{\log^{K/9}(d)}{N}.\]
Consider the integral on the right side of the above inequality. By definition of $y_{i,j}(t)$ (given in \eqref{yi0def} and \eqref{yijdef}) and definition of $\gamma(t)$ in \eqref{ft}, we have
$$\int_0^{t_m}y_{i,j}(t)\cdot (t+1)^{(K/10)-1}dt \le \int_0^{t_m}\binom{r-1}{i}\binom{r-1-i}{j}(t + 1)^{K/10 -1}(c + \alpha t)^{i + (r-1)j}dt.$$
Letting $C' := \binom{r-1}{i}\binom{r-1-i}{j}\cdot (\max \{1,c,\alpha\})^{r^3},$ gives that this is at most
\begin{align*}
 C'\int_0^{t_m}(t + 1)^{i + (r-1)j + K/10 -1}dt  &\le \frac{10C'}{K}(t_m+1)^{i + (r-1)j + K/10}\\ 
 & \le \tilde{C}\cdot y_{i,j}(t_m)\cdot (t_m+1)^{K/10},
 \end{align*}
for some positive constant $\tilde{C}$ depending on $r,c,\alpha$ and $K$. Moreover, by choosing $K$ large with respect to $r,c$ and $\alpha$, we may take $\tilde{C}$ arbitrarily close to zero. So, provided that $K$ is large enough, we have
\begin{align}\label{secondSum}
20\sum_{\ell=0}^{m-1}\frac{\left|y_{i,j}'(t_\ell)\right|\cdot \epsilon(t_\ell)}{N}&\leq \left(\frac{20\cdot C}{\log^{K/5}(d)}\right)\left(\tilde{C}\cdot y_{i,j}(t_m)\cdot (t_m+1)^{K/10} + \frac{\log^{K/9}(d)}{N}\right) \nonumber \\
& \le \frac{y_{i,j}(t_m)\epsilon(t_m)}{4},
\end{align}
as we may choose $\tilde{C}$ arbitrarily close to zero.
Now let us combine \eqref{YandB}, \eqref{y2} and \eqref{secondSum} with our hypothesis to give an upper bound for $Y_v^{i,j}(m)$:
\begin{align*}
Y_v^{i,j}(m) \le B_v^{-}(m) &+  \left(1+\log^{-3K/10}(d)\right)y_{i,j}(0)d^{1- \frac{i}{r-1}}\\ &+ d^{1 - \frac{i}{r-1}}\left(y_{i,j}(t_m) - y_{i,j}(0) +  \frac{\log^{r^2}(d)}{N} + \frac{y_{i,j}(t_m)\epsilon(t_m)}{4}\right).
\end{align*}
By the definition of $\epsilon(t)$ in \eqref{errorterm} we can absorb the error term $\log^{-3K/10}(d)y_{i,j}(0)d^{1- \frac{i}{r-1}}$ into the main error term. This gives
$$Y_v^{i,j}(m) \le B_v^{-}(m) + \left(1 + \frac{\epsilon(t_m)}{2}\right)d^{1 - \frac{i}{r-1}}y_{i,j}(t_m).$$
Similarly, we get
$$Y_v^{i,j}(m) \ge B_v^{+}(m) + \left(1 - \frac{\epsilon(t_m)}{2}\right)d^{1 - \frac{i}{r-1}}y_{i,j}(t_m),$$
as required.
\end{proof}

It remains to prove Claim~\ref{Yclaim}.
\begin{proof}[Proof of Claim~\ref{Yclaim}]
To prove these bounds, we show that the sequence $B_v^-(0),\dots,B_v^-(m)$ is a supermartingale and that $B_v^+(0),\dots,B_v^+(m)$ is a submartingale which satisfy certain maximum change and variance increment bounds and apply Theorem~\ref{Freed}. This amounts to bounding the expectation, maximum/minimum possible values and variance of $A_v^\pm(\ell)$. As in the previous two subsections, we will always assume $\omega\notin \mathcal{B}(\ell)$; otherwise, we have $A_v^\pm(\ell)=0$ and so all of the required bounds hold trivially. Recall the definition of $y_{i,j}(t)$ from \eqref{yi0def} and \eqref{yijdef} and the definition of $\gamma(t)$ from \eqref{ft}. The following expression, obtained by differentiating $y_{i,j}(t)$, will be useful in what follows:
\begin{align}\label{yijprime}y_{i,j}'(t) &= \binom{r-1}{i}\binom{r-1-i}{j}\left(\alpha \cdot i (c + \alpha t)^{i-1} \gamma(t)^j + (c + \alpha t)^i j\gamma'(t) \gamma(t)^{j-1}\right) \nonumber \\
&= \frac{(j+1)\alpha y_{i-1,j+1}(t)}{\gamma(t)} + \frac{(r-i-j)\alpha y_{i,j-1}(t)y_{r-2,1}(t)}{\gamma(t)} - \frac{jy_{i,j}(t)}{\gamma(t)}.\end{align}

\begin{subclaim}\label{Ymart}
 $B_v^+(0),\dots,B_v^+(M)$ is a submartingale and $B_v^-(0),\dots,B_v^-(M)$ is a supermartingale.
\end{subclaim}
\begin{proof}[Proof of Subclaim~\ref{Ymart}]
First, we bound the expectation of $Y_v^{i,j}(\ell+1)-Y_v^{i,j}(\ell)$ given $\mathcal{F}_\ell$. To do this we must consider copies of $Y^{i,j}$ rooted at $v$ that are created by successfully sampling an open hyperedge at the $(\ell + 1)$th step, and also copies that are destroyed. Let $C_{\ell + 1}$ be the number of copies of $Y^{i,j}$ rooted at $v$ which are present in the $(\ell+1)$th step but not in the $\ell$th step. Any such copy must come from either:
\begin{enumerate}[(1)]
\item  an element $\mathcal{F}'$ of $Y_v^{i-1,j}(\ell)\setminus Y_v^{i,j}(\ell)$  such that a healthy vertex of $\mathcal{F}'$, say $x$, contained only in the central hyperedge becomes infected in the $(\ell+1)$th step, or 
\item an element $\mathcal{F}'$ of $Y_v^{i,j-1}(\ell)$ and a copy $\mathcal{G}'$ of $Y^{r-2,0}$ in $\mathcal{H}(\ell)$ where:
\begin{itemize}
\item $\mathcal{G}'$ contains precisely $r-2$ infected vertices,
\item $\mathcal{G}'$ is rooted at some non-root $u \in \mathcal{F}'$, where $u$ is contained only in the central hyperedge of $\mathcal{F}'$,
\item $\mathcal{G}' \cap \mathcal{F}' = \{u\}$,
\item the unique healthy non-root $x$ of $\mathcal{G}'$ becomes infected at the $(\ell + 1)$th step.  
\end{itemize}
\end{enumerate}
See Figure~\ref{Ymake} for examples of (1) and Figure~\ref{Ymake2} for examples of (2) (both in the case $r=6$). See also Figure~\ref{Ymaking} for other examples of these ($u$ there plays the role of $x$ here). Both (1) and (2) will contribute to the main term in the expectation of $Y_v^{i,j}(\ell+1)-Y_v^{i,j}(\ell)$ given $\mathcal{F}_\ell$. 
\begin{figure}[hbtp]
\centering
\includegraphics[width=0.9\textwidth]{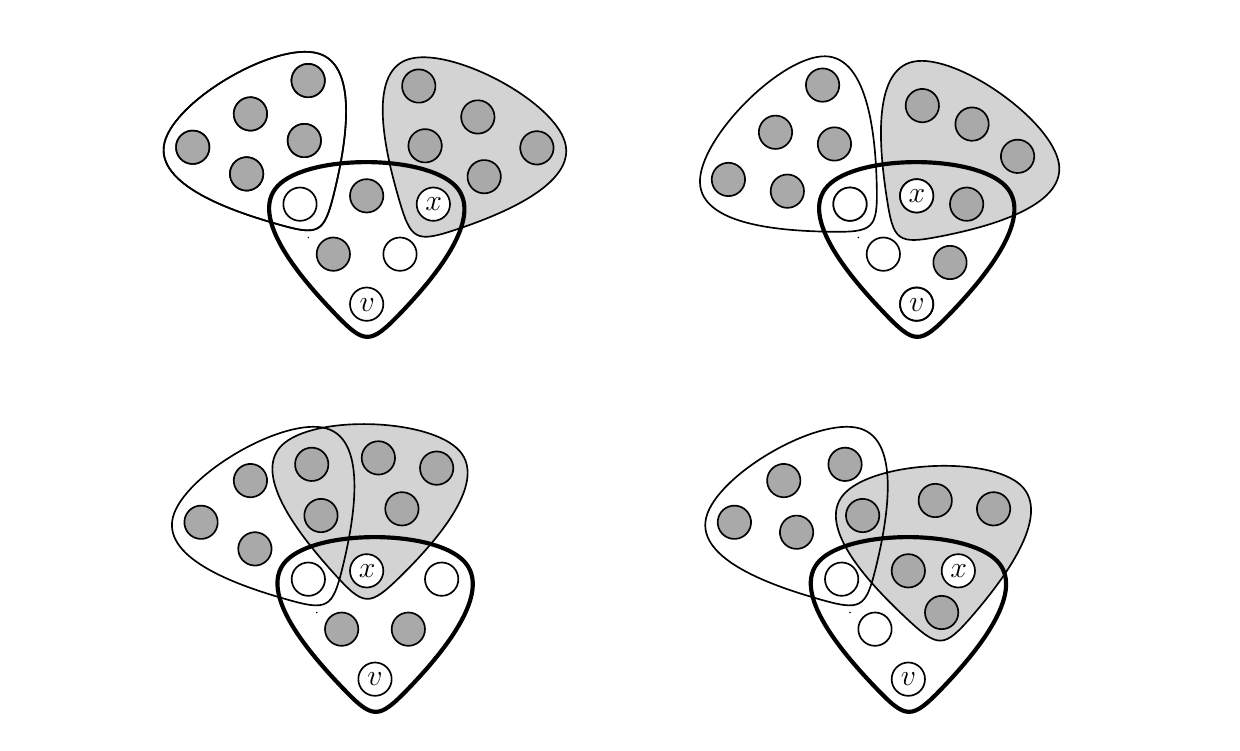}
\caption{Here $r=6$. Some examples of ways that a copy of $Y^{3,1}$ rooted at $v$ can be created via (1) when the hyperedge shaded light grey becomes sucessfully sampled and $x$ becomes infected. The central hyperedge of the copy of $Y^{3,1}$ created is drawn with a thick outline. Infected vertices are shaded dark grey and healthy vertices are unshaded.}
\label{Ymake}
\end{figure}
\begin{figure}[hbtp]
\centering
\includegraphics[width=0.9\textwidth]{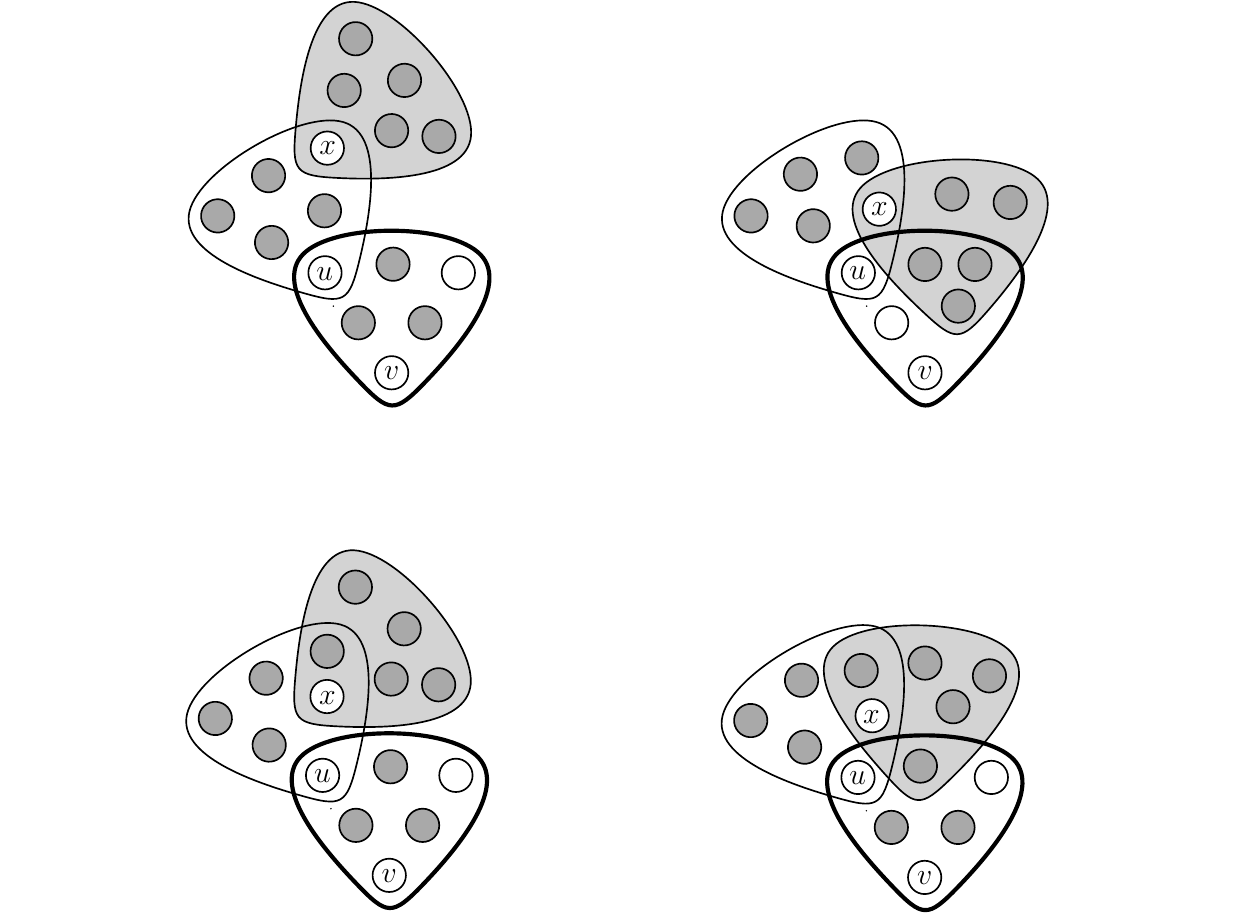}
\caption{Here $r=6$. Some examples of ways that a copy of $Y^{3,1}$ rooted at $v$ can be created via (2) when the hyperedge shaded light grey becomes successfully sampled and $x$ becomes infected. The  central hyperedge of the copy of $Y^{3,1}$ created is drawn with a thick outline. Infected vertices are shaded dark grey and healthy vertices are unshaded.}
\label{Ymake2}
\end{figure}

The only way that a vertex can become infected is if one of the open hyperedges containing that vertex is successfully sampled. The probability that any given open hyperedge is successfully sampled is 
\begin{equation}
\label{sampled}
\frac{q}{Q(\ell)} \in \frac{\alpha}{(1\mp 4\epsilon(t_\ell))\gamma(t_{\ell})Nd^{1/(r-1)}}.
\end{equation}
Let us count the number of ways that we can create a copy of $Y^{i,j}$ rooted at $v$ that is not present in the $\ell$th step using (1). We will break this into two cases. If $e^*$ is the hyperedge that is successfully sampled at the $(\ell + 1)$th step, then either $|e^* \cap \mathcal{F}'|= 1$, or $|e^* \cap \mathcal{F}'| > 1$.  Consider the first case.

We wish to count the number of ways to choose:
\begin{itemize}
\item an element $\mathcal{F}'$ of $Y_v^{i-1,j}(\ell) \setminus Y_v^{i,j}(\ell)$;
\item a healthy vertex $x$ contained only in the central hyperedge of $\mathcal{F}'$; and
\item  an open hyperedge $e^*$ containing $x$ such that $e^*$ intersects the copy of $Y^{i-1,j}$ only on $x$.
\end{itemize} 
Observe that, in this case, $e^* \cup \mathcal{F}'$ is a member of $Y_v^{i-1,j+1}(\ell)$ that contains exactly $i-1$ infected vertices in its central hyperedge. For each such member of $Y_v^{i-1,j+1}(\ell)$, there are $j+1$ possible choices of the hyperedge $e^*$.

Therefore the number of ways to create a member of $Y_v^{i,j}(\ell+1)$ via (1) when $|e^* \cap \mathcal{F}'|= 1$ is at most $(j+1)Y^{i-1,j+1}_v(\ell)$. However, this also unnecessarily counts $(j+1)$ times the number of copies of $Y^{i-1,j+1}$ rooted at $v$ that contain at least $i$ infected vertices in their central hyperedge. Therefore we need to subtract off $(j+1)$ times the number of such copies. Let us now crudely bound these. Such a copy consists of a member $\mathcal{F}''$ of $Y_v^{i,0}(\ell)$ and $j+1$ copies of $W^1$ rooted at vertices of $\mathcal{F}''$. Since $\omega \notin \mathcal{B}_{\ell}$, by (\ref{Wstate}) we have that $W^1_u(\ell) = \plog(d)$ for all $u \in V(\mathcal{H})$.

So the number of ways to create a member of $Y_v^{i,j}(\ell+1)$ via (1) when $|e^* \cap \mathcal{F}'|= 1$,  is within
\begin{align*}
\label{Ytyp1}
(j+1)Y^{i-1,j+1}_v(\ell) \pm Y_v^{i,0}(\ell)\plog(d).
\end{align*}
Since $\omega\notin \mathcal{B}_\ell$, we have
\begin{equation}
\label{Yexp1}
(j+1)Y^{i-1,j+1}_v(\ell) \pm Y_v^{i,0}(\ell)\plog(d)\in (j+1)(1\pm 2\epsilon(t_\ell))y_{i-1,j+1}(t_\ell)d^{1-\frac{i-1}{r-1}}.
\end{equation}

Now consider the case when $|e^* \cap \mathcal{F}'| > 1.$ As we will see, this will give a lower order term. In this case, $e^* \cup \mathcal{F}'$ is a member of $Z_v^{i-1,j+1}(\ell)\setminus Y_v^{i-1,j+1}(\ell)$. So by Observation~\ref{Zsecondary}, $e^* \cup \mathcal{F}'$ consists of: 
\begin{itemize}
\item[(a)] a copy of a secondary configuration $X = (\mathcal{F}^*,R^*,D^*)$ in $\mathcal{H}(\ell)$ with a unique root and precisely $i-1$ marked vertices in the central hyperedge such that every non-central hyperedge has a unique neutral vertex (also contained in the central hyperedge), and
\item [(b)] at most $j$ copies of $W^1$ rooted at vertices of the central hyperedge.
\end{itemize}

So we can bound the number of ways of creating a member of $Y_v^{i,j}(\ell +1)$ via (1) in the second case by the sum of $X_v(\ell)$ multiplied by $\plog(d)$ (for the copies of $W^1$), over all secondary configurations $X$ satisfying (a).

Since $\omega\notin\mathcal{B}_{\ell}$ and $|V(\mathcal{F}^*)| - |R^*| - |D^*|= r-i$ for each secondary configuration $X$ satisfying (a), by (\ref{Xstate}) this is at most 
\begin{equation}
\label{Yexp2}
\plog(d)\cdot d^{1-\frac{i-1}{r-1}} \log^{-3K/5}(d) = o\left(d^{1-\frac{i-1}{r-1}}\epsilon(t_\ell)\right).
\end{equation}
Putting \eqref{Yexp1} and \eqref{Yexp2} together with \eqref{sampled} shows that the expected number of copies of $Y^{i,j}$ rooted at $v$ created via (1) is 
\begin{equation}
\label{term1}
\frac{\alpha (j+1)(1\pm 5\epsilon(t_\ell))y_{i-1,j+1}(t_\ell)d^{1-\frac{i}{r-1}}}{(1\mp 4\epsilon(t_\ell))\gamma(t_{\ell})N} \subseteq \frac{\alpha (j+1)(1\pm 20\epsilon(t_\ell))y_{i-1,j+1}(t_\ell)d^{1-\frac{i}{r-1}}}{\gamma(t_{\ell})N},
\end{equation}
since $\left|\frac{1+ 5x}{1- 4x}\right|\leq 1 \pm 20x$ for $x$ sufficiently small.

Counting the number of ways to create a copy of $Y^{i,j}$ rooted at $v$ via (2) (recall that Figure~\ref{Ymake2} provides examples of this) is equivalent to counting
\begin{itemize}
\item the number of ways to choose an element $\mathcal{F}'$ of $Y^{i,j-1}_v(\ell)$;
\item a healthy vertex $u \in V(\mathcal{F}')\setminus \{v\}$ contained only in the central hyperedge; 
\item  an element $\mathcal{G}'$ of $Y_u^{r-2,0}(\ell)$ (a single edge) that contains exactly $r-2$ infected vertices and intersects $\mathcal{F}'$ precisely on $u$,
\item an element $e^*$ of $Q_x(\ell)$, where $x$ is the unique healthy vertex of $\mathcal{G}' \setminus \{u\}$. ($e^*$ is the hyperedge shaded light grey in the examples in Figure~\ref{Ymake2}.)
\end{itemize}

As $\omega \not\in \mathcal{B}_{\ell}$, there are at most $O\left(d^{1 - \frac{i+1}{r-1}}\plog(d)\right)$ members of 
$Y_v^{i,j-1}(\ell)$ that do not contain exactly $i$ infected vertices in their central hyperedge (as we just did in (1), we can bound the number of possible central hyperedges by $Y_v^{i+1,0}(\ell)$ and other hyperedges (which are copies of $W^1$) using (\ref{Wstate})). For each $\mathcal{F}'$, the number of choices of $u$ is precisely the number of healthy vertices in the central hyperedge of $\mathcal{F}'$. So, the number of ways to choose such an $\mathcal{F}'$ and $u$ is contained in
$$(r-i-j)Y_v^{i,j-1}(\ell) \pm O\left(d^{1 - \frac{i+1}{r-1}}\plog(d)\right).$$
As $\omega \not\in \mathcal{B}_{\ell}$, by (\ref{Ystate}) this is contained in
\begin{equation}
\label{fucount}
(r-i-j)\left(1 \pm 2\epsilon(t)\right)y_{i,j-1}(t)d^{1 - \frac{i}{r-1}}.
\end{equation}

Let $C(\mathcal{F}',u, \mathcal{G}'')$ be the number of ways to pick such an $\mathcal{F}'$ and $u$ multiplied by the number of ways to pick some $\mathcal{G}'' = \mathcal{G}' \cup \{e^*\} \in Y_v^{r-2,1}(\ell)$. So as $\omega \notin \mathcal{B}_{\ell}$, by (\ref{Ystate}) and \eqref{fucount} we have
\begin{equation}
\label{cfug1}
C(\mathcal{F}', u, \mathcal{G}'')= (r-i-j)\left(1 \pm 4\epsilon(t)\right)y_{i,j-1}(t) y_{r-2,1}(t)d^{1 - \frac{i-1}{r-1}}.
\end{equation}

The quantity $C(\mathcal{F}', u, \mathcal{G}'')$ both fails to count some triples we wish to count and also counts some triples we do not want to count. However, we will show that both these terms are of a lower order. So up to an error term, the number of ways to create a copy of $Y^{i,j}$ rooted at $v$ via (2) is $C(\mathcal{F}', u, \mathcal{G}'')$. See Figure~\ref{cug1} for examples of the triples we fail to count and Figure~\ref{cug2} for examples of the triples we count unwantedly. 

\begin{figure}[hbtp]
\centering
\includegraphics[width=0.9\textwidth]{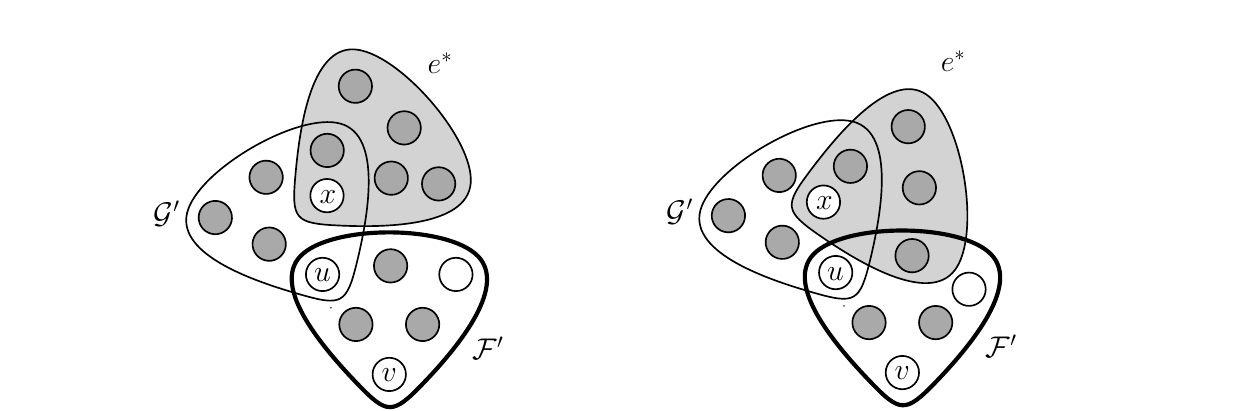}
\caption{Here $r=6$, $i=3$ and $j=1$. Two examples of hypergaphs that are not counted in $C(\mathcal{F}', u, \mathcal{G}'')$. If the shaded hyperedge were successfully sampled, this would create a copy of $Y^{3,1}$ whose central hyperedge is drawn with a thick outline. Infected vertices are shaded dark grey and healthy vertices are unshaded.}
\label{cug1}
\end{figure}

\begin{figure}[hbtp]
\centering
\includegraphics[width=0.9\textwidth]{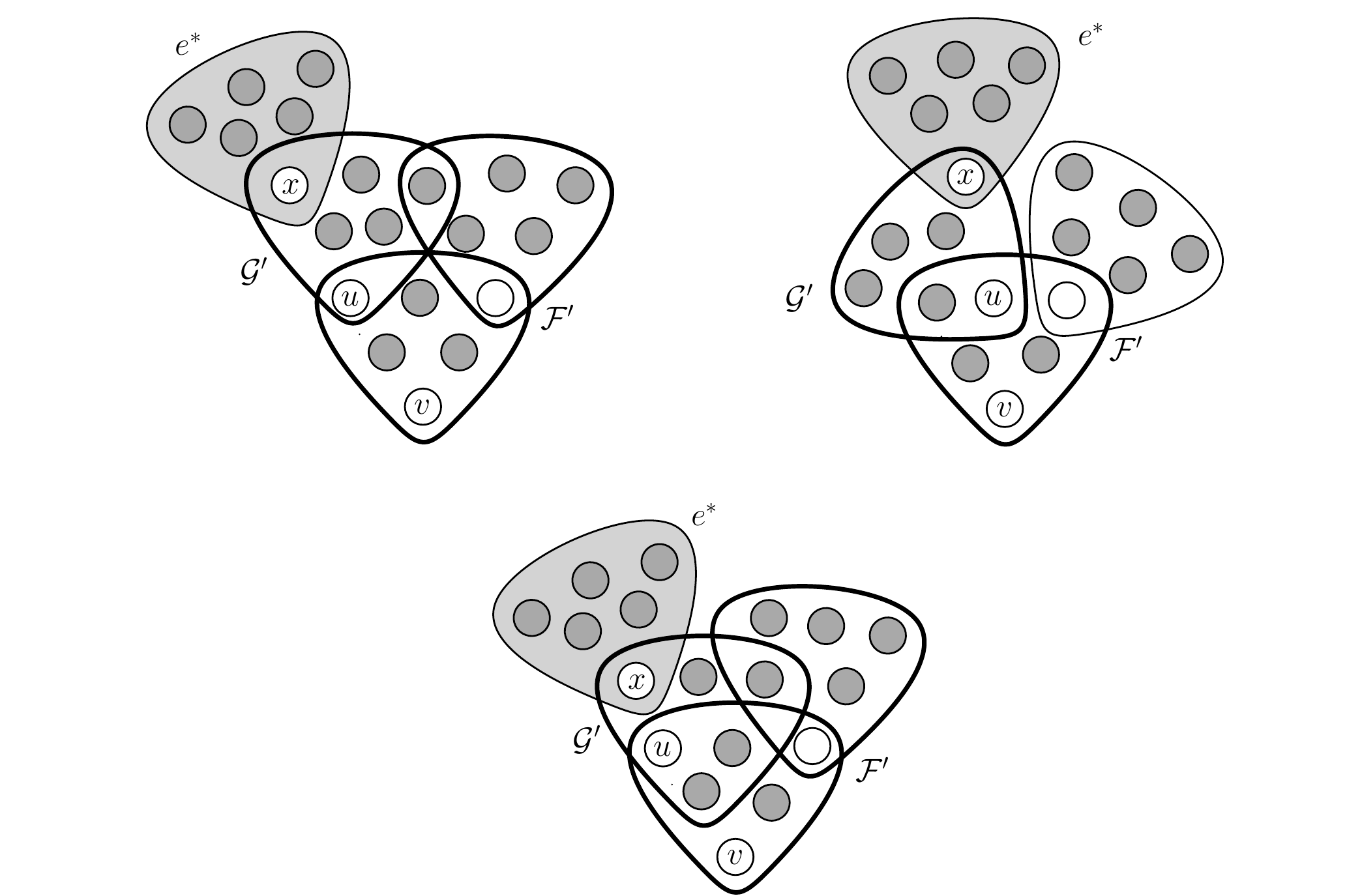}
\caption{Here $r=6$, $i=3$ and $j=2$. Three examples of hypergraphs that are unwantedly counted in $C(\mathcal{F}', u, \mathcal{G}'')$. Observe that $|\mathcal{F}' \cap \mathcal{G}'| \ge 2$ and $\mathcal{F}' \cup \mathcal{G}'$ consists of a copy $X$ of a secondary configuration  (whose edges are drawn with a thick outline) and some additional copies of $W^1$ rooted at healthy vertices of $X$. The central hyperedge of $X$ is the hyperedge containing $v$. In any of these cases, if the shaded hyperedge were successfully sampled, it would not turn $\mathcal{F}' \cup \mathcal{G}'$ into a copy of $Y^{3,2}$.}
\label{cug2}
\end{figure}

The triples which $C(\mathcal{F}', u, \mathcal{G}'')$ fails to count are those where $e^*$ intersects $\mathcal{G}'$ in more than one vertex. However, for such a pair $\mathcal{G}'$ and $e^*$, we have that $\mathcal{G}' \cup e^*$ is a copy in $\mathcal{H}(\ell)$ of a secondary configuration $X$ rooted at $u$ with precisely one neutral vertex. As $\omega \notin \mathcal{B}_{\ell}$, by (\ref{Xstate}) we have
$X_u(\ell) \le d^{\frac{1}{r-1}}\log^{-2K/5}$. Summing over all such configurations $X$, and using \eqref{fucount} gives that there are \begin{equation}
\label{cfug2}
O\left(d^{1 - \frac{i-1}{r-1}}\log^{-2K/5}\right)
\end{equation} 
triples that are not counted in $C(\mathcal{F}', u, \mathcal{G}'')$.

Now consider the triples which $C(\mathcal{F}', u, \mathcal{G}'')$  counts that we do not want (Figure~\ref{cug2} provides examples of these). The triples we wish to exclude are precisely those where $|\mathcal{F}' \cap \mathcal{G}'| \ge 2$. Note that $\mathcal{G}'$ could intersect either the central hyperedge or a non-central hyperedge of $\mathcal{F}'$  (or both). However for any such pair, $\mathcal{F}' \cup \mathcal{G}' \cup e^{*}$ consists of:
\begin{itemize}
\item[(a)] a copy in $\mathcal{H}(\ell)$ of a secondary configuration $X$ with a unique root ($v$) and exactly $i$ marked (so at least $i$ infected) vertices in the central hyperedge, and
\item[(b)] at most $j$ copies of $W^1$ (one is $e^*$ and the others are the non-central hyperedges of $\mathcal{F}'$). (The reason we have \emph{at most} $j$ is because the secondary configuration could contain 2 or 3 hyperedges.)
\end{itemize}

So similarly to the second case of (1) above, since $\omega\notin\mathcal{B}_{\ell}$ and $|V(\mathcal{F}^*)| - |R^*| - |D^*|= r-(i+1)$ for each secondary configuration $X = (\mathcal{F}^*,R^*,D^*)$ satisfying (a), by (\ref{Xstate}) and (\ref{Wstate}) the number of triples we wish to exclude is 
\begin{equation}
\label{cfug3}
\log^{O(1)}(d) d^{1 - \frac{i}{r-1}} \log^{-3K/5}(d)= O\left(d^{1 - \frac{i}{r-1}}\log^{-2K/5}\right),
\end{equation}
for $K$ chosen sufficiently large with respect to $r$ (once $\mathcal{F}'$, $u$ and $\mathcal{G}''$ are chosen, by (\ref{Wstate}) there are $\plog(d)$ choices for $e^*$).

So putting \eqref{cfug1}, \eqref{cfug2} and \eqref{cfug3} together, the number of ways to choose $\mathcal{F}'$, $u$, $\mathcal{G}'$ and $e^*$ is contained in
\begin{equation}
\label{Yexp3}
(r-i-j)\left(1 \pm 5\epsilon(t)\right)y_{i,j-1}(t) y_{r-2,1}(t)d^{1 - \frac{i-1}{r-1}}.
\end{equation}

Combining \eqref{Yexp3} and \eqref{sampled} gives that the expected number of copies of $Y^{i,j}$ rooted at $v$ created in the second way is contained in
\begin{equation}
\label{term2}
\frac{\alpha(r-i-j)(1\pm 20\epsilon(t_\ell)) y_{i,j-1}(t_\ell)y_{r-2,1}(t_\ell)d^{1-\frac{i}{r-1}}}{\gamma(t_{\ell})N}.
\end{equation}
Therefore, by \eqref{term1} and \eqref{term2} we have:
\begin{equation}
\label{gain}
\mathbb{E}(C_{\ell + 1}\given \mathcal{F}_{\ell}) \in \frac{\alpha(1\pm 20\epsilon(t_\ell))d^{1-\frac{i}{r-1}}\left( (j+1)y_{i-1,j+1}(t_\ell) +(r-i-j)y_{i,j-1}(t_\ell)y_{r-2,1}(t_\ell)\right)}{\gamma(t_{\ell})N}.
\end{equation}
By Definition~\ref{copyDef}, we see that the only way in which a copy of $Y^{i,j}$ rooted at $v$ can be destroyed in the $(\ell+1)$th step is if one of its hyperedges is sampled. Let $D_{\ell + 1}$ be the number of such destroyed copies. As discussed above, since $\omega\notin \mathcal{B}_\ell$, all but at most $\plog{(d)}\cdot d^{1-\frac{i+1}{r-1}}$ copies of  $Y^{i,j}$ in $\mathcal{H}(\ell)$ rooted at $v$ have exactly $i$ infected vertices in the central hyperedge. A copy whose central hyperedge is open is destroyed with probability $(j+1)q/Q(\ell)$. However, for the vast majority of copies of $Y^{i,j}$ rooted at $v$, the central hyperedge is not open. Using \eqref{sampled}, the probability that any such copy is destroyed is precisely 

$$\frac{jq}{Q(\ell)} \in \frac{\alpha j}{(1\mp 4\epsilon(t_\ell))\gamma(t_{\ell})N}.$$

Therefore we have
$$\mathbb{E}(D_{\ell + 1}\given \mathcal{F}_{\ell}) \in \frac{\alpha j (Y_v^{i,j}(\ell) - \plog{(d)}\cdot d^{1-\frac{i+1}{r-1}}) \pm \alpha (j+1)\plog{(d)}\cdot d^{1-\frac{i+1}{r-1}}}{(1\mp 4\epsilon(t_\ell))\gamma(t_{\ell})N}.$$  
As $\omega \notin \mathcal{B}_{\ell}$, by (\ref{Ystate}) we have
\begin{equation}
\label{term3}
\mathbb{E}(D_{\ell + 1}\given \mathcal{F}_{\ell}) \in \frac{\alpha j(1\pm 5\epsilon(t_\ell))y_{i,j}(t_\ell)d^{1- \frac{i}{r-1}}}{(1\mp 4\epsilon(t_\ell))\gamma(t_{\ell})N} \subseteq \frac{\alpha j(1\pm 20\epsilon(t_\ell))y_{i,j}(t_\ell)d^{1- \frac{i}{r-1}}}{\gamma(t_{\ell})N}.
\end{equation}
As $Y_v^{i,j}(\ell +1) - Y_v^{i,j}(\ell) = C_{\ell +1} - D_{\ell+1}$, by linearity of expectation, \eqref{gain} and \eqref{term3} we get that 
\begin{equation}
\label{Yexpectation}
\mathbb{E}\left(Y_v^{i,j}(\ell + 1) - Y_v^{i,j}(\ell)\given \mathcal{F}_{\ell}\right) \in \frac{(1\pm 20\epsilon(t_\ell))y_{i,j}'(t_\ell)}{N}d^{1 - \frac{i}{r-1}},
\end{equation}
by \eqref{yijprime}.
Therefore the expectation of $A_v^+(\ell)$ given $\mathcal{F}_\ell$ is positive and the expectation of $A_v^-(\ell)$ given $\mathcal{F}_\ell$ is negative. So we have that $B_v^+(0),\dots,B_v^+(M)$ is a submartingale and $B_v^-(0),\dots,B_v^-(M)$ is a supermartingale, as required for Subclaim~\ref{Ymart}. 
\end{proof}

Now, let us bound the maximum and minimum possible values of $A_v^{\pm}(\ell)$. 
\begin{subclaim}
\label{Yetabound}
Both $B_v^+(0),\dots,B_v^+(M)$ and $B_v^-(0),\dots,B_v^-(M)$ are $\eta$-bounded, for 
\[\eta:=d^{1-\frac{i}{r-1}}\log^{-K/2}(d).\]
\end{subclaim}
\begin{proof}
The maximum value of $Y_v^{i,j}(\ell+1)-Y_v^{i,j}(\ell)$ occurs when the $(\ell+1)$th sampling is successful and the newly infected vertex creates many new copies of $Y^{i,j}$ rooted at $v$ in either the first or second way (see above). For each $u \in D$ let $\tilde{Y}^u$ denote the configuration $(\mathcal{F}, R \cup \{u\}, D \setminus \{u\})$. If some vertex, say $x$, becomes infected in the $(\ell + 1)$th step, the number of copies of $Y^{i,j}$ created is $\sum_{u \in D}\tilde{Y}^u_{S \cup x}$. Each configuration $\tilde{Y}^u$ consists of a copy $\mathcal{F}'$ of firstly a secondary configuration with $(r - 1 - i)$ neutral vertices, and secondly some additional copies of $W^1$ rooted at vertices of $\mathcal{F}'$. So as $\omega\notin \mathcal{B}_\ell$ the maximum number of copies of $Y^{i,j}$ rooted at $v$ that can be created in a time step is
\begin{equation}\label{eta1}
\plog(d)\cdot d^{1-\frac{i}{r-1}}\log^{-3K/5}(d) = o(\eta).
\end{equation}

As mentioned above, a copy of $Y^{i,j}$ rooted at $v$ can only be destroyed in the $(\ell+1)$th step if it contains an open hyperedge which is sampled. As $\omega\notin \mathcal{B}_\ell$, the number of copies of $Y^{i,j}$ rooted at $v$ in which the hyperedge containing $v$ is open is $\log^{O(1)}(d)$. If $x\neq v$ is the unique healthy vertex in the sampled hyperedge, then the maximum number of copies that can be destroyed is at most the number of hyperedges containing $\{v, x\}$ and $i$ (other) infected vertices, multiplied by $\plog(d)$ (for the other hyperedges, which are copies of $W^1$). A hyperedge containing $\{v, x\}$ and $i$ infected vertices is either a copy in $\mathcal{H}(\ell)$ of $W^2$ rooted at $\{v, x\}$ (when $i= r- 2$)  or a copy in $\mathcal{H}(\ell)$ of a secondary configuration $X = (\mathcal{F},R,D)$ where $\mathcal{F}$ is a single hyperedge, $|R|= 2$ and $|D|= i$ (when $i < r-2$).  So as $\omega \not\in \mathcal{B}_{\ell}$, by (\ref{Xstate}) and (\ref{Wstate}), the maximum number of copies that can be destroyed in a time step is
\begin{equation}\label{eta2}
\max\{\plog(d),\plog(d)\cdot d^{1-\frac{i+1}{r-1}}\log^{-3K/5}(d)\} = o(\eta).
\end{equation}
 As $t = O\left(\log(d)\right)$, $N \ge d^{\frac{1}{r-1}}\log^{bK}(d)$ and $K$ is chosen to be large with respect to $r,c$ and $\alpha$, \begin{equation}\label{albound}
 a_{\ell}^{\pm} = \frac{\plog(d)d^{1 - \frac{i}{r-1}}}{N} = o(\eta).
 \end{equation} 
By \eqref{eta1}, \eqref{eta2} and \eqref{albound}, both $B_v^+(0),\dots,B_v^+(M)$ and $B_v^-(0),\dots,B_v^-(M)$ are $\eta$-bounded.
\end{proof}
Now we bound the variance of $A_v^{\pm}(\ell)$ given $\mathcal{F}_\ell$. 
\begin{subclaim}
\label{varY}
$$\sum_{\ell=0}^{m-1}\Var\left(A_S(\ell)\mid\mathcal{F}_\ell\right) \le d^{2-\frac{2i}{r-1}}\log^{-9K/20}(d).$$
\end{subclaim}

\begin{proof}
The calculation follows similarly to the calculations for Claim~\ref{varW} and Claim~\ref{varX}. When $\omega \in \mathcal{B}_{\ell}$, we have $\Var\left(A_S(\ell)\mid\mathcal{F}_\ell\right)=0$. So now consider when $\omega \not\in \mathcal{B}_{\ell}$. We have
\begin{align}\label{expYchange}
\Var\left(A^{\pm}_S(\ell)\mid\mathcal{F}_\ell\right) &\le \mathbb{E}\left(\left(Y_v^{i,j}(\ell+1)-Y_v^{i,j}(\ell)\right)^2\mid\mathcal{F}_\ell\right) \nonumber\\ 
& \le \mathbb{E}\left(|Y_v^{i,j}(\ell+1) - Y_v^{i,j}(\ell)|\mid \mathcal{F}_{\ell} \right) \cdot \sup |Y_v^{i,j}(\ell+1) - Y_v^{i,j}(\ell)|.
\end{align}

 In the proof of Subclaim~\ref{Yetabound} we saw that the maximum value of $|Y_v^{i,j}(\ell+1)-Y_v^{i,j}(\ell)|$ is $\eta$. So we have
$$\Var\left(A^{\pm}_S(\ell)\mid\mathcal{F}_\ell\right) \le \eta\cdot\mathbb{E}\left(|Y_v^{i,j}(\ell+1) - Y_v^{i,j}(\ell)|\mid \mathcal{F}_{\ell} \right).$$  

By expressing $Y_v^{i,j}(\ell+1) - Y_v^{i,j}(\ell)$ as $C_{\ell+1} - D_{\ell + 1}$ as in the proof of Subclaim~\ref{Ymart}, we have
$$\mathbb{E}\left(|Y_v^{i,j}(\ell+1) - Y_v^{i,j}(\ell)|\mid \mathcal{F}_{\ell} \right) \le \mathbb{E}\left(C_{\ell+1} + D_{\ell + 1}\mid \mathcal{F}_{\ell}\right).$$
By \eqref{gain}, \eqref{term3} and definition of $\eta$, this is at most
\[\eta(1+20\epsilon(t_\ell))d^{1-\frac{i}{r-1}}\left(\frac{(j+1)\alpha y_{i-1,j+1}(t_\ell)}{\gamma(t_{\ell})N} + \frac{(r-i-j)\alpha y_{i,j-1}(t)y_{r-2,1}(t_\ell)}{\gamma(t_{\ell})N} + \frac{j\alpha y_{i,j}(t_\ell)}{\gamma(t_{\ell})N}\right)\]
\[=o\left(\frac{d^{2-\frac{2i}{r-1}}}{M \log^{9K/20}(d)}\right).\]
So the sum of $\Var(A_v^\pm(\ell)\mid\mathcal{F}_\ell)$ over $0\leq \ell\leq M$ is at most $d^{2-\frac{2i}{r-1}}\log^{-9K/20}(d)$, as required.
\end{proof}
Set 
\[\nu:=d^{2-\frac{2i}{r-1}}\log^{-9K/20}(d),\]
and set $a:=\frac{1}{2}\epsilon(t_m)y_{i,j}(t_m)d^{1- \frac{i}{r-1}}$. Theorem~\ref{Freed} implies that
\[\mathbb{P}\left(B_v^-(m)\geq a \right) \leq \exp\left(-\frac{a^2}{2(\nu + a\eta)}\right)\ll N^{-5\sqrt{\log(N)}}\]
and 
\[\mathbb{P}\left(B_v^+(m)\leq -a \right) \leq \exp\left(-\frac{a^2}{2(\nu + a\eta)}\right)\ll N^{-5\sqrt{\log(N)}}.\]
This completes the proof of Claim~\ref{Yclaim}.
\end{proof}
This claim was the final piece in the proof of Lemma~\ref{Yprop}. Proving Lemma~\ref{Yprop} completes the proof of Lemma~\ref{phaseoneevent} and concludes our discussion of the first phase. 

\section{The Second Phase in the Subcritical Case}
\label{sec:phaseTwoSub}

In this section we will complete the proof of Theorem~\ref{hypmainThm} in the subcritical case. It may be helpful to recall the definition of the processes we run in the second phase in the subcritical case from Subsection~\ref{subcritrough} and the definition of $\mathcal{F}_m'$ from Subsection~\ref{sec:probspace}. Remember that in the second phase we ``restart the clock'' again from time zero letting $\mathcal{H}(0):= \mathcal{H}(M)$, $I(0):= I(M)$. The number of steps in the second phase is $M_2$, which was defined in \eqref{M2def}. We repeat the definition of $M_2$ in the subcritical case here, for the sake of convenience:
\begin{equation}\label{M2againsub}M_2= 2\left\lceil\log_{1/(1-\lambda)}(N)\right\rceil.\end{equation}

We begin by presenting some technical definitions of various constants that will be used throughout this section. These constants are required to state our main lemmas and are useful to simplify future calculations.

First we present some further analysis of the function $\gamma(t)$ (defined in \eqref{ft}). We have
\[\gamma'(t)=\alpha(r-1)(c+\alpha t)^{r-2} - 1.\]
Recall that $t_{\min}$ was chosen to be the unique positive root of $\gamma'(t)$ and that $T_0 < t_{\min} < T_1$ where $T_0,T_1$ are the only two positive roots of $\gamma(t)$. As $\gamma(0) >0$, $\gamma(T_0)=0$ and $T_0 < t_{\min}$, the function $\gamma(t)$ is strictly decreasing in the interval $[0,t_{\min}]$ and so $\gamma'(T_0)<0$. Therefore $\alpha(r-1)(c+\alpha T_0)^{r-2} - 1 <0$, which implies that
$$(r-1)(c+\alpha T_0)^{r-2} < \frac{1}{\alpha}.$$
So there exists a constant $0 < \lambda<1/8$ such that 
\begin{equation}
\label{T0q'}
(r-1)(c+\alpha T_0)^{r-2} < \frac{1-4\lambda}{\alpha}.
\end{equation}
Recall from \eqref{Mdef} and \eqref{Tdef} that we run the first phase process until time $T$, where
$$T:= \frac{1}{N}\min\{m\geq0: (1 + 4\epsilon(t_m))\gamma(t_m)<\zeta\}$$
and $\zeta$ will be chosen in a moment (in Definition~\ref{zetaDef}). Using the fact that  $T<T_0$ and that $(r-1)(c+\alpha t)^{r-2}$ is increasing for positive $t$, we obtain the following from \eqref{T0q'},
\begin{equation}
\label{Tq'}
(r-1)(c+\alpha T)^{r-2} < \frac{1-4\lambda}{\alpha}.
\end{equation}
We are now ready to define $\zeta$ and some other useful constants. 

\begin{defn}
\label{zetaDef}
For $0\leq i\leq r-2$, define
\[\chi(i,0):=\binom{r-1}{i}\left(\frac{1-4\lambda}{\alpha(r-1)}\right)^{\frac{i}{r-2}} +\frac{\lambda}{\alpha}.\]
Also, define
\[\chi_{\max}:=\max_{0\leq i\leq r-2}\chi(i,0),\]
\[\chi_{\min}:=\min_{0\leq i\leq r-2}\chi(i,0),\]
\[\zeta:=\frac{\lambda^2\chi_{\min}}{r^{6r+1}\left(1+\alpha+\alpha^r\right)^2\left(1+\chi_{\max}\right)^3},\]
and
\[\chi(r-2,1):=\left(1+\chi_{\max}\right)\cdot\zeta.\]
For $0\leq i\leq r-3$ and $1\leq j\leq r-i-1$, define
\[\chi(i,j):=r^{3r}\chi(i,0)(1+\alpha^j)\chi(r-2,1)^j.\]
\end{defn}
Note that as $\lambda < 1$ and $ \chi_{\min} < 1 + \chi_{\max}$,
\begin{equation}\label{chir20}
\chi(r-2,1):=\left(1+\chi_{\max}\right)\cdot\zeta = \frac{\lambda^2\chi_{\min}}{r^{6r+1}\left(1+\alpha+\alpha^r\right)^2\left(1+\chi_{\max}\right)^2} < 1.
\end{equation}
Note in addition that 
\begin{equation}
\label{r2def}
\chi(r-2,0) = \frac{1 - 3\lambda}{\alpha}.
\end{equation}
Also observe that as $\lambda$ is a function of $r,c$ and $\alpha$, so is $\zeta$. Before presenting the main results and proofs of the section, let us deduce a simple fact about $\chi(i,j)$ which will be used a number of times below.
\begin{prop}
\label{chiijBound}
For $0 \le i \le r-3$ and $1 \le j \le r-i-1$,
$$\chi(i,j) < \frac{\lambda^2}{r^{3r+1}\alpha^{j+1}}.$$
\end{prop}
\begin{proof}
Note that, for any $\alpha>0$ and $1\leq j\leq r-1$, we have that $1+\alpha +\alpha^r$ is greater than both $1+\alpha^j$ and $\alpha^{j+1}$. Thus, for $0\leq i\leq r-3$ and $1\leq j\leq r-1-i$, we have
\begin{align}\label{prop7.5}
\chi(i,j)&=\left[r^{3r}\chi(i,0)(1 +\alpha^j)\chi(r-2,1)\right]\chi(r-2,1)^{j-1} \nonumber \\
& = \left[r^{3r}\chi(i,0)(1+\alpha^j)(1+\chi_{\max})\cdot \zeta\right]\chi(r-2,1)^{j-1} \nonumber \\
&=\left(\frac{r^{3r}\chi(i,0)(1+\alpha^j)\lambda^2\chi_{\min}}{r^{6r+1}(1+\alpha+\alpha^r)^2(1+\chi_{\max})^2}\right)\chi(r-2,1)^{j-1}\nonumber \\
&<  \left(\frac{\lambda^2\chi_{\min}}{r^{3r+1}(1+\alpha+\alpha^r)(1+\chi_{\max})}\right)\chi(r-2,1)^{j-1} 
\\&< \frac{\lambda^2}{r^{3r+1}\alpha^{j+1}}, \nonumber 
\end{align}
as by \eqref{chir20}, $\chi(r-2,1)<1$.
\end{proof}

Throughout the section we will also use the following function.
\begin{equation}
\label{psidef}
\psi(m):= \max\left\{(1-\lambda)^m, \log^{-10}(N) \right\}.
\end{equation}

We have now completed the tedious technicalities and are able to present the real meat of the section. Let us briefly outline how the proof will proceed. For each time $m$, we will define $\mathcal{A}_m$ (similarly to how we defined $\mathcal{B}_m$ for the first phase) to be the event that some bound on the number of copies of a particular configuration in $\mathcal{H}(m)$ fails to hold. The event $\mathcal{A}_0$ will be formally defined in Definition~\ref{time0} and, for $0 < m \le M_2$, $\mathcal{A}_m$ will be defined in Definition~\ref{Amevent}.

We will use the following lemma. (Recall Definition~\ref{Zij} for the definition of $Z^{i,j}$.)

\begin{lem}\label{subtrack}
Let $0 \le m \le M_2$. When $\omega \notin \mathcal{A}_m$ and $K$ is sufficiently large, the following statement holds. For all $S$ and $v$ contained in $V(\mathcal{H})$:
\begin{enumerate}
\renewcommand{\theenumi}{L.\arabic{enumi}}
\renewcommand{\labelenumi}{(\theenumi)}
\item \label{yi0} For $0\leq i\leq r-2$, 
$$Y_v^{i,0}(m) \le \chi(i,0)d^{1 - \frac{i}{r-1}}.$$ 
\item \label{zij} For $0 \le i \le r-2$ and $1 \le j \le r-1-i$, 
$$Y_v^{i,j}(m) \le \psi(m)\chi(i,j)d^{1 - \frac{i}{r-1}}.$$
\item \label{ws} For $1 \le i \le r$,
$$W^i_S(m) \le (m+1)\log^{r^3(r-i)}(d).$$
\item \label{xs} For every secondary configuration $X = (\mathcal{F},R,D)$,
$$X_S(m) \leq (m+1)\log^{2|D|r^7}(d)\cdot d^{\frac{|V(\mathcal{F})| - |R|-|D|}{r-1}}\log^{-3K/5}(d).$$
\item \label{Zbound} For $0 \le i \le r-2$ and $1 \le j \le r-1-i$,
\[Z_v^{i,j}(m)  \le  \left(1+\frac{\lambda^2}{r^{3r}}\right)\psi(m)\chi(i,j)d^{1-\frac{i}{r-1}}.\]
\end{enumerate}
\end{lem}

 We will also prove that  with high probability $\mathcal{A}_{M_2}$ does not occur.

\begin{lem}\label{Amthm}
$$\mathbb{P}\left(\mathcal{A}_{M_2}\right) \le N^{-\sqrt{\log(N)}}.$$
\end{lem}

 As we will see in Definition~\ref{Amevent}, the event $\mathcal{A}_{M_2}$ contains the events $\mathcal{A}_m$ for $0 \le m < M_2$. Thus Lemma~\ref{Amthm} implies that with high probability for all $S$ and $v$ contained in $V(\mathcal{H})$ and all $0 \le m \le M_2$ the bounds (\ref{yi0})-(\ref{Zbound}) hold. 

We now derive Lemma~\ref{submain} from Lemmas~\ref{Amthm} and~\ref{subtrack}, which implies Theorem~\ref{hypmainThm} in the subcritical case. After this, we will focus on the proof of Lemmas~\ref{Amthm} and~\ref{subtrack}. 

\begin{proof}[Proof of Lemma~\ref{submain}]
Our goal is to show that, with probability $1-o(1)$, we have 
\begin{enumerate}[(i)]
\item $Q(M_2)=\emptyset$, and\label{QM2}
\item $|I(M_2)|= \frac{N\cdot\plog(N)}{d^{1/(r-1)}}$.\label{notManyInfected}
\end{enumerate}
From this, it follows easily that the probability of percolation is at most $\varepsilon$. 

For~\ref{QM2}, by Markov's Inequality (Theorem~\ref{Markov}), we have
\[\mathbb{P}(Q(M_2)>0)\leq \mathbb{E}(Q(M_2)).\]
Thus, it suffices to show that the right side is $o(1)$. This is implied by combining our choice of $M_2$ (see \eqref{M2againsub}) with the following claim. 

\begin{claim}
\label{expQproof}
For $0\leq m\leq M_2$,
$$\mathbb{E}(Q(m)) \le 2(1 - \lambda)^m \zeta \cdot N.$$
\end{claim}
\begin{proof}
The proof proceeds by induction on $m$.  First consider the base case $m=0$. 
 By definition of $T$, Lemma~\ref{phaseoneevent} and Lemma~\ref{QfromOthers}, we have $Q(0) \le \zeta \cdot N$ with probability at least $1 - N^{-2\sqrt{\log(N)}}$. As $\mathcal{H}$ has maximum degree $d$, we have $Q(0) \le N\cdot d$ and so
$$\mathbb{E}(Q(0))\le \zeta \cdot N + N\cdot d \cdot N^{-2\sqrt{\log N}} \le 2\zeta \cdot N,$$
as required.

Now suppose that the statement of the claim is proved for all $0 \le \ell \le m$ and we wish to show it holds for $m+1$. If $Q(m) = \emptyset$ we are done, so we may assume that $Q(m) \not= \emptyset$. Recall that in each round of the second phase, every open hyperedge is sampled. Thus, $Q(m+1)\cap Q(m)=\emptyset$. This implies that, for each $e\in Q(m+1)$, there must be at least one vertex $x$ of $e$ such that $x\in I(m+1)\setminus I(m)$; that is, $x\notin I(m)$ and there is a hyperedge $e^*\in Q(m)$ containing $x$ which was successfully sampled in the $m$th step. Therefore, conditioned on $\mathcal{F}_m'$, the expectation of $Q(m+1)$ is at most the product of 
\begin{itemize}
\item the number of ways to choose a hyperedge $e^*\in Q(m)$ containing a vertex $x\notin I(m)$, and
\item the sum of $Z_x^{r-2-j,j}(m)q^{j+1}$ over all $0\leq j\leq r-2$.
\end{itemize}

When $\omega \not\in \mathcal{A}_m$ (and $K$ is sufficiently large), by Lemma~\ref{subtrack} for every $w\in V(\mathcal{H})$  and $1 \le j \le r-2$, we have the following two bounds. First, using  (\ref{Zbound}) we have
\[Z_w^{r-2-j,j}(m)q^{j+1} \leq \left(1+\frac{\lambda^2}{r^{3r}}\right)\psi(m)\chi(r-2-j,j) d^{\frac{j+1}{r-1}}q^{j+1}\leq 2\chi(r-2-j,j) \alpha^{j+1}\]
\[\leq 2\left(\frac{\lambda^2}{r^{3r+1}\alpha^{j+1}}\right)\alpha^{j+1} < \frac{\lambda}{2r}.\]
and by \eqref{r2def} and  (\ref{yi0}) we have
\[Z_w^{r-2,0}(m)q = Y_w^{r-2,0}(m)q \leq \chi(r-2,0)\alpha = 1-3\lambda.\]

So when $\omega \notin \mathcal{A}_m$, we have 
\begin{align*}
\mathbb{E}(Q(m+1)\mid \mathcal{F}_m')&\leq \sum_{Q \in Q(m)} \sum_{w \in Q \setminus I(m)} \sum_{j = 0}^{r-2}Z_w^{r-2-j,j}(m) q^{j+1} \\
& \le Q(m) \left(\left(1-3\lambda\right) + \sum_{j=1}^{r-2}\frac{\lambda}{2r}\right) \\
& \leq Q(m)\left(1- 2\lambda\right).
\end{align*}
As $\mathcal{H}$ has maximum degree $d$, we have $Q(m) \le N \cdot d$. So, letting $\mathbbm{1}_{\mathcal{E}}$ be the indicator function of the event $\mathcal{E}$ occuring, we have
\begin{align}
\label{silly}
\mathbb{E}(Q(m+1)\mid \mathcal{F}_m') &\le Q(m)\left(1- 2\lambda\right)\cdot \mathbbm{1}_{\mathcal{A}_m^c} + N \cdot d \cdot \mathbbm{1}_{\mathcal{A}_m} \nonumber \\
& \le  Q(m)\left(1- 2\lambda\right) + N\cdot d \cdot \mathbbm{1}_{\mathcal{A}_m}.
\end{align}
So using Lemma~\ref{Amthm} to bound $\mathbb{P}(\mathcal{A}_m)$, by \eqref{silly} and the law of iterated expectation we have
\begin{align*}
\mathbb{E}(Q(m+1)) &= \mathbb{E}(\mathbb{E}(Q(m+1)\mid \mathcal{F}_m')) \\
&\leq \left(1- 2\lambda\right)\mathbb{E}(Q(m)) + N^{-\sqrt{\log N}}\cdot  N \cdot d
\end{align*}
Applying our induction hypothesis gives that this is at most
\[\left(1- 2\lambda\right)2\left(1-\lambda\right)^{m}\zeta\cdot N  + N^{-\Omega(1)}.\]
Since $(1-\lambda)^{m}\geq (1-\lambda)^{M_2}=N^{-\Theta(1)}$, the above expression is at least
\[2\left(1-\lambda\right)^{m+1}\zeta\cdot N\]
which completes the proof of the claim.
\end{proof}

To complete the proof we show that~\ref{notManyInfected} also holds with probability $1-o(1)$. By Proposition~\ref{Ibound}, with probability $1 - N^{-\Omega\left(\sqrt{\log N}\right)}$ the number of vertices infected during the first phase is $O\left(\frac{N\cdot\log(N)}{d^{1/(r-1)}}\right)$. 

By (\ref{ws}) of Lemma~\ref{subtrack} and Lemma~\ref{Amthm} and letting $K$ be large, with probability at least $1 - N^{-\sqrt{\log(N)}}$ for every $0 \le m \le M_2$ and $v \in V(\mathcal{H})\setminus I(m)$,  we have $Q_v(m) = \plog(N)$. Using this, with high probability the number of edges sampled in each round is $N \cdot \plog(N)$. So as $M_2 = \plog(d)$ by \eqref{M2againsub}, the expected number of vertices infected during phase two is $\frac{N\cdot\plog(N)}{d^{1/(r-1)}}$. Applying the Chernoff bound, by \eqref{Nnotsmall} we get that with probability at least $1-N^{-100}$ there are at most $\frac{N\cdot\plog(N)}{d^{1/(r-1)}}$ vertices infected during phase two. So overall, with high probability there are at most $\frac{N\cdot\plog(N)}{d^{1/(r-1)}}=o(N)$ infected vertices when the process terminates. This completes the proof.
\end{proof} 

The remainder of the subsection is devoted to proving Lemma~\ref{subtrack} and Lemma~\ref{Amthm}.

\subsection{Defining \texorpdfstring{$\boldsymbol{\mathcal{A}_m}$}{Am} and Proof of Lemma~\ref{subtrack}}
 We begin by formally defining the events $\mathcal{A}_m$, for $0 \le m \le M_2$. It will be convenient for the proof of Lemma~\ref{Amthm} to define $\mathcal{A}_0$ separately. Recall the definitions of $\lambda$ and $\chi(i,j)$ from \eqref{T0q'} and Definition~\ref{zetaDef}, respectively. 
 
 \begin{defn}\label{time0}
 Let $\mathcal{A}_0$ be the event (in $\Omega'$, which was defined in Subsection~\ref{sec:probspace}) that for some $v \in V(\mathcal{H})$ or $S \subseteq V(\mathcal{H})$ one of the following statements fails to hold.
 \begin{enumerate}
 \renewcommand{\theenumi}{A0.\arabic{enumi}}
 \renewcommand{\labelenumi}{(\theenumi)}
 \item\label{yi0time0} For all $0 \le i \le r-2$,
 $$Y^{i,0}_v(0)\leq \left(\chi(i,0) - \frac{\lambda}{\alpha}\right)d^{1-\frac{i}{r-1}}.$$
 
 \item\label{yijtime0} For $0 \le i \le r-2$ and $1 \le j \le r-1-i$,
 $$Y_v^{i,j}(0) < \chi(i,j)d^{1-\frac{i}{r-1}}.$$
 \item\label{wtime0} For $1 \le i \le r$,
 $$W^i_S(0) \le \log^{r^3(r-i)}(d).$$
 \item\label{xtime0} For any secondary configuration $X=(\mathcal{F},R,D)$,
 $$X_S(0) \leq \log^{2|D|r^4}(d)\cdot d^{\frac{|V(\mathcal{F})| - |R|-|D|}{r-1}}\log^{-3K/5}(d).$$
 \end{enumerate}
 \end{defn}
 
 We remark that the case $i=0$ of (\ref{yi0}) in Lemma~\ref{subtrack} always holds trivially as $\Delta(\mathcal{H})\leq d$. This is reflected below in our definition of $\mathcal{A}_{m}$. 
 
 \begin{defn}\label{Amevent}
 For $1 \le m \le M_2$, let $\mathcal{A}_m$ be the event (in $\Omega'$, which was defined in Subsection~\ref{sec:probspace}) that either $\mathcal{A}_0$ occurs, or there exists $1 \le \ell \le m$ such that, for some $v \in V(\mathcal{H})$ or $S \subseteq V(\mathcal{H})$ one of the following statements fails to hold:
 \begin{enumerate}
 \renewcommand{\theenumi}{A.\arabic{enumi}}
 \renewcommand{\labelenumi}{(\theenumi)}
 \item\label{yi0prime}  For $1\leq i\leq r-2$,
 $$Y_v^{i,0}(\ell) - Y_v^{i,0}(\ell-1)\leq \psi(\ell-1)\left(\frac{\lambda^2}{2\alpha}\right)d^{1-\frac{i}{r-1}}.$$
 \item\label{zijprime} For $0 \le i \le r-2$ and $1 \le j \le r-1-i$, 
 $$Y_v^{i,j}(\ell) \le \psi(\ell)\chi(i,j)d^{1 - \frac{i}{r-1}}.$$
 \item\label{wprime} For $1\leq i\leq r$,
 $$W_S^i(\ell)-W_S^i(\ell-1)\leq \log^{r^3(r-i)}(d).$$
 \item \label{xprime} For every secondary configuration $X=(\mathcal{F},R,D)$,
 $$X_S(\ell)-X_S(\ell-1) \leq \log^{2|D|r^7}(d)\cdot d^{\frac{|V(\mathcal{F})| - |R|-|D|}{r-1}}\log^{-3K/5}(d).$$
 \end{enumerate}
 \end{defn}
 
 We now present the proof of Lemma~\ref{subtrack}.
 \begin{proof}[Proof of Lemma~\ref{subtrack}]
 
We first show that if $\omega \notin \mathcal{A}_m$, then (\ref{yi0})-(\ref{xs}) hold. When $m=0$ this follows by definition of $\mathcal{A}_0$ as $\psi(0) = 1$. Now consider $m > 0$. 
 
  Note that (\ref{zijprime}) is precisely the same statement as property (\ref{zij}). Observe in addition that if $\omega \notin \mathcal{A}_m$, then clearly (\ref{ws}) and (\ref{xs}) hold. Also, if $\omega \notin \mathcal{A}_m$, then using the definition of $\psi$ (see \eqref{psidef}),
  \begin{align*}
Y_v^{i,0}(m) &= Y_v^{i,0}(0) + \sum_{\ell=1}^{m}\left(Y_v^{i,0}(\ell) - Y_v^{i,0}(\ell-1)\right)\\
 &\leq \left(\chi(i,0) - \frac{\lambda}{\alpha} + \left(\frac{\lambda^2}{2\alpha}\right)\sum_{\ell=1}^{M_2}\psi(\ell-1)\right)d^{1-\frac{i}{r-1}}\\
 &\leq\left(\chi(i,0) - \frac{\lambda}{\alpha}+ \left(\frac{\lambda^2}{2\alpha}\right)\left(\sum_{\ell=0}^{\infty}\left(1-\lambda\right)^{\ell} + \sum_{\ell=1}^{M_2}\log^{-10}(N)\right)\right)d^{1-\frac{i}{r-1}}\\
 &\leq \left(\chi(i,0)-\frac{\lambda}{\alpha} +\left(\frac{\lambda^2}{2\alpha}\right)\left(\frac{1}{\lambda} + o(1)\right)\right)\\
 &\leq \chi(i,0)d^{1-\frac{i}{r-1}},
 \end{align*}
 since $M_2=O\left(\log N\right)$. 
 Hence, $\omega \notin \mathcal{A}_m$ implies that (\ref{yi0}) holds. 
 
 We now show that if (\ref{yi0})-(\ref{xs}) hold at time $m$, then so does (\ref{Zbound}). By Observation~\ref{Zsecondary}, each member of $Z_v^{i,j}(m)$ is either a member of $Y_v^{i,j}(m)$ or consists of a copy $\mathcal{F}'$ of a secondary configuration with $r-1-i$ neutral vertices and a set of copies of $W^1$ rooted at vertices of $\mathcal{F}'$. By (\ref{ws}), (\ref{xs}) and the fact that $m\leq M_2=O\left(\log N\right)$ (from \eqref{Nisd} and \eqref{M2againsub}), the number of members of this type is at most 
  \[\log^{O(1)}(d)\cdot d^{1-\frac{i}{r-1}}\cdot \log^{-3K/5}(d),\]
   This is $o\left(\psi(m)d^{1-\frac{i}{r-1}}\right)$, provided that $K$ is sufficiently large. The result now follows by the bound on $Y_v^{i,j}(m)$ given by (\ref{zij}).
\end{proof}

\subsection{Proof of Lemma~\ref{Amthm}}
Given a point $\omega \in \mathcal{A}_{M_2}\setminus \mathcal{A}_0$, let
$$\mathcal{J} = \mathcal{J}(\omega) := \min \{i : \omega \in \mathcal{A}_i\}.$$
Define $\mathcal{Y}^0$ to be the set of $\omega \in \mathcal{A}_{M_2}\setminus \mathcal{A}_0$ such that (\ref{yi0prime}) is violated at time $\mathcal{J}(\omega)$ for some $v \in V(\mathcal{H})$ (the superscript 0 represents that $j=0$ in the configurations we consider in (\ref{yi0prime})). Similarly, define $\mathcal{Y}^{>0}, \mathcal{W}$ and $\mathcal{X}$ respectively to be the events that (\ref{zijprime}), (\ref{wprime}) and (\ref{xprime}) are violated at time $\mathcal{J}(\omega)$. 

So the event $\mathcal{A}_{M_2}$ is contained within the events $\mathcal{A}_{0}$, and $\mathcal{Y}^0\cup \mathcal{Y}^{>0}\cup\mathcal{W} \cup \mathcal{X}$. First we bound the probability of $\mathcal{A}_0$ occurring. The following Lemma is an easy consequence of Lemma~\ref{phaseoneevent}.
\begin{lem}\label{phase2zero}
$$\mathbb{P}(\mathcal{A}_0) \le N^{-2\sqrt{\log(N)}}.$$
\end{lem}

 \begin{proof}
 First suppose $\omega \notin \mathcal{A}_0$. By definition of $\mathcal{B}_M$ (see Definition~\ref{bmdef}), with probability at least $1 - N^{-2\sqrt{\log(N)}}$ the bounds (\ref{Ystate}), (\ref{Xstate}) and (\ref{Wstate}) hold for all $v \in V(\mathcal{H})$ and $S \subseteq V(\mathcal{H})$ when $\ell = M$.
 
 When (\ref{Ystate}) holds, for $0\leq i\leq r-2$ and $0\leq j\leq r-1-i$, we have  
 \[Y_v^{i,j}(0) \le (1 + \epsilon(T))\binom{r-1}{i}\binom{r-1-i}{j}(c + \alpha T)^i \gamma(T)^jd^{1 - \frac{i}{r-1}}\]
 By \eqref{Tq'} and the fact that $(1 + \epsilon(T))\gamma(T)<\zeta$ (by \eqref{Mdef}), the above inequality implies that
 \begin{equation}
 \label{Yinitialsub}
 \begin{gathered}
 Y_v^{i,j}(0)\leq \binom{r-1}{i}\binom{r-1-i}{j}\left(\frac{1-4\lambda}{\alpha(r-1)}\right)^{\frac{i}{r-2}}\zeta^j d^{1 - \frac{i}{r-1}}\\
 \leq \left(\chi(i,0)-\frac{\lambda}{\alpha}\right)\binom{r-1-i}{j}\zeta^jd^{1-\frac{i}{r-1}}.
 \end{gathered}
 \end{equation}
 For $0\leq i\leq r-2$ and $j=0$, \eqref{Yinitialsub} implies that
 $$Y^{i,0}_v(0)\leq \left(\chi(i,0) - \frac{\lambda}{\alpha}\right)d^{1-\frac{i}{r-1}},$$
 as required for (\ref{yi0time0}).
 Also, if $i=r-2$ and $j=1$, \eqref{Yinitialsub} implies that
 \begin{equation}\label{yr-20time0}
 Y^{r-2,1}_v(0)\leq \chi(r-2,0)\cdot \zeta \cdot d^{\frac{1}{r-1}} < \chi(r-2,1)d^{\frac{1}{r-1}},
 \end{equation}
 thus proving (\ref{yijtime0}) when $i = r-2$. Now, for $0\leq i\leq r-3$ and $1\leq j\leq r-1-i$, using \eqref{Yinitialsub} and the fact that $\zeta<\chi(r-2,1)$, gives
 $$Y_v^{i,j}(0)\leq \chi(i,0)\binom{r-1-i}{j}\zeta^j \cdot d^{1-\frac{i}{r-1}} < \chi(i,j)d^{1-\frac{i}{r-1}},$$
 completing the proof of (\ref{yijtime0}). For $1\leq i\leq r$ and $S\subseteq V(\mathcal{H})$, when (\ref{Wstate}) holds at $\ell = M$,
 $$W^i_S(0) \le \log^{r^3(r-i)}(d),$$
  thus proving (\ref{wtime0}). Finally, for any secondary configuration $X=(\mathcal{F},R,D)$ and $S\subseteq V(\mathcal{H})$, when (\ref{Xstate}) holds at $\ell = M$,
 $$X_S(0) \leq \log^{2|D|r^4}(d)\cdot d^{\frac{|V(\mathcal{F})| - |R|-|D|}{r-1}}\log^{-3K/5}(d),$$
 as required for (\ref{xtime0}).
 \end{proof}

Our goal now is to show that the probability of each of $\mathcal{Y}^0$,  $\mathcal{Y}^{>0}$, $\mathcal{W}$ and $\mathcal{X}$ occurring is $N^{-10\sqrt{\log N}}$, from which Lemma~\ref{Amthm} will follow. This will be done in Propositions~\ref{W2prop},~\ref{X2prop},~\ref{Y02prop} and~\ref{Yij2prop} to come.

Given a (general) configuration $X$ that we care about controlling, we wish to apply Corollary~\ref{vunew} to bound the probability of $X_S(m+1) - X_S(m)$ being too large for each $S \subseteq V(\mathcal{H})$. It will be helpful to introduce some general framework, which will aid us in our application of Corollary~\ref{vunew}. Before stating the next lemma, we require quite a few technical definitions. Let us briefly motivate the definitions before stating them formally.

For a fixed configuration $X = (\mathcal{F},R,D)$ and non-empty $U \subseteq D$, we will define a family $X^U$ of configurations, such that each configuration in the family is created by changing the set $U$ of marked vertices of $X$ into neutral vertices and, for each $u \in U$, adding a hyperedge $e_u$ containing $u$ such that the vertices in $e_u \setminus \{u\}$ are all marked. The family $X^U$ will be helpful when bounding how many new copies of $X$ are made at some time step. Indeed, a new copy of $X$ is made from a copy of a configuration in $X^U$ when the open hyperedges rooted at vertices of $U$ are all successfully sampled in some time step. 

We will also define a configuration $X^{U,U'}$ created by taking some set $U$ of marked vertices of $X$, turning some subset $U' \subseteq U$ into roots and turning $U \setminus U'$ into neutral vertices. This will be used to bound $\partial_A f$ (where $f$ will be an upper bound on $X_S(m+1) - X_S(m)$ and $A \subseteq V(\mathcal{H})$) in the application of Corollary~\ref{vunew}. As it turns out, we will always be able to express bounds on the number of copies of some $X' \in X^U$ or $X^{U,U'}$ in terms of copies of configurations we are keeping control over. 

For a configuration $X = (\mathcal{F},R,D)$, call a hyperedge of $\mathcal{F}$ with $r-1$ vertices in $D$ \emph{unstable}. If $\mathcal{F}$ contains an unstable hyperedge then every copy of $X$ in $\mathcal{H}(m)$ is destroyed in the $(m+1)$th time step, as at each time step every open hyperedge is sampled and deleted from our hypergraph. So in particular, a new copy of $X$ can only be made from some $X^U$ such that $U$ intersects every unstable hyperedge of $\mathcal{F}$. Call such a set $U$ \emph{fruitful}.

We apologise that the following set of definitions are fairly technical. But the introduction of these concepts and Lemma~\ref{easycalc} will greatly simplify and clarify the calculations that are to come in the proof of Lemma~\ref{subtrack}.

\begin{defn}\label{xuu}
Let $X =(\mathcal{F},R,D)$ be a configuration and let $\emptyset \not= U \subseteq D$ be fruitful and $U' \subseteq U$. Define $X^{U,U'}$ to be the configuration $X =(\mathcal{F},R \cup U',D\setminus U)$. 

Say that $U_1,U_2 \subseteq D$ are \emph{non-isomorphic} if the configurations $(\mathcal{F},R,D \setminus U_1)$ and $(\mathcal{F}, R, D \setminus U_2)$ are not isomorphic. Define $\mathcal{U}(X)$ to be a maximal collection of
pairwise non-isomorphic fruitful subsets of $D$.

For $k, a \ge 0$, define 
$$X^{k,a}:= \bigcup_{\substack{ U \in \mathcal{U}(X)\\ U' \subseteq U \\ |U| - |U'|=k\\ |U'| = a}}X^{U,U'}.$$ 
Also, define $X^U$ to be the set of all configurations $X' = (\mathcal{F}',R',D')$ such that
\begin{itemize}
\item[(a)] $\mathcal{F}' \supseteq \mathcal{F}$, $R' = R$ and $D' \cap D = D \setminus U$ (so the vertices of $U$ are neutral here);
\item[(b)] $E(\mathcal{F}')\setminus E(\mathcal{F}) = \{e_u: u \in U\}$, where for each $u \in U$, $e_u$ is an unstable hyperedge of $\mathcal{F}'$ containing $u$. (So $D' := D \setminus U \cup \bigcup_{u \in U}e_u\setminus \{u\}$.)
\end{itemize}
See Figure~\ref{xuufig} for an example of a configuration in $X^U$ and a configuration $X^{U,U'}$. 

\end{defn}

\begin{figure}[htbp]
\centering
\includegraphics[width=1\textwidth]{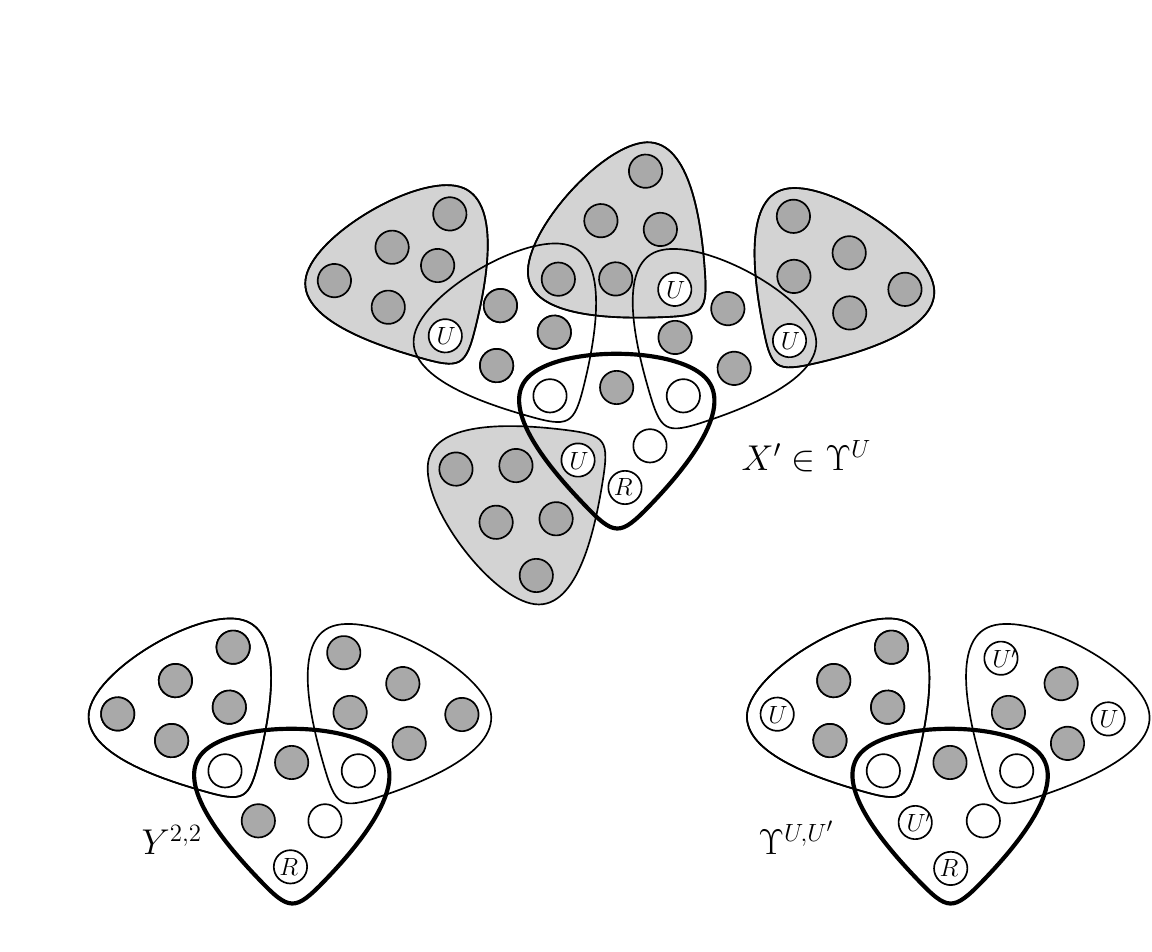}
\caption{Here $\conf:= Y^{2,2}$ and $r=6$. For a particular choice of fruitful $U \subseteq D$ and $U'\subseteq U$, we have an example of some configuration $ X' \in \conf^{U}$ and a configuration $\conf^{U,U'}$. The vertices of $U'$ are labelled with a $U'$ and the vertices of $U \setminus U'$ are labelled with a $U$. The root is labelled by $R$, the marked vertices are dark grey, the neutral vertices are white and the central hyperedge is drawn with a thick outline. The hyperedges $\bigcup_{u \in U}e_u$ are shaded light grey. Observe that if each of these hyperedges in a copy of $X'$ is successfully sampled in the $(m+1)$th step, then this creates a new copy of $Y^{2,2}$ in $\mathcal{H}(m+1)$.}
\label{xuufig}
\end{figure}

 Throughout the rest of the section we will condition on $\mathcal{F}_m'$ and assume that $\omega \notin \mathcal{A}_m$. We may not always write explicitly that we are assuming this and conditioning on $\mathcal{F}_m'$. We will use the following random variables throughout. Given $w\in V(\mathcal{H})$, let $\xi_w$ be the Bernoulli random variable which is equal to one if $w\in I(m+1)\setminus I(m)$ and zero otherwise. Also, for $e\in Q(m)$, we let $\xi_e$ be the Bernoulli random variable which is equal to one if $e$ is successfully sampled and zero otherwise. Clearly, for each $v\notin I(m)$, we have that $\xi_v$ is equal to one if and only if $\xi_e=1$ for some $e\in Q_v(m)$. From this (and as we are conditioning on $\mathcal{F}_m'$), we have that $\xi_u$ is independent of $\xi_w$ for $u,w$ distinct vertices of $V(\mathcal{H})\setminus I(m)$. Also, we have
 \begin{equation}
 \label{vtoe}
 \xi_v \leq \sum_{e\in Q_v(m)}\xi_e.\end{equation}

\begin{lem}
\label{easycalc}
Let $X =(\mathcal{F},R,D)$ be a configuration and let $S$ be a subset of $V(\mathcal{H})$ such that $|S|= |R|$. If (\ref{ws}) holds at time $m$, then there exists a polynomial $f(x_w:w \in V(\mathcal{H}))$ of degree at most $|D|$ where no variable has an exponent greater than 1 such that:
\begin{enumerate}[(i)]
\item\label{XSdiffthing} $X_S(m+1) - X_S(m) \le |X_S(m+1)\setminus X_S(m)| \le f(\xi_w:w \in V(\mathcal{H}))$
\item\label{partialAthing} For any set $A \subseteq V(\mathcal{H})$,
$$\mathbb{E}(\partial_A f(\xi_w:w \in V(\mathcal{H}))\mid \mathcal{F}_m') \le \sum_{k=0}^{|D| - |A|}\sum_{X' \in X^{k,|A|}}X'_{S \cup A}(m) (\log^{r^4}(d)q)^k .$$
\item \label{expThing}$$\mathbb{E}(f(\xi_w:w \in V(\mathcal{H}))\mid \mathcal{F}_m') \le \sum_{k=1}^{|D|}\sum_{\substack{U \in \mathcal{U}(X)\\ |U| = k}}\sum_{X' \in X^U}X'_{S}(m)q^{k}.$$
\end{enumerate}
\end{lem} 

We remark that \ref{partialAthing} does gives a bound on everything we need to apply Corollary~\ref{vunew}. However, when dealing with $W^1$ and the $Y$ configurations we need a more careful bound on the expected change and this is why we have \ref{expThing}.

\begin{proof}[Proof of Lemma~\ref{easycalc}]
If (\ref{ws}) holds at time $m$, as $m = O\left(\log(d)\right)$ (by \eqref{M2againsub} and \eqref{Nisd}), by \eqref{vtoe} we have 
\begin{equation}
\label{xiBound}
\mathbb{E}(\xi_v\mid \mathcal{F}_m')\leq q\cdot (m+1)\log^{r^3(r-1)}(d) \le q \log^{r^4}(d).
\end{equation}

Let us consider how a copy of $X$ rooted at $S$ which is present in the $(m+1)$th step, but not the $m$th step is created. Such a copy can only come from a subhypergraph $\mathcal{F}' \subseteq \mathcal{H}$ containing $S$ that is like a copy of $X$ missing some infections.  $\mathcal{F}'$ will become a new copy of $X$ when it gains these missing infections. 

More formally, a new copy comes from a subhypergraph $\mathcal{F}' \subseteq \mathcal{H}(m)$, a non-empty fruitful set $U \subseteq D$ and subsets $D', V' \subseteq V(\mathcal{F}')$ such that: 
\begin{itemize}
\item[(1)] there exists an isomorphism $\phi$ from $\mathcal{F}$ to $\mathcal{F}'$ with $\phi(R)=S$, $\phi(U) = V'$ and $\phi(D\setminus U)=D'$,
\item[(2)] $V' \cap (S \cup I(m)) = \emptyset$ and $D' \subseteq I(m)$,
\item[(3)] every vertex $v \in V'$ becomes infected at the $(m+1)$th step. 
\end{itemize}

Let $Z(X)$ be the set of triples $(\mathcal{F}', V',D')$ that satisfy (1), (2) and (3). Then the variable $X_S(m+1) - X_S(m)$ is bounded above by $f(\xi_w: w\in V(\mathcal{H}))$, where
\begin{equation}
\label{fdef}
f(x_w: w\in V(\mathcal{H})) := \sum_{(\mathcal{F}',V',D') \in Z(X)} \left(\prod_{v \in V'} x_v\right).
\end{equation}
As $V' \subseteq D$, $f$ has degree at most $|D|$. Observing that no variable in $f$ has an exponent greater than 1 completes the proof of \ref{XSdiffthing}.

By definition, $V'$ is a subset of $V(\mathcal{F}') \setminus (S \cup I(m))$. So if $A$ contains an element of $(S \cup I(m))$, then $\partial_A f = 0$. So suppose that $A$ is a subset of $V(\mathcal{H}) \setminus (S \cup I(m))$, Then for any such $A$, we have
\begin{equation}\label{partialgen}
\partial_A f = \sum_{\substack{(\mathcal{F}',V',D') \in Z(X)\\ A \subseteq V'}} \left(\prod_{v \in V'\setminus A} x_v\right).
\end{equation}
We can partition the set $Z(X)$ by the cardinality of each $V'$. So for $0\le k \le |D|-|A|$, let 
$$Z^k(X):= \{(\mathcal{F}',V', D') \in Z(X): |V'| - |A|\text{ } = k\}.$$ 
So we can rewrite \eqref{partialgen} as follows.
\begin{equation}
\label{partialk}
\partial_A f = \sum_{k=0}^{|D|-|A|}\sum_{\substack{(\mathcal{F}',V',D') \in  Z^k(X)\\ A \subseteq V'}}\left(\prod_{v \in V'\setminus A} x_v\right).
\end{equation}
Now recall Definition~\ref{xuu} and observe that any $(\mathcal{F}',V',D') \in Z^k(X)$ such that $A \subseteq V'$ is in fact a copy of $X^{U,U'}$ rooted at $S \cup A$ for some $U' \subseteq U \subseteq D$ with $|U|= |V'|$, $|U'|= |A|$ and $|U| - |U'|= k$. So from \eqref{partialk} and \eqref{xiBound}, as the variables $\xi_w$ are independent (because $V'\setminus A \subseteq V(\mathcal{H}) \setminus I(m)$ and we are conditioning on $\mathcal{F}_m'$) we have
\[
\mathbb{E}(\partial_A f(\xi_w: w\in V(\mathcal{H}))\mid \mathcal{F}_m' ) \le  \sum_{k=0}^{|D| - |A|}\sum_{X' \in  X^{k,|A|}}X'_{S \cup A}(m)(\log^{r^4}(d)q)^k,
\]
as required for \ref{partialAthing}.

Now substituting \eqref{vtoe} into \eqref{fdef} gives
$$f(\xi_w: w\in V(\mathcal{H})) \le g(\xi_e: e\in E(\mathcal{H})) ):= \sum_{(\mathcal{F}',V',D') \in Z(X)} \left(\prod_{v \in V'}\sum_{e \in Q_v(m)} \xi_e\right).$$
Notice that the random variables on the right hand side are now associated to hyperedges, not vertices. 

Given a fixed $(\mathcal{F}',V',D') \in Z(X)$, let $\mathcal{G}'$ be a subhypergraph of $\mathcal{H}$ containing $\mathcal{F}'$ such that $E(\mathcal{G}')\setminus E(\mathcal{F}') = \{e_v: v \in V'\}$, where $e_v \in Q_v(m)$ for all $v \in V'$. Recall Definition~\ref{xuu} again. By definition of $Z(X)$, there exists some non-empty fruitful $U \subseteq D$ such that $\mathcal{G}'$ is a copy of some $X' \in X^U$ rooted at $S$ in $\mathcal{H}(m)$ .

As, for $e,e'$ distinct hyperedges in $\mathcal{H}(m)$, we have $\xi_e$ is independent of $\xi_{e'}$ and 
$$\mathbb{E}(g(\xi_e: e\in E(\mathcal{H}))\mid \mathcal{F}_m') \le \sum_{U \in \mathcal{U}(X)}\sum_{X' \in X^U}X'_{S}(m)q^{|U|}.$$
By considering the cardinality of $U$, this can be rewritten as
$$\mathbb{E}(g(\xi_e: e\in E(\mathcal{H}))\mid \mathcal{F}_m') \le \sum_{k=1}^{|D|}\sum_{\substack{U \in \mathcal{U}(X)\\ |U| = k}}\sum_{X' \in X^U}X'_{S}(m)q^{k},$$
as required for \ref{expThing}.

\end{proof}
Now we have developed our tools, we are all set to prove Lemma~\ref{Amevent}. As usual we begin with the proof of the bound on $W_S^i(m)$, as it is the simplest case.

\begin{prop}\label{W2prop}
$$\mathbb{P}(\mathcal{W}) \le N^{-10\sqrt{\log N}}.$$
\end{prop}

\begin{proof}
For $0 \le m \le M_2 -1$ and $S \subseteq V(\mathcal{H})$, let $\mathcal{W}(S,m+1)$ be the set of $\omega$ in $\mathcal{A}_{M_2}$ such that $\mathcal{J}(\omega) = m+1$ and, for some $1\leq i\leq r$,
$$W^i_S(m+1) - W^i_S(m) > \log^{r^3(r-i)}(d).$$
It follows that
$$\mathcal{W} = \bigcup_{\substack{0 \le m \le M_2 -1\\S \subseteq V(\mathcal{H})}} \mathcal{W}(S,m+1).$$
We will prove that for $0 \le m \le M_2 -1$,
$$\mathbb{P}(\mathcal{W}(S,m+1)\given \mathcal{F}_m') \le N^{-20\sqrt{\log N}}.$$
The proof of the proposition will then follow by taking the union bound over all $0 \le m \le M_2-1$ and $S \subseteq V(\mathcal{H})$ of size at most $r-1$. 

Conditioning on $\mathcal{F}_m'$ for $m \ge 0$, the events $\mathcal{A}_m$ and $\mathcal{W}(S,m+1)$ are disjoint and so we assume that $\omega \notin \mathcal{A}_{m}$. Fix $S \subseteq V(\mathcal{H})$. We wish to apply Lemma~\ref{easycalc} along with Corollary~\ref{vunew} to obtain the required bound on $\mathbb{P}(\mathcal{W}(S,m+1))$. As $\omega \notin \mathcal{A}_{m}$, (\ref{ws}) holds at time $m$ and we may apply Lemma~\ref{easycalc}. So let $\conf := W^{i}$ and let $\tilde{W}:= f(\xi_w: w\in V(\mathcal{H}))$, where $f(x_w: w\in V(\mathcal{H}))$ is the polynomial of degree at most $r-i$ obtained by applying Lemma~\ref{easycalc} to $\conf$.

We wish to apply Corollary~\ref{vunew} with 
$$\tau:= 2M_2\log^{r^3(r-i-1)}(d)$$
and $\mathcal{E}_0:= \log^{2r}(N)$ to obtain an upper bound on $\tilde{W}$ which holds with high probability. In order to apply Corollary~\ref{vunew} we must bound $\mathbb{E}_j(\tilde{W}\mid \mathcal{F}_{m}')$ for $0\leq j\leq r-i$. 
\begin{claim}
\label{Wclaim2}
Let $\omega \notin \mathcal{A}_m$. For $0 \le a \le r-i$, we have
$$\mathbb{E}_a(\tilde{W}\mid \mathcal{F}_{m}')\leq \tau.$$
\end{claim}

\begin{proof}
By Lemma~\ref{easycalc} \ref{partialAthing}, for any set $A \subseteq V(\mathcal{H})$ with $|A|=a$,

\begin{equation}
\label{sloppyBound}
\mathbb{E}(\partial_A f(\xi_w: w\in V(\mathcal{H}))\mid \mathcal{F}_{m}') \le \sum_{k=0}^{r -i- a}\sum_{\conf' \in \conf^{k,a}}\conf_{S \cup A}'(m) (\plog (d)q)^k.
\end{equation}
So let us evaluate this expression. Recalling Definition~\ref{xuu} we see that $\conf^{k,a}$ contains only one configuration, the configuration $\overbar{\conf}^{k,a} = (\mathcal{F}',R',D')$ where $\mathcal{F}'$ is a set of $r$ vertices contained in a single hyperedge with $i + a$ roots and $k$ neutral vertices (and therefore $r - i - (k+a)$ marked vertices). 

We will break the cases up by the value of $a + i$. First suppose $a + i=r$. Here $k=0$ and $\mathbb{E}(\partial_A\tilde{W}) \le 1$ since $\mathcal{H}$ contains no hyperedge with multiplicity greater than one. 

Now suppose $a \ge 1$ and $a +i < r$, so $2 \le a + i < r$. When $k \ge 1$, $\overbar{\conf}^{k,a}$ is a secondary configuration, and as $\omega \notin \mathcal{A}_{m}$, by (\ref{xs}) we have
\begin{equation}
\label{Xbar}
\overbar{\conf}^{k,a}_{S \cup A}(m) \le (m+1)\plog (d) d^{\frac{k}{r-1}}\log^{-3K/5}(d).
\end{equation} 
When $k=0$, as $\omega \notin \mathcal{A}_{m}$, by (\ref{ws}) we have
\begin{equation}
\label{Xbar2}
\overbar{\conf}^{0,a}_{S \cup A}(m) = W_{S \cup A}^{i + a}(m) \le (m+1)\log^{r^3(r-i-a)}(d).
\end{equation}
Using \eqref{Xbar} and \eqref{Xbar2} to evaluate \eqref{sloppyBound} gives that, when $\omega \notin \mathcal{A}_m$ and $a \ge 1$, the expectation of $\partial_A f(\xi_v: v\in V(\mathcal{H}))$ conditioned on $\mathcal{F}_{m}'$ is at most
\begin{align*}
\label{Aibig}
\sum_{k=0}^{r - i-a}& \overbar{\conf}_{S \cup A}^{k,a}(m)(\plog (d)q)^k 
\le (m+1)\log^{r^3(r-i-a)}(d) +  \sum_{k=1}^{r-i-a}\plog(d)\cdot \log^{-\frac{3K}{5}}(d).
\end{align*}
Provided that $K$ is large enough, this is at most $\tau$ when $a \ge 1$. This concludes the argument in the case when $a \ge 1$.
 
It remains to bound $\mathbb{E}_a(\tilde{W}\mid \mathcal{F}_m')$ when $a =0$ and $i < r$. We need to be more careful here. By Lemma~\ref{easycalc} \ref{expThing}, 
\begin{equation}\label{w0phase2}
\mathbb{E}(\tilde{W}\mid \mathcal{F}_m') \le \sum_{k=1}^{r-i}\sum_{\substack{U \in \mathcal{U}(\conf)\\ |U| = k}}\sum_{\conf' \in \conf^U}\conf_{S}'(m)q^{k}.
\end{equation}
Recall Definition~\ref{xuu} and observe that, for each $k \ge 1$, there is precisely one $U \in \mathcal{U}$ with $|U|= k$. 

First consider when $i=1$. We see that when $|U| = k$, we have $\conf^{U} = Z^{r-1-k,k}$. So as $\omega \notin \mathcal{A}_{m}$, applying (\ref{Zbound}) to evaluate \eqref{w0phase2} gives that the expectation of $\tilde{W}$ given $\mathcal{F}_{m}'$ is at most
\[\sum_{k=1}^{r-1}Z_v^{r-1-k,k}(m)q^k\leq \left(1+\frac{\lambda^2}{r^{3r}}\right)\psi(m)\sum_{k=1}^{r-1}\chi(r-1-k,k)\alpha^k = O(1) \le \tau.\]
When $1 < i < r$ and $|U| = k$, we see each configuration  $\conf\in\conf^U$ consists of:
\begin{enumerate}[(i)]
\item a secondary configuration $X = (\mathcal{G},R',D')$, where $\mathcal{G}$ consists of one hyperedge containing $i$ roots and $k$ neutral vertices, and in addition
\item a collection of $k$ unstable hyperedges rooted at neutral vertices of $\mathcal{G}$ that are not in the non-central hyperedge. 
\end{enumerate}
So in this case, using (\ref{ws}) and (\ref{xs}) as $\omega \notin \mathcal{A}_m$ gives
\begin{align*}
\conf'_S(m) &\le X_S(m)\left((m+1)\log^{r^3(r-1)}(d)\right)^{k} \\
& \le (m+1)^{k+1} \log^{2|D|r^r}(d) \cdot d^{\frac{k}{r-1}} \log^{-3K/5}(d) \log^{kr^3(r-1)}(d)\\
&= \log^{O(1)}(d)\cdot \log^{-3K/5}(d) \cdot d^{\frac{k}{r-1}}.
\end{align*}
Using this to evaluate \eqref{w0phase2} gives that the expectation of $\tilde{W}$ given $\mathcal{F}_{m}'$ is at most
$$\sum_{k=1}^{r-i} \log^{O(1)}(d) \cdot \log^{-3K/5}(d) \cdot d^{\frac{k}{r-1}} q^k=o(1),$$
for large enough $K$. 
This concludes the proof of the claim.
\end{proof}

Recall that $\mathcal{E}_0:= \log^{2r}(N)$. As $2r + 1 < r^3$, by Claim~\ref{Wclaim2} we have $$\mathbb{E}(\tilde{W}\given \mathcal{F}_m') \le \tau  = o\left(\log^{r^3(r-i)}(d)\right).$$ We also have
\[\tau \log^{r-i}(N) \sqrt{\mathscr{E}_0} = o\left(\log^{r^3(r-i)}(d)\right),\]
Thus applying Corollary~\ref{vunew} gives that
$$\mathbb{P}(\mathcal{W}(S,m+1)\given \mathcal{F}_m') \le N^{-20\sqrt{\log(N)}},$$
as required.
\end{proof}

\begin{prop}\label{X2prop}
$$\mathbb{P}(\mathcal{X}) \le N^{-10\sqrt{\log N}}.$$
\end{prop}

\begin{proof}
For a secondary configuration $X = (\mathcal{F},R,D)$, $S \subseteq V(\mathcal{H})$ with $|S|= |R|$ and $0 \le m \le M_2-1$, let $\mathcal{X}(X,S,m+1)$ be the set of $\omega$ in $\mathcal{A}_{M_2}$ such that $\mathcal{J}(\omega) = m+1$ and  
$$X_S(m+1) - X_S(m) > \log^{2|D|r^7}(d) \cdot d^{\frac{|V(\mathcal{F})| - |R| - |D|}{r-1}}\log^{-3K/5}(d).$$
It follows that
$$\mathcal{X} = \bigcup_{X,S,m+1} \mathcal{X}(X,S,m+1).$$
We will prove that for $0 \le m \le M_2 -1$,
$$\mathbb{P}(\mathcal{X}(X,S,m+1)\given \mathcal{F}_m') \le N^{-20\sqrt{\log N}}.$$
The proof of the proposition will then follow by taking the union bound over all choices of $X$, $S$ and $m$. 

Conditioning on $\mathcal{F}_m'$ for $m \ge 0$, the events $\mathcal{A}_m$ and $\mathcal{X}(X,S,m+1)$ are disjoint and so we assume that $\omega \notin \mathcal{A}_{m}$.  Fix $S \subseteq V(\mathcal{H})$ and a secondary configuration $X = (\mathcal{F},R,D)$. We wish to apply Lemma~\ref{easycalc} along with Corollary~\ref{vunew} to obtain the required bound on $\mathbb{P}(\mathcal{X}(X,S,m+1))$. As $\omega \notin \mathcal{A}_{m}$, (\ref{ws}) holds at time $m$ and we may apply Lemma~\ref{easycalc}. So let $\tilde{X}:= f(\xi_w: w\in V(\mathcal{H}))$, where $f(x_w: w\in V(\mathcal{H}))$ is the polynomial of degree at most $|D|$ obtained by applying Lemma~\ref{easycalc} to $X$.

We wish to apply Corollary~\ref{vunew} with
\[\tau:= \log^{2(|D|-1)r^7+3r^6 + 1}(d)\cdot d^{\frac{|V(\mathcal{F})|-|R|-|D|}{r-1}}\log^{-3K/5}(d)\]
and $\mathcal{E}_0:= \log^{2|D| + 1}(d)$ to obtain an upper bound on $\tilde{X}$ which holds with high probability. We will prove the following claim.
\begin{claim}\label{Xclaim2}
Let $\omega \notin \mathcal{A}_m$. For $0\leq j\leq |D|$,
$$ \mathbb{E}_j(\tilde{X} \vert \mathcal{F}_{m}')=o(\tau).$$ 
\end{claim}

\begin{proof}
By Lemma~\ref{easycalc} \ref{partialAthing}, for any set $A \subseteq V(\mathcal{H})$ with $|A| = a$,

\[
\mathbb{E}(\partial_A f(\xi_w: w\in V(\mathcal{H}))\mid \mathcal{F}_{m}') \le \sum_{k=0}^{|D|-a}\sum_{X' \in X^{k,a}}X'_{S \cup A}(m) (\log^{r^4}(d)q)^k.
\]
So let us evaluate this expression. By Remark~\ref{rootTransfer}, every $X' \in X^{k,a}$ is a secondary configuration $(\mathcal{F}',R',D')$ with $|D'| \le |D| - 1$ and $|V(\mathcal{F}')| - |R'| - |D'|= |V(\mathcal{F})| - |R| - |D| + k$. By definition, $|X^{k,a}|$ $= O(1)$. So as $\omega \notin \mathcal{A}_m$ and $|D| < 3r$, using (\ref{xs}) to bound each such $X'_{S \cup A}(m)$ gives that $\mathbb{E}(\partial_A f(\xi_w: w\in V(\mathcal{H}))\mid \mathcal{F}_{m}')$ is at most
\begin{align*}
& O\left(\sum_{k=0}^{|D|-a}(m + 1)\log^{2(|D|-1)r^7}(d)\cdot d^{\frac{|V(\mathcal{F})| - |R| - |D|+k}{r-1}}\log^{-3K/5}(d)\left(\log^{r^4}(d)q\right)^k\right)\\
&= O\left(\log^{2(|D|-1)r^7 + 3r^6}(d)\cdot d^{\frac{|V(\mathcal{F})| - |R| - |D|}{r-1}}\log^{-3K/5}(d)\right) \\
&= o(\tau),
\end{align*} 
as required.
\end{proof}
As $X$ is secondary, $|D| \le 3r$ and
\[\tau\log^{|D|}(d)\sqrt{\mathscr{E}_0}=o\left(\log^{2|D|r^7}(d)\cdot d^{\frac{|V(\mathcal{F})|-|R|-|D|}{r-1}}\log^{-3K/5}(d)\right).\]
Using this and the fact that $\mathbb{E}(\tilde{X}) = o(\tau)$ (by Claim~\ref{Xclaim2}), applying Corollary~\ref{vunew} gives that 
$$\mathbb{P}(\mathcal{X}(X,S,m+1)\given \mathcal{F}_m') \le N^{-20\sqrt{\log N}},$$
as required.

\end{proof}

\begin{prop}\label{Y02prop}
$$\mathbb{P}(\mathcal{Y}^0) \le N^{-10\sqrt{\log N}}.$$
\end{prop}

\begin{proof}
For $v \in V(\mathcal{H})$, $1 \le i \le r-2$ and $0 \le m \le M_2 -1$, let $\mathcal{Y}^0(v,i,m+1)$ be the set of $\omega$ in $\mathcal{A}_{M_2}$ such that at time $\mathcal{J}(\omega) = m +1$, 
$$Y_v^{i,0}(m+1) - Y_v^{i,0}(m) > \psi(m)\left(\frac{\lambda^2}{2\alpha}\right)d^{1-\frac{i}{r-1}}.$$
It follows that
$$\mathcal{Y}^0 = \bigcup_{\substack{0 \le m \le M_2 -1\\v \in V(\mathcal{H})\\ 0 \le i \le r-2}} \mathcal{Y}^0(v,i, m + 1).$$
We will prove that for $0 \le m \le M_2 -1$,
$$\mathbb{P}(\mathcal{Y}^0(v,i,m+1)\given \mathcal{F}_m') \le N^{-20\sqrt{\log N}}.$$
The proof of the proposition will then follow by taking the union bound over all choices of $v,i$ and $m$. 

Conditioning on $\mathcal{F}_m'$ for $m \ge 0$, the events $\mathcal{A}_m$ and $\mathcal{Y}^0(v,i,m+1)$ are disjoint and so we assume that $\omega \notin \mathcal{A}_{m}$.  Fix $v \in V(\mathcal{H})$, $1 \le i \le r-2$ and set $\conf:= Y^i = (\mathcal{F},R,D)$. We wish to apply Lemma~\ref{easycalc} along with Corollary~\ref{vunew} to obtain the required bound on $\mathbb{P}(\mathcal{Y}^0(v,i,m+1))$. As $\omega \notin \mathcal{A}_{m}$, (\ref{ws}) holds at time $m$ and we may apply Lemma~\ref{easycalc}. So let $\tilde{Y}:= f(\xi_w: w\in V(\mathcal{H}))$, where $f(x_w: w\in V(\mathcal{H}))$ is the polynomial of degree at most $i$ obtained by applying Lemma~\ref{easycalc} to $\conf$.

We wish to apply Corollary~\ref{vunew} with
\[\tau:=\frac{\psi(m)d^{1-\frac{i}{r-1}}}{\log^{r^2}(N)}\]
and $\mathcal{E}_0 := \log^{r^2 + 1}(N)$ to obtain an upper bound on $\tilde{Y}$ which holds with high probability. We will prove the following claim.

\begin{claim}
\label{Yi0claim2}
Let $\omega \notin \mathcal{A}_m$. For $1\leq a\leq i$,
$$\mathbb{E}_a\left(\tilde{Y}\mid \mathcal{F}_{m}'\right)=o(\tau).$$
Also,
$$\mathbb{E}(\tilde{Y}\mid \mathcal{F}_{m}') \le  \psi(m)\left(\frac{\lambda^2}{4\alpha}\right)d^{1-\frac{i}{r-1}}.$$
\end{claim}
\begin{proof}
By Lemma~\ref{easycalc} \ref{partialAthing}, for any set $A \subseteq V(\mathcal{H})$ with $|A| = a$,
\[
\mathbb{E}(\partial_A f(\xi_w: w\in V(\mathcal{H}))\mid \mathcal{F}_{m}') \le \sum_{k=0}^{i-a}\sum_{\conf' \in \conf^{k,a}}\conf'_{\{v\} \cup A}(m) (\plog (d)q)^k.
\]
So let us evaluate this expression. When $|A| \ge 1$, as $i \le r-2$ each $\conf' \in \conf^{k,a}$ is a secondary configuration $(\mathcal{F}',R',D')$ with $|V(\mathcal{F}')|= r$, $|R'|= |A| + 1$, $|D'|= i - |A| - k$. So $|V(\mathcal{F}')| - |R'| - |D'|= r-1 - i + k$. Also, $|\conf^{k,a}|= O(1)$. So as $\omega \notin \mathcal{A}_m$ and $|D| < 3r$, using (\ref{xs}) to bound each such $\conf'_{\{v\} \cup A}(m)$ gives that $\mathbb{E}(\partial_A f(\xi_w: w\in V(\mathcal{H}))\mid \mathcal{F}_{m}')$ is at most
\begin{align*}
& O\left( \sum_{k=0}^{i-a} \left((m+1)\log^{2|D|r^7}(d)\cdot d^{\frac{r  -1-i+k}{r-1}}\log^{-3K/5}(d)\right)\left(\log^{r^4}(d)q\right)^{k}\right) \\
& = O\left( \plog(d)\log^{-3K/5}(d) d^{1 - \frac{i}{r-1}}\right)\\
& = o(\tau),
\end{align*}
when $K$ is large as $m \le M_2 = O\left(\log(d)\right)$. This proves the first statement of the claim.

In the case $A=\emptyset$, we need to bound the expectation of $f(\xi_w:w\in V(\mathcal{H}))$ more carefully. By Lemma~\ref{easycalc} \ref{expThing}, 
\begin{equation}\label{y0phase2}
\mathbb{E}(\tilde{Y}\mid \mathcal{F}_{m}')  \le \sum_{k=1}^{i}\sum_{\substack{U \in \mathcal{U}(\conf)\\ |U| = k}}\sum_{\conf' \in \conf^U}\conf'_{v}(m)q^{k}.
\end{equation}
Recall Definition~\ref{xuu} and observe that in this case, for each $1 \le k \le i$, there is precisely one $U \in \mathcal{U}$ with $|U|= k$. Here, we see that $\conf^{U} = Z^{i-k,k}$. So as $\omega \notin \mathcal{A}_{m}$, applying (\ref{Zbound})  and using \eqref{chiijBound} to bound $\chi(i-k,k)$ (as $i-k \le r-3$) to evaluate \eqref{y0phase2} gives 
\begin{align*}
\mathbb{E}(\tilde{Y}\mid \mathcal{F}_{m}') & \le \sum_{k=1}^{i}Z_v^{i-k,k}(m)q^k \\
& \leq  \sum_{k=1}^i \left(1+\frac{\lambda^2}{r^{3r}}\right)\psi(m)\chi(i-k,k)\alpha^kd^{1-\frac{i}{r-1}} \\
& \leq \sum_{k=1}^i2\psi(m)\left(\frac{\lambda^2}{r^{3r+1}\alpha^{k+1}}\right)\alpha^kd^{1-\frac{i}{r-1}} \\
&\leq \psi(m)\left(\frac{\lambda^2}{4\alpha}\right)d^{1-\frac{i}{r-1}}.
\end{align*}
So
\[
\label{Yexp2prime}
\mathbb{E}(\tilde{Y}\mid \mathcal{F}_{m}') \le \psi(m)\left(\frac{\lambda^2}{4\alpha}\right)d^{1-\frac{i}{r-1}}.
\]
This proves the second statement of the claim. 
\end{proof}
As
\[\tau\log^i(N)\sqrt{\mathscr{E}_0} =o\left(\psi(m)d^{1- \frac{i}{r-1}}\right),\]
then applying Corollary~\ref{vunew} with the bound on $\mathbb{E}(\tilde{Y})$ obtained in Claim~\ref{Yi0claim2} gives that $$\mathbb{P}(\mathcal{Y}^0(v,i,m+1)\given \mathcal{F}_m') \le  N^{-20\sqrt{\log N}},$$
as required.
\end{proof}

\begin{prop}\label{Yij2prop}
$$\mathbb{P}(\mathcal{Y}^{>0}) \le N^{-10\sqrt{\log N}}.$$
\end{prop}

\begin{proof}
For $v \in V(\mathcal{H})$, $0 \le i \le r-2$, $1 \le j \le r-1$ and $0 \le m \le M_2 -1$, let $\mathcal{Y}^{>0}(v,i,j,m+1)$ be the set of $\omega$ in $\mathcal{A}_{M_2}$ such that at time $\mathcal{J}(\omega) = m +1$, 
$$Y_v^{i,j}(m+1) > \psi(m+1)\chi(i,j)d^{1-\frac{i}{r-1}}.$$
It follows that
$$\mathcal{Y}^{>0} = \bigcup_{m,v,i,j} \mathcal{Y}^{>0}(v,i,j, m+1).$$
We will prove that for $0 \le m \le M_2 - 1$,
$$\mathbb{P}(\mathcal{Y}^{>0}(v,i,j,m+1)\given \mathcal{F}_m') \le N^{-20\sqrt{\log N}}.$$
The proposition will then follow by taking the union bound over all choices of $v$, $i,j$ and $m$. 

Conditioning on $\mathcal{F}_m'$ for $m \ge 0$, the events $\mathcal{A}_m$ and $\mathcal{Y}^{>0}(v,i,m+1)$ are disjoint and so we assume that $\omega \notin \mathcal{A}_{m}$. Fix $v \in V(\mathcal{H})$, $0 \le i \le r-2$, $1 \le j \le r-1$ and set $\conf:= Y^{i,j} = (\mathcal{F},R,D)$. We wish to apply Lemma~\ref{easycalc} along with Corollary~\ref{vunew} to obtain the required bound on $\mathbb{P}(\mathcal{Y}^{>0}(v,i,j,m+1))$. As $\omega \notin \mathcal{A}_{m}$, (\ref{ws}) holds at time $m$ and we may apply Lemma~\ref{easycalc}. So let $\tilde{Z}:= f(\xi_w: w\in V(\mathcal{H}))$, where $f(x_w: w\in V(\mathcal{H}))$ is the polynomial of degree at most $i+(r-1)j$ obtained by applying Lemma~\ref{easycalc} to $\conf$.

As $j \ge 1$, $Y^{i,j}$ contains an unstable hyperedge and so every copy of $Y^{i,j}$ in $\mathcal{H}(m)$ is destroyed when we sample every open hyperedge in $Q(m)$. So $Y_v^{i,j}(m+1) \cap Y_v^{i,j}(m) = \emptyset$. 

We wish to apply Corollary~\ref{vunew} with
\[\tau:=\frac{\psi(m)d^{1-\frac{i}{r-1}}}{\log^{r^6}(N)}\]
and $\mathcal{E}_0:= \log^{r^6 +1}(N)$ to obtain an upper bound on $Y_v^{i,j}(m)$ which holds with high probability. We will use the following claim.
\begin{claim}\label{yijclaim2}
Let $\omega \notin \mathcal{A}_m$. For $1\leq a\leq i+(r-1)j$, 
$$\mathbb{E}_a(\tilde{Z}\mid \mathcal{F}_{m}')=o(\tau).$$ 
Also,
$$\mathbb{E}(\tilde{Z}\mid \mathcal{F}_{m}') \le \left(1-\frac{3\lambda}{2}\right)\psi(m)\chi(i,j)d^{1-\frac{i}{r-1}}.$$
\end{claim}
\begin{proof}
By Lemma~\ref{easycalc} \ref{partialAthing}, for any set $A \subseteq V(\mathcal{H})$ with $|A| = a$,
\begin{equation}\label{yijbad}
\mathbb{E}(\partial_A f(\xi_w: w\in V(\mathcal{H}))\mid \mathcal{F}_{m}') \le \sum_{k=0}^{i + (r-1)j -a}\sum_{\conf' \in \conf^{k,a}}\conf'_{\{v\} \cup A}(m) (\log^{r^4}(d)q)^k.
\end{equation}
First let us evaluate this expression when $|A| \ge 1$. For some $0\le k \le i + (r-1)j-a$, consider $\conf' = (\mathcal{F}',R',D') \in \conf^{k,a}$ (as $|A| \ge 1$, we have $|R'|\ge 2$). By definition, $\mathcal{F}'$ is isomorphic to $\mathcal{F}$ (ignoring which vertices are roots and marked). In particular, $\mathcal{F}' = e_0 \cup e_1 \cup \cdots \cup e_j$, where $e_0$ is the central hyperedge, the hyperedges $e_1, \ldots, e_j$ are pairwise disjoint and each $e_i$ intersects $e$ in a neutral vertex $v_i$.

For each $0 \le \ell \le j$, define $R_{\ell} := e_{\ell} \cap R'$ and define $D_{\ell} := e_{\ell} \cap D'$. For $1 \le \ell \le j$ define the configuration $Z^{\ell} := (e_{\ell}, R_{\ell} \cup \{v_{\ell}\}, D_{\ell})$. By Definition~\ref{xuu}, $\conf'$ is obtained from $Y^{i,j}$ by making a fruitful set $\emptyset \not= U \subseteq D$ neutral and turning some subset $U' \subseteq U$ into roots. So as $U$ is fruitful, $|D_{\ell}| < r-1$ for each $\ell$. 

For each $1 \le \ell \le j$, the configuration $Z^{\ell}$ satisfies one of the following.
\begin{itemize}
\item[(1)] $R_{\ell} = \emptyset$ and $|D_{\ell}|= s$ for some $s \le r-2$: in this case $Z^{\ell} = Y^{s,0}$.
\item[(2)] $R_{\ell} \not= \emptyset$ and $Z^{\ell}$ contains a neutral vertex: in this case $Z^{\ell}$ is a secondary configuration.
\item[(3)]  $R_{\ell} \not= \emptyset$ and $Z^{\ell}$ contains no neutral vertex: in this case $Z^{\ell} = W^{|R_{\ell}| + 1}$. 
\end{itemize}
Let $S \subseteq V(\mathcal{H})$ be a set of cardinality $|R_{\ell}| +1$. If $Z_{\ell}$ satisfies (1) then as $\omega \notin \mathcal{A}_m$, using (\ref{yi0}) gives
\begin{equation}
\label{zisy}
Z^{\ell}_{S}(m) \le \chi(s,0)d^{1 - \frac{s}{r-1}} = O\left(d^{\frac{r - 1 - |R_{\ell}|-|D_{\ell}|}{r-1}}\right)
\end{equation}
If $Z_{\ell}$ satisfies (2), then using (\ref{xs}) (as $\omega \notin \mathcal{A}_m$) and the fact that $m = O\left(\log(d)\right)$ gives
\begin{equation}
\label{zisx}
Z^{\ell}_{S}(m) \le \plog(d)\cdot d^{\frac{r - 1- |R_{\ell}|-|D_{\ell}|}{r-1}}\log^{-3K/5}(d).
\end{equation}
If $Z_{\ell}$ satisfies (3), then using (\ref{ws}) (as $\omega \notin \mathcal{A}_m$) and the fact that $m = O\left(\log(d)\right)$ gives
\begin{equation}
\label{zisw}
Z^{\ell}_{S}(m) \le \log^{r^4}(d) = O\left(\plog(d)\cdot d^{\frac{r - 1-|R_{\ell}|-|D_{\ell}|}{r-1}}\right)
\end{equation}

We consider two cases. Firstly $|e_0 \cap R'|= 1$. Then as $|A| \ge 1$, without loss of generality $|e_1 \cap R'| \ge 1$. In this case, the configuration $\hat{Z} := (e_0 \cup e_1,R_0 \cup R_1,D_0 \cup D_1)$ is secondary. 

Let $\mathcal{S}$ be the set of all partitions $\mathcal{P}:= (S_1, \ldots, S_j)$ of $\{v\} \cup A$ such that $|S_1|= |R_0 \cup R_1|$ and for $2 \le \ell \le j$, $|S_{\ell}|=|R_{\ell}|$. Using this notation, we can bound the number of copies of $\conf'  = (\mathcal{F}',R',D') \in \conf^{k,a}$ rooted at $\{v\} \cup A$. We have

\begin{equation}\label{firstbound}
\conf_{\{v\} \cup A}'(m) \le \sum_{\mathcal{P} \in \mathcal{S}} \sum_{Z \in \hat{Z}_{S_1}(m)} \sum_{(u_2,\ldots,u_j) \subseteq V(Z)} \prod_{\ell = 2}^{j}Z^{\ell}_{S_{\ell} \cup \{u_{\ell}\}}(m)\\
\end{equation}
where the sum is taken over all distinct ordered $(u_2,\ldots,u_j)$. 

For any $S_1 \subseteq V(\mathcal{H})$ with $|S_1|= |R_0 \cup R_1|$, as  $\omega \notin \mathcal{A}_m$ we can use (\ref{xs}) to give
$$\hat{Z}_{S_1}(m)  \le \plog(d)\cdot d^{\frac{2r - 1- |R_0 \cup R_1|-|D_0 \cup D_1|}{r-1}}\log^{-3K/5}(d).$$  
As $|\mathcal{S}|= O(1)$, using this with \eqref{zisy}, \eqref{zisx} and \eqref{zisw} to evaluate \eqref{firstbound} gives
\begin{align*}
\conf'_{\{v\} \cup A}(m)= \plog(d)d^{\frac{(j+1)r -j - \sum_{\ell=0}^j(|R_{\ell}| + |D_{\ell}|)}{r-1}} \log^{-3K/5}(d).
\end{align*}
As $\conf' \in \conf^{k,a}$, we have $j(r-1) + i + 1 = |R| + |D| = |R'| + |D'| + k$, therefore $|R'| + |D'| = j(r-1) + i + 1 - k$ and
\begin{align*}
(j+1)r - j - \sum_{\ell=0}^j(|R_{\ell}| + |D_{\ell}|) &= (j+1)r - j - |R'| - |D'| \\
& = (j+1)r - j - j(r-1) - i - 1 + k\\
& = r - 1 - i + k.
\end{align*}
So in this case,
$$\conf_{\{v\} \cup A}'(m)= \plog(d)d^{1 - \frac{i - k}{r-1}} \log^{-3K/5}(d).$$

Now consider $\conf' = (\mathcal{F}',R',D') \in \conf^{k,a}$ such that $|e_0 \cap R'| > 1$. This case follows very similarly to the previous case. As $j > 1$, there is a neutral vertex of $e_0$. So the configuration $Z^0:= (e_0, R_0,D_0)$ is secondary. 

This time, let $\mathcal{S}$ be the set of all partitions $\mathcal{P}:= (S_0, \ldots, S_j)$ of $\{v\} \cup A$ such that $|S_0|= |R_0|$ and for $1 \le \ell \le j$, $|S_{\ell}|= |R_{\ell}|$. As in the previous case, $|\mathcal{S}|= O(1)$. Using this notation, we can bound the number of copies of $\conf'  = (\mathcal{F}',R',D') \in \conf^{k,a}$ rooted at $\{v\} \cup A$. As $\omega \notin \mathcal{A}_m$, we can use (\ref{xs}) to bound $Z^0_{S_0}$ and, for $1 \le \ell \le j$, we can use \eqref{zisy}, \eqref{zisx} and \eqref{zisw} as in the previous case to bound $Z^{\ell}_{S_{\ell} \cup \{u_{\ell}\}}$. This gives
\begin{align*}
\conf_{\{v\} \cup A}'(m) & \le \sum_{\mathcal{P} \in \mathcal{S}} \sum_{Z \in Z^0(m)} \sum_{(u_1,\ldots,u_j) \subseteq V(Z)} \prod_{\ell = 1}^{j}Z^{\ell}_{S_{\ell} \cup \{u_{\ell}\}}(m)\\
&= \plog(d) d^{\frac{(j+1)r -j - \sum_{\ell=0}^j(|R_{\ell}| + |D_{\ell}|)}{r-1}} \log^{-3K/5}(d)\\
&= \plog(d) d^{1 - \frac{i-k}{r-1}} \log^{-3K/5}(d),
\end{align*}
where, as before, the sum is taken over all distinct ordered $(u_1,\ldots,u_j)$. 

Now we are able to evaluate \eqref{yijbad} for $|A| \ge 1$. As $|\conf^{k,a}|= O(1)$,
\begin{align*}
\mathbb{E}(\partial_A f(\xi_w: w\in V(\mathcal{H}))\mid \mathcal{F}_{m}') &\le \sum_{k=0}^{i + (r-1)j-a}\sum_{\conf' \in \conf^{k,a}}\conf'_{\{v\} \cup A}(m) (\log^{r^4}(d)q)^k \\
& = \sum_{k=0}^{i + (r-1)j-a}\plog(d)d^{1 - \frac{i - k}{r-1}} \log^{-3K/5}(d)(\log^{r^4}(d)q)^k\\
& = \plog(d) d^{1 - \frac{i}{r-1}} \log^{-3K/5}(d)\\
& = o(\tau),
\end{align*}
when $K$ is sufficiently large, as required for the first statement of the claim.

In the case $A=\emptyset$, we need to bound the expectation of $f(\xi_w:w\in V(\mathcal{H}))$ more carefully. By Lemma~\ref{easycalc} \ref{expThing}, 
\begin{equation}\label{yijphase2}
\mathbb{E}(\tilde{Z}\mid \mathcal{F}_{m}') \le \sum_{k=1}^{i + (r-1)j}\sum_{\substack{U \in \mathcal{U}(\conf)\\ |U| = k}}\sum_{\conf' \in \conf^U}\conf'_{v}(m)q^{k}.
\end{equation}
Recall the definition of $\conf^U$ from Definition~\ref{xuu}. Consider $\conf' := (\mathcal{F}',D',R') \in \conf^U$ for some $\emptyset \not= U \subseteq D$. Recall that $\mathcal{F}' \supseteq \mathcal{F}$. Let $e_0$ be the central hyperedge of $\mathcal{F}$ and let $e_1,\ldots, e_j$ be the non-central hyperedges of $\mathcal{F}$.

 For $0 \le \ell \le j$, define $U_{\ell}:= e_{\ell} \cap U$. By definition, each vertex $u$ of $U$ is the unique neutral vertex of an unstable hyperedge $e_u$ in $\mathcal{F}'$. Without loss of generality, suppose that $|U_{1}| \ge |U_2| \ge \ldots \ge |U_j|$. As (by definition) $U$ is fruitful (i.e.~it intersects every unstable hyperedge of $\conf'$) it must intersect every hyperedge $e_1,\ldots,e_j$ and hence $|U_{\ell}| \ge 1$ for all $1 \le \ell \le j$.
 
 We will now define some configurations that will help us break up $\conf'$ in order to bound the number of copies of it in $\mathcal{H}(m)$. For each $0 \le \ell \le j$, define $\mathcal{G}_{\ell}$ to be the subhypergraph of $\mathcal{F}'$ containing $e_{\ell} \cup \{e_u: u \in U_{\ell}\}$. Define $R_0:= R'$  $(= \{v\})$ and, for $1 \le \ell \le j$, define $R_{\ell}:= e_{\ell} \cap e_0$ (so $|R_{\ell}| = 1$). For $0 \le \ell \le j$, define $D_{\ell}:= \mathcal{G}_{\ell} \cap D'$ and let $Y^{\ell}$ be the configuration $(\mathcal{G}_{\ell},R_{\ell},D_{\ell})$. 
 
 First consider $Y^0$. We have $|R_0|= 1$ and $|D_0|\ge i- |U_0|$. If $|U_0| \ge 1$, we have $Y^0\in Z^{i - |U_0|,|U_0|}$. If $|U_0| = 0$, we have $Y^0 = Y^{i,0}$. Hence as $\omega \notin \mathcal{A}_m$, by (\ref{yi0}) and (\ref{Zbound}) we have for any $v \in V(\mathcal{H})$,
 \begin{equation}\label{z0bound}
 Y^0_v(m) \le \begin{cases}
   \chi(i,0)d^{1 - \frac{i}{r-1}} & \text{if } |U_0| = 0,\\    
   \left(1 + \frac{\lambda^2}{r^{3r}}\right)\psi(m)\chi(i - |U_0|,|U_0|)d^{1 - \frac{i - |U_0|}{r-1}} & \text{otherwise.}   
 \end{cases} 
 \end{equation}
 Note that when $|U_0| = 0$, we can write $ \chi(i,0)d^{1 - \frac{i}{r-1}} =  \chi(i - |U_0|,|U_0|)d^{1 - \frac{i - |U_0|}{r-1}}.$
 
Now when $1 \le \ell \le j$, we have $|U_{\ell}| \ge 1$ (as $U$ is fruitful), $|R_{\ell}|= 1$  and $|U_{\ell}| + |e_{\ell} \cap D'|= r-1$. So in particular, for each $1 \le \ell \le j$, we have $Y^{\ell}\in Z^{r-1-|U_{\ell}|,|U_{\ell}|}$.
 
For $Z \in Y_v^0(m)$, let $h(Z)$ denote the central hyperedge of $Z$. Putting this together and applying (\ref{Zbound}) (as $\omega \notin \mathcal{A}_m$) gives 
 \begin{align*}\label{Xabove}
 \conf_v'(m) \le &\sum_{Z \in Y^0_v(m)} \sum_{(u_1,\ldots,u_j) \in h(Z)\setminus (\{v\} \cup I(m))} \prod_{\ell = 1}^{j}Z_{u_\ell}^{r-1-|U_{\ell}|,|U_{\ell}|}(m) \nonumber \\
 & \le Y^0_v(m)\frac{(r-1-i)!}{(r-1-i-j)!}\prod_{\ell = 1}^{j}  \left(\left(1+\frac{\lambda^2}{r^{3r}}\right)\psi(m)d^{\frac{|U_{\ell}|}{r-1}}\chi(r - 1 - |U_{\ell}|,|U_{\ell}|)\right).
\end{align*}
 Applying \eqref{z0bound} to this gives 
\begin{equation}
\label{usethis}
\conf_v'(m) \le C_{i,j} \cdot \chi(i-|U_0|,|U_0|) d^{1 - \frac{i- |U_0|}{r-1}} \prod_{\ell = 1}^{j}\left(d^{\frac{|U_{\ell}|}{r-1}} \cdot\chi(r - 1 - |U_{\ell}|,|U_{\ell}|)\right),
\end{equation}
where 
\begin{equation}
\label{Cdefine}
C_{i,j}:= \frac{(r-1-i)!}{(r-1-i-j)!}\left(1+\frac{\lambda^2}{r^{3r}}\right)^j \psi^j(m) \cdot \max\left\{1,  \left(1 + \frac{\lambda^2}{r^{3r}}\right)\psi(m) \right\}.
\end{equation}

Rewriting the right hand side of \eqref{yijphase2} gives
$$\mathbb{E}(\tilde{Z}\mid \mathcal{F}_{m}') \le \sum_{b = 0}^{i}\sum_{\substack{U \in \mathcal{U}(\conf)\\ |U_0| = b}}\sum_{\conf' \in \conf^U}\conf'_{v}(m)q^{|U|}.$$

So using \eqref{usethis} to bound $\conf'_v(m)$  and using the fact that $|U| = \sum_{\ell = 0}^{j} |U_{\ell}|$ gives 
\begin{equation}
\label{Xagain}
\mathbb{E}(\tilde{Z}\mid \mathcal{F}_{m}') \le C_{i,j} \cdot \sum_{b = 0}^{i}q^b\sum_{\substack{U \in \mathcal{U}(\conf)\\ |U_0| = b}}\chi(i - b,b)d^{1-\frac{i -b}{r-1}}\prod_{\ell = 1}^{j}d^{\frac{|U_{\ell}|}{r-1}} \cdot\chi(r - 1 - |U_{\ell}|,|U_{\ell}|)q^{|U_{\ell}|}.
\end{equation}
Now for each distinct ordered partition $(k_1,\ldots,k_j)$ of $[k]$ where $k_1 \ge \ldots \ge k_j$ and for each $b$, by the definition of $\mathcal{U}(\conf)$ (in  Definition~\ref{xuu}) and the definition of $Y^{i,j}$ (recall Definition~\ref{Ydef}), there is precisely one $U \in \mathcal{U}(\conf)$ such that $(|U_1|, \ldots, |U_{\ell}|) = (k_1,\ldots, k_{\ell})$ and $|U_0| = b$. Therefore, from \eqref{Xagain} we have 
\begin{align}
\label{sums}
\mathbb{E}(\tilde{Z}\mid \mathcal{F}_{m}') &\le C_{i,j} \cdot \sum_{b = 0}^{i}q^b\chi(i - b,b)d^{1-\frac{i -b}{r-1}}\left(\sum_{\ell = 1}^{r-1} q^{\ell}\chi(r-1-\ell, \ell)d^{\frac{\ell}{r-1}}\right)^{j} \nonumber\\
&= C_{i,j} \cdot d^{1-\frac{i}{r-1}} \cdot \sum_{b = 0}^{i}\alpha^b \chi(i - b,b) \left(\sum_{\ell = 1}^{r-1} \alpha^{\ell}\chi(r-1-\ell, \ell)\right)^{j}.
\end{align}

Consider the first summation in \eqref{sums}. We have
\begin{align}\label{tedi2}
\sum_{b=0}^i\chi(i-b,b)\alpha^b &= \chi(i,0) + \sum_{b=1}^i\chi(i-b,b)\alpha^b \nonumber \\
&= \chi(i,0)\left(1+\sum_{b=1}^i\frac{\chi(i-b,b)\alpha^b}{\chi(i,0)}\right) \nonumber \\
& \le \chi(i,0)\left(1+\sum_{b=1}^i\frac{\alpha^b}{\chi(i,0)}\cdot \frac{\lambda^2\chi_{\min}}{r^{3r+1}(1+\alpha+\alpha^r)}\cdot \chi(r-2,1)^{b-1}\right),   
\end{align}
where in the final line \eqref{prop7.5} was used to bound $\chi(i-b,b)$ (and we ignored the factor of $(1 + \chi_{\max})$). Now using (in this order) the fact that $\alpha^s < 1 + \alpha + \alpha^r$ for $0 \le s \le r-1$, the fact that $\chi_{\min} \le \chi(i,0)$ (by definition) and that $\chi(r-2,1) <1$, we have
\begin{align}
\label{tedius}
\frac{\alpha^b}{\chi(i,0)}\cdot \frac{\lambda^2\chi_{\min}}{r^{3r+1}(1+\alpha+\alpha^r)}\cdot \chi(r-2,1)^{b-1} &< \frac{\chi_{\min}}{\chi(i,0)}\cdot \frac{\lambda^2}{r^{3r+1}}\cdot \chi(r-2,1)^{b-1} \nonumber \\
&<  \frac{\lambda^2}{r^{3r+1} }\cdot \chi(r-2,1)^{b-1} \nonumber \\
&\le \frac{\lambda^2}{r^{3r+1}}.
\end{align}
Using \eqref{tedi2} and \eqref{tedius}, as $i < r$ we obtain the bound
\begin{equation}
\label{sumone}
\sum_{b=0}^i\chi(i-b,b)\alpha^b < \chi(i,0)\left(1+\frac{\lambda^2}{r^{3r}}\right).
\end{equation}

Now consider the second summation in \eqref{sums}. We have
\begin{align*}
\sum_{\ell=1}^{r-1}\chi(r\!-\!1\!-\!\ell,\ell)\alpha^\ell &= \chi(r\!-\!2,1)\alpha + \sum_{\ell=2}^{r-1}\chi(r\!-\!1\!-\!\ell,\ell)\alpha^\ell  \\
& = \chi(r\!-\!2,1)\alpha\left(1+\sum_{\ell=2}^{r-1}\frac{\chi(r\!-\!1\!-\!\ell,\ell)\alpha^{\ell-1}}{\chi(r\!-\!2,1)}\right) \\
&= \chi(r\!-\!2,1)\alpha\left(1+\sum_{\ell=2}^{r-1}\frac{\alpha^{\ell - 1}}{\chi(r\!-\!2,1)} \cdot \frac{\lambda^2\chi_{\min}}{r^{3r+1}(1+\alpha+\alpha^r)}\cdot \chi(r\!-\!2,1)^{\ell-1}\right),  \stepcounter{equation}\tag{\theequation}\label{brack}
\end{align*}
where in the final line \eqref{prop7.5} was used to bound $\chi(r-1-\ell,\ell)$. We will now show that we can bound the second term within the bracket in \eqref{brack} by $\frac{\lambda^2}{r^{3r}}$, which will give
\begin{equation}
\label{sumtwo}
\sum_{\ell=1}^{r-1}\chi(r-1-\ell,\ell)\alpha^\ell <\chi(r-2,1)\alpha\left(1+\frac{\lambda^2}{r^{3r}}\right).
\end{equation}

To bound the second term within the bracket in \eqref{brack}, we will bound each term of the sum by $\frac{\lambda^2}{r^{3r+1}}$. First, observe that as $\ell \ge 2$ and $\chi(r-2,1) < 1$, we have
\begin{align*}
	\frac{\alpha^{\ell - 1}}{\chi(r\!-\!2,1)} \cdot \frac{\lambda^2\chi_{\min}}{r^{3r+1}(1+\alpha+\alpha^r)}\cdot \chi(r\!-\!2,1)^{\ell-1} &\le  \alpha^{\ell - 1}\frac{\lambda^2\chi_{\min}}{r^{3r+1}(1+\alpha+\alpha^r)}\\
	&\le \frac{\lambda^2\chi_{\min}}{r^{3r+1}}\frac{\alpha^{\ell -2}}{(1+\alpha+\alpha^r)}\\
	& \le \frac{\lambda^2\chi_{\min}}{r^{3r+1}},
\end{align*}
where for the second inequality we used that $\chi_{\min} \le \chi(r-2,0) = \frac{1-3\lambda}{\alpha}$ (by \eqref{r2def}), and for the final inequality we used that $\ell \in \{2,\ldots,r-1\}$ so $\frac{\alpha^{\ell -2}}{(1+\alpha+\alpha^r)} < 1$.
Putting together \eqref{sums}, \eqref{sumone} and \eqref{sumtwo} gives 
\begin{align}
\label{yijbig}\mathbb{E}(\tilde{Z}\mid \mathcal{F}_{m}') &\le C_{i,j} \cdot d^{1-\frac{i}{r-1}} \cdot \chi(i,0)\left(1+\frac{\lambda^2}{r^{3r}}\right)\left(\chi(r-2,1)\alpha\left(1+\frac{\lambda^2}{r^{3r}}\right)\right)^{j} \nonumber\\
& \le C_{i,j} \left(1+\frac{\lambda^2}{r^{3r}}\right)^{j+1}\chi(i,0)\chi(r-2,1)^j\alpha^jd^{1-\frac{i}{r-1}} 
\end{align}
First consider the case $(i,j)\neq (r-2,1)$. By Definition~\ref{zetaDef} and \eqref{yijbig} we have
\begin{equation}\label{andanother}
\mathbb{E}(\tilde{Z}\mid \mathcal{F}_{m}') \le \frac{C_{i,j}}{r^{3r}}\left(1+\frac{\lambda^2}{r^{3r}}\right)^{j+1} \chi(i,j)d^{1-\frac{i}{r-1}}.
\end{equation}
Using the definition of $C_{i,j}$ (in~\ref{Cdefine}), we have
\begin{align*}
\frac{C_{i,j}}{r^{3r}}\left(1+\frac{\lambda^2}{r^{3r}}\right)^{j+1} & = \frac{1}{r^{3r}} \frac{(r-1-i)!}{(r-1-i-j)!}\left(1+\frac{\lambda^2}{r^{3r}}\right)^{2j+1} \psi^j(m) \cdot \max\left\{1,  \left(1 + \frac{\lambda^2}{r^{3r}}\right)\psi(m) \right\} \nonumber \\
& \le \frac{1}{r^{2r}}\left(1+\frac{\lambda^2}{r^{3r}}\right)^{2j+1} \psi(m) \cdot \max\left\{1,  \left(1 + \frac{\lambda^2}{r^{3r}}\right)\psi(m) \right\} \nonumber 
\end{align*}
\begin{align}\label{soomanylabels}
& \le \frac{1}{r^{2r}}\left(1+\frac{\lambda^2}{r^{3r}}\right)^{2j+2} \psi(m) \nonumber \\
& \le \frac{1}{r^{2r}}(1 + \lambda^2)\psi(m) \nonumber \\
& \le \left(1-\frac{3\lambda}{2}\right)(1-\lambda)\psi(m) \nonumber \\
&= \left(1-\frac{3\lambda}{2}\right)\psi(m+1),
\end{align}
where the final inequality follows from the fact that $\lambda < 1/8$ and $r \ge 3$. Combining \eqref{andanother} with \eqref{soomanylabels} gives
$$\mathbb{E}(\tilde{Z}\mid \mathcal{F}_{m}') \leq   \left(1-\frac{3\lambda}{2}\right)\psi(m+1)\chi(i,j)d^{1-\frac{i}{r-1}},$$
as required. This completes the case when $(i,j) \not= (r-2,1)$.

In the case $i=r-2$ and $j=1$, from \eqref{Cdefine} we have
\begin{equation}\label{newC}
C_{r-2,1} \le  \left(1+\frac{\lambda^2}{r^{3r}}\right)^{2}\psi(m).
\end{equation}
So from \eqref{newC}, \eqref{yijbig} and \eqref{r2def} we have
\begin{align*}\label{r21}
\mathbb{E}(\tilde{Z}\mid \mathcal{F}_{m}') &\le  C_{r-2,1} \left(1+\frac{\lambda^2}{r^{3r}}\right)^{2}\chi(r-2,0)\chi(r-2,1)\alpha d^{1-\frac{r-2}{r-1}} \nonumber \\
& \le C_{r-2,1} \left(1+\frac{\lambda^2}{r^{3r}}\right)^{2} \frac{1 - 3\lambda}{\alpha}\chi(r-2,1)\alpha d^{1-\frac{r-2}{r-1}} \nonumber \\
& \le \left(1+\frac{\lambda^2}{r^{3r}}\right)^{4}(1 - 3\lambda)\psi(m)\chi(r-2,1)d^{1-\frac{r-2}{r-1}} \\
& \le \left(1 - \frac{3\lambda}{2}\right)(1 - \lambda)\psi(m)\chi(r-2,1)d^{1-\frac{r-2}{r-1}}\\
& \le  \left(1 - \frac{3\lambda}{2}\right)\psi(m+1)\chi(r-2,1)d^{1-\frac{r-2}{r-1}},
\end{align*}
where the penultimate expression follows from the fact that $\lambda < 1/8$ and $r \ge 3$.

To summarise, when $\omega \notin \mathcal{A}_m$ we have
\[
\mathbb{E}(\tilde{Z}\mid \mathcal{F}_{m}') \le \left(1-\frac{3\lambda}{2}\right)\psi(m+1)\chi(i,j)d^{1-\frac{i}{r-1}},
\]
 as required.
\end{proof}
As 
\[\tau\log^{i + (r-1)j}(N)\sqrt{\mathscr{E}_0} = o\left(\psi(m)d^{1 - \frac{i}{r-1}}\right),\]
then applying Corollary~\ref{vunew} with the bound on $\mathbb{E}(\tilde{Z})$ obtained in Claim~\ref{yijclaim2} we get that 
$$\mathbb{P}(\mathcal{Y}^{>0}(v,i,j,m+1)\given \mathcal{F}_m') \le  N^{-20\sqrt{\log N}},$$ as required. 
\end{proof}
This completes the proof of Lemma~\ref{subtrack} and the proof of Theorem~\ref{hypmainThm} in the subcritical case.

\section{The Second Phase in the Supercritical Case} 
\label{sec:super}
In this section we prove the ``supercritical case'' of Theorem~\ref{hypmainThm} (i.e.~when $c^{r-2}\alpha>\frac{(r-2)^{r-2}}{(r-1)^{r-1}}$). Recall, as in the previous section, that we have ``restarted the clock'' and that in this case the first phase runs until time $T:= \left\lfloor\frac{1}{\alpha}\log N \right\rfloor$. It may be helpful to recall the exact process we run in this phase, given in Subsection~\ref{subsup}. Let us repeat the definition of $M_2$ from \eqref{M2def} in the supercritical case for the sake of convenience:
\begin{equation}
\label{M2againsuper}
M_2=\min\left\{m:\left(\log N\right)^{\left(\frac{3}{2}\right)^m}>d^{\frac{1}{r-1} + \frac{1}{10}}\right\}
\end{equation}

Theorem~\ref{hypmainThm} in this case is implied by the following lemma. 

\begin{lem}\label{superlem}
With probability at least $1 - N^{-7}$:  For all $1 \le m \le M_2$ the bound
$$Q_v(m) \ge (\log N)^{\left(\frac{3}{2}\right)^{m}}$$
holds for all $v \in V(\mathcal{H})\setminus I(m)$.
\end{lem}

We first show how the supercritical case of Theorem~\ref{hypmainThm} follows from Lemma~\ref{superlem}, before proving Lemma~\ref{superlem}.

\begin{proof}[Proof of Theorem~\ref{hypmainThm} in the supercritical case.]
By Lemma~\ref{superlem} and the definition of $M_2$ (see \eqref{M2againsuper}), with probability at least $1 - N^{-7}$ the bound
$$Q_v(M_2) \ge d^{\frac{1}{r-1} + \frac{1}{10}},$$
holds for all $v \in V(\mathcal{H})\setminus I(M_2)$.
We now sample every open hyperedge in $Q(M_2)$. As each hyperedge is independently in $\mathcal{H}_q$, the probability that some vertex $v$ fails to be infected after this is
\begin{equation}
\label{Qgetsinf}
(1-q)^{Q_v(M_2)} \le (1-q)^{d^{\frac{1}{r-1} + \frac{1}{10}}}\le e^{-qd^{\frac{1}{r-1} + \frac{1}{10}}} \le e^{-\alpha d^{\frac{1}{10}}} < N^{-10\sqrt{\log N}}.
\end{equation} 
So by the union bound, with probability at least $1 - N^{-5}$, every healthy vertex becomes infected in this round. This completes the proof of Theorem~\ref{hypmainThm} in the supercritical case.
\end{proof} 

Recall from Subsection~\ref{subsup} that in the supercritical case each round consists of two steps. In the first step a subset $Q'(m) \subseteq Q(m)$ is chosen and every hyperedge in $Q'(m)$ is sampled. In the second step, for any healthy vertex contained in a large number of open hyperedges, all hyperedges containing that vertex are sampled. 

\begin{defn}
Define $\mathcal{A}_0:= \mathcal{B}_M$ (given in Definition~\ref{bmdef}). For $1 \le m \le M_2$, define $\mathcal{A}_m$ to be the event (in $\Omega'$, which was defined in Subsection~\ref{sec:probspace}) that either $\mathcal{A}_0$ occurs, or there exists some $v \in V(\mathcal{H}) \setminus I(m)$ such that either:
\begin{itemize}
\item[(S.1)] $Q_v(m) < (\log N)^{\left(\frac{3}{2}\right)^{m}}$, or
\item[(S.2)] $Q_v^i(\ell)$ is large, for some $\ell \le m$. 
\end{itemize}
\end{defn}
So to prove Lemma~\ref{superlem}, it suffices to prove that
\begin{equation}\label{superlemenough}
\mathbb{P}(\mathcal{A}_{M_2}) \le N^{-7}.
\end{equation}

As mentioned above, to prove Lemma~\ref{superlem} we use a lower tail version of Janson's Inequality, Theorem~\ref{JansonLower}. In each of our applications of Theorem~\ref{JansonLower}, we will simply set $\varepsilon=1/2$ and use the fact that $\varphi(-1/2) = (1/2)\log(1/2)+ (1/2) \geq 1/10$. We are now ready to complete the proof of Lemma~\ref{superlem}.

\begin{proof}[Proof of Lemma~\ref{superlem}.]
Given a point $\omega \in \mathcal{A}_{M_2}$, let
$$\mathcal{J} = \mathcal{J}(\omega) := \min \{i : \omega \in \mathcal{A}_i\}.$$ 
For $0 \le m \le M_2 -1$ and $v \in V(\mathcal{H})$, let $\mathcal{S}(v,m+1)$ be the set of $\omega \in \mathcal{A}_{M_2}$ such that $\omega \notin \mathcal{A}_0$, $\mathcal{J}(\omega) = m+1$, $v \notin I(m+1)$ and either:
\begin{enumerate}[(i)]
\item\label{Qvsmall} $Q_v(m+1) < (\log N)^{\left(\frac{3}{2}\right)^{m+1}}$, or
\item\label{wasinopen} $Q_v^i(m+1)$ is large. 
\end{enumerate}
We remark that as $\mathcal{J}(\omega) = m+1$, $Q_v^i(\ell)$ is not large for $\ell < m+1$. 

It follows that
\begin{equation}\label{supersplit}
\mathcal{A}_{M_2} = \mathcal{A}_0 \cup \bigcup_{\substack{0 \le m \le M_2 -1\\ v \in V(\mathcal{H})}} \mathcal{S}(v,m+1).
\end{equation}
We will prove that for all $0 \le m \le M_2 -1$ and $v \in V(\mathcal{H})$,
\begin{equation}
\label{sv}
\mathbb{P}(\mathcal{S}(v,m+1)) \le N^{-10}.
\end{equation}
By Lemma~\ref{phaseoneevent}, $\mathbb{P}(\mathcal{A}_0) \le N^{-2\sqrt{\log N}}.$ Lemma~\ref{superlem} will follow from \eqref{sv} by taking the union bound over all $m$ and $v \in V(\mathcal{H})$ (as $M_2 = O\left(\log \log(d)\right)$ by \eqref{Nisd} and \eqref{M2againsuper}).

So for $0 \le m \le M_2 -1$, assume $\omega \notin \mathcal{A}_{m}$. Fix $v \in V(\mathcal{H})\setminus I(m)$ and consider $\mathcal{S}(v,m+1)$. First let us bound the probability that $v \notin I(m+1)$ but \ref{wasinopen} holds.  
 
Suppose $Q_v^i(m+1)$ is large. By the argument above (culminating in \eqref{Qgetsinf}), with probability at least $1 - N^{-10\sqrt{\log N}}$ the vertex $v$ becomes infected (and is hence not present in $I(m+1)$) when every hyperedge of $Q_v^i(m+1)$ is sampled. So the probability that $v \notin I(m+1)$ but \ref{wasinopen} holds is at most $N^{-10\sqrt{\log N}}$. 
 
We will now show that when $\omega \notin \mathcal{A}_m$,
\begin{equation}\label{qve}
\mathbb{P}\left(Q_v(m+1) < (\log N)^{\left(\frac{3}{2}\right)^{m+1}} \Big\vert \mathcal{F}_{m}'\right) \le N^{-15}.
\end{equation}
Then \eqref{sv} will follow from this and the argument of the previous paragraph.

As $\omega \notin \mathcal{A}_m$, for all $u \in V(\mathcal{H})\setminus I(m)$ we have
\begin{equation}
\label{Qvbound}
Q_u(m) \le d^{\frac{1}{r-1} + \frac{1}{10}}.
\end{equation}
(By definition of the process, if $Q_u(m) > d^{\frac{1}{r-1} + \frac{1}{10}}$ then we would have sampled these open hyperedges.) Recall from Subsection~\ref{subsup} that $Q'(0) := Q(0)$ and for $m >0$, for each $v \in V(\mathcal{H})\setminus I(m)$ we choose $Q_v'(m)$ to be a subset of $Q_v(m)$ with cardinality $(\log N)^{\left( \frac{3}{2}\right)^{m}}$ (which is possible here by definition of $\mathcal{A}_{m}$) and set 
$$Q'(m):= \bigcup_{v \in V(\mathcal{H})\setminus I(m)} Q_v'(m).$$ 
Fix $v \in V(\mathcal{H})\setminus I(m)$. Define $s:= s(m)$, to be the integer in $[0,r-2]$ such that the cardinality of $Y^s_v(m):= Y_v^{s,0}(m) -  Y_v^{s+1,0}(m)$ is maximised ($Y^s_v(m)$ is the number of hyperedges of $\mathcal{H}(m)$ containing $v$ and exactly $s$ infected vertices). As $\omega \notin \mathcal{A}_{0}$, by (\ref{Ystate}) from Definition~\ref{bmdef} we have $s(0)=0$. 

Define $\mathcal{G} := \mathcal{G}(m)$ to be the $(r-1-s)$-uniform hypergraph on vertex set $Q'(m)$, where $S:= \{e_1, \dots, e_{r-1-s}\} \in E(\mathcal{G})$ if and only if there exists some $\mathcal{F} \in Y^{s}_v(m)$ such that $V(\mathcal{F})\setminus (\{v\} \cup I(m)) = V(S)\setminus I(m)$. In other words, for each healthy vertex $x$ of $\mathcal{F}\setminus \{v\}$ there is a hyperedge in $S$ containing $x$. Then $\mathcal{F}$ will become a member of $Q_v(m)$ if each of $e_1, \dots, e_{r-1-s}$ is successfully sampled from $Q'(m)$. 

Recall the definition of $Q_v^0(m+1)$ from the description of the second phase process in the supercritical case, given in Subsection~\ref{subsup}. Given $e\in E(\mathcal{G})$, let $\xi_e$ be the Bernoulli random variable which is equal to 1 if and only if $e \subseteq V(\mathcal{G})_{q}$ and define
$$X := \sum_{e \in E(\mathcal{G}(m))}\xi_e.$$ 
The random variable $X$ counts the number of sets $\{e_1,\dots,e_{r-1-s}\}$ of open hyperedges in $Q'(m)$ whose unique healthy vertices are precisely the $r-1-s$ healthy vertices, other than $v$, in an element of $Y_v^s(m)$ such that all of $e_1,\dots,e_{r-1-s}$ are successfully sampled.  Therefore, to bound the cardinality of $Q_v^0(m+1)$ for some $v\notin I_0(m+1)$, it is useful to consider the cardinality of $X$. However, note that several of the events counted by $X$ may give rise to the same open hyperedge in $Q_v^0(m+1)$, as many different sets of open hyperedges in $Q'(m)$ may have the same set of healthy vertices. However, due to the way that we defined $Q'(m)$, we have some control over the amount of ``over-counting'' that occurs. 

Let us formalise this intuition. For $0\leq m\leq M_2$, define
\[b(m):=\begin{cases}\log\log(N) & \text{if }\left(\log{N}\right)^{\left(\frac{3}{2}\right)^m} < d^{\frac{1}{2(r-1)}}\\
(\log(N))^2&\text{if }d^{\frac{1}{2(r-1)}}\leq \left(\log{N}\right)^{\left(\frac{3}{2}\right)^m} < \left(\log(N)\right)^2(100q)^{-1}\\
2q\cdot \left(\log{N}\right)^{\left(\frac{3}{2}\right)^m}& \text{otherwise}.\end{cases}\]
Say that a vertex $w$ of $V(\mathcal{H})\setminus I(m)$ is \emph{bad} at time $m$ if at least $b(m)$ hyperedges of $Q_w'(m)$ are successfully sampled. Let us pause briefly to show that, with high probability, there are no bad vertices at any point in the second phase. 

\begin{claim}
\label{noBadGuys}
Let $0\leq m\leq M_2$. If $\omega\notin \mathcal{A}_m$, then, with probability at least $1-N^{-\Omega(\log\log(N))}$, there are no bad vertices at time $m$.
\end{claim}

\begin{proof}
First, suppose that $m$ satisfies 
\[\left(\log{N}\right)^{\left(\frac{3}{2}\right)^m} < d^{\frac{1}{2(r-1)}}.\]
For any such $m$, since $\omega\in \mathcal{A}_m$, we know that $Q_w'(m) < d^{\frac{1}{2(r-1)}}$ for every vertex $w\notin I(m)$ (note that this is also clearly true in the case $m=0$). The expected number of subsets of $Q'_w(m)$ of cardinality $\log\log(N)$ which are all successfully sampled is at most
\[\binom{Q_w'(m)}{\log\log(N)}q^{\log\log(N)}\leq (Q'_w(m)q)^{\log\log(N)} = N^{-\Omega\left(\log\log(N)\right)}.\]
So, by Markov's Inequality, the probability that there is at least one such set is at most $N^{-\Omega\left(\log\log(N)\right)}$, as is the probability that there is at least one bad vertex. 

Next, suppose that
\[d^{\frac{1}{2(r-1)}}\leq \left(\log{N}\right)^{\left(\frac{3}{2}\right)^m} < \left(\log(N)\right)^2(100q)^{-1}.\]
In this case, since $\omega\notin \mathcal{A}_m$, for any vertex $w\notin I(m)$, the expected number of successfully sampled open hyperedges in $Q_w'(m)$ is $Q_w'(m)\cdot q < \left(\log(N)\right)^2/100$. By the Chernoff bound, the probability that there are more than $(\log(N))^2$ successful samples is at most 
\[2^{-(\log(N))^2} =o\left(N^{-\Omega(\log\log(N))}\right).\]
By Markov's Inequality, we again get that the probability of having at least one bad vertex is at most $N^{-\Omega(\log\log(N))}$. 

Finally, suppose that 
\[\left(\log{N}\right)^{\left(\frac{3}{2}\right)^m} \geq \left(\log(N)\right)^2(100q)^{-1}.\]
Since $\omega\notin\mathcal{A}_m$, the expected number of successfully sampled hyperedges in $Q_w'(m)$ is $q\cdot\left(\log{N}\right)^{\left(\frac{3}{2}\right)^m}$ for any vertex $w\notin I(m)$. By the Chernoff bound, the probability that $2q\cdot\left(\log{N}\right)^{\left(\frac{3}{2}\right)^m}$ are successfully sampled is at most
\[e^{-\Omega\left(q\cdot\left(\log{N}\right)^{\left(\frac{3}{2}\right)^m}\right)} \leq e^{-\Omega\left(\left(\log(N)\right)^2\right)} = o\left(N^{-\Omega(\log\log(N))}\right).\]
Again, the result follows by Markov's Inequality. 
\end{proof}

Therefore, the above claim tells us that, with high probability,
\begin{equation}\label{Qv0X}Q_v^0(m+1)\geq X/b(m)^{r-1-s}.\end{equation}
So, to prove \eqref{sv}, it suffices to prove a lower bound on $X$ which hold with high probability. Before analysing the variable $X$, it is useful to bound $|E(\mathcal{G})|$.
\begin{claim}
\label{EGbound}
If $\omega \notin \mathcal{A}_0$:
$$\frac{d}{4}\cdot (\log N)^{(r-1)^2} \le |E(\mathcal{G}(0))| \le  4d \cdot (\log N)^{(r-1)^2}$$
and, for $m \ge 1$, if $\omega \notin \mathcal{A}_m$:
$$\frac{d}{2(r-1)}(\log N)^{(r-1-s)\left(\frac{3}{2}\right)^{m}} \le |E(\mathcal{G}(m))| \le d \cdot (\log N)^{(r-1-s)\left(\frac{3}{2}\right)^{m}}.$$
\end{claim}
\begin{proof}
First consider $E(\mathcal{G}(0))$. As mentioned above, as $\omega \notin \mathcal{A}_0$ we have that $s(0) = 0$ and so $\mathcal{G}(0)$ is a $(r-1)$-uniform hypergraph on vertex set $Q'(0)=Q(0)$. For an open hyperedge $e$, let $h(e)$ be the healthy vertex of $e$. By definition, the set $\{e_1,\ldots,e_{r-1}\}$ is a hyperedge of $\mathcal{G}(0)$ if and only if there exists a hyperedge $\{v,h(e_1),\ldots,h(e_{r-1})\} \in E(\mathcal{H}(m))$. 

Say that a member of $Z_v^{0,r-1}(m)$ is \emph{centrally healthy} if its central hyperedge contains no vertex of $I(m)$ and let $\tilde{Z}_v^{0,r-1}(0)$ be the set of centrally healthy members of $Z_v^{0,r-1}(0)$. Consider the map $\phi$ from $E(\mathcal{G}(0))$ to $\tilde{Z}_v^{0,r-1}(0)$ which maps $\{e_1,\ldots,e_{r-1}\} \in  E(\mathcal{G}(0))$ to the unique member of $\tilde{Z}_v^{0,r-1}(0)$ with central hyperedge $\{v,h(e_1),\ldots,h(e_{r-1})\}$ and non central hyperedges $e_1,\ldots,e_{r-1}$. By definition of $\mathcal{G}(0)$, we have that $\phi$ is a bijection.

So to give a lower bound on $|E(\mathcal{G}(0))|$, it suffices to prove a lower bound on the number of centrally healthy members of $Z_v^{0,r-1}(\ell)$. By definition of $Z^{0,r-1}$, a lower bound is given by the number of centrally healthy members of $Y_v^{0,r-1}(0)$. We will bound this number from below.

First let us bound the number of $Y$ in $Y_v^{0,r-1}(0)$ that are not centrally healthy. Each such $Y$ can be thought of as the union of some $Y'$ in $Y_v^{1,r-2}(0)$ with a copy of $W^1$ rooted at the vertex of $Y'$ representing the marked vertex of $Y^{1,r-2}$. So in particular, as $\omega \notin \mathcal{A}_{0}$, using (\ref{Ystate}) and (\ref{Wstate}) from Definition~\ref{bmdef} we can bound the number of such $Y$ above by $\plog(d)\cdot d^{\frac{r-2}{r-1}}$.

So by the previous two paragraphs, the number of members of $Y_v^{0,r-1}(0)$ that are not centrally healthy is $\plog(d)\cdot d^{\frac{r-2}{r-1}}$. As $\omega \notin \mathcal{A}_{0}$, using (\ref{Ystate}) from Definition~\ref{bmdef} we have
\begin{equation}\label{Yboundhere}
Y_v^{0,r-1}(0) \in (1 \pm o(1))d((c + \alpha T)^{r-1} - T)^{r-1}.
\end{equation}

So 
\begin{equation}\label{e0low}
\frac{1}{2} d((c + \alpha T)^{r-1} - T)^{r-1} \le |E(\mathcal{G}(0))|.
\end{equation}

Now let us think about an upper bound for $|E(\mathcal{G}(0))|$. By the above discussion we have $|E(\mathcal{G}(0))| \le Z_v^{0,r-1}(0)$. We bound $Z_v^{0,r-1}(0)$ analogously to in the proof of Lemma~\ref{subtrack}. The variable $Z_v^{0,r-1}(m)$ counts the number of ways to choose 
\begin{enumerate}[(a)]
\item a hyperedge $e = \{v,w_1,\ldots,w_{r-1}\} \in \mathcal{H}(m)$ containing $v$, and 
\item hyperedges $e_1,\dots,e_{r-1}$ such that $ e_{\ell} \in W^1_{w_{\ell}}(m)$ for $1\leq \ell \leq r-1$.  
\end{enumerate}
The number of choices in which each of the hyperedges $e_1,\dots,e_{r-1}$ intersects $e$ on only one vertex and no pair of them intersect one another is precisely $Y_v^{0,r-1}(m)$. If one of the hyperedges $e_1,\dots,e_{r-1}$ intersects $e$ on more than one vertex or two of them intersect one another, then the union of $e$ and $e_1,\ldots, e_{r-1}$ consists of a copy $\mathcal{F}'$ of a secondary configuration with $r-1$ neutral vertices and a set of copies of $W^1$ rooted at vertices of $\mathcal{F}'$.  As $\omega \notin \mathcal{A}_0$, by (\ref{Xstate}) and (\ref{Wstate}) from Definition~\ref{bmdef} the number of choices in this case is at most 
\[\log^{O(1)}(d)\cdot d\cdot \log^{-3K/5}(d),\]
 This is $o(d)$, provided that $K$ is sufficiently large. 

As $\omega \notin \mathcal{A}_0$, by (\ref{Ystate}) from Definition~\ref{bmdef} and the previous argument we have
$$Z_v^{0,r-1}(0) \le (1 + o(1))Y_v^{0,r-1}(0).$$
So $|E(\mathcal{G}(0))| \le 2\cdot Y_v^{0,r-1}(0)$ and combining this with \eqref{Yboundhere} and \eqref{e0low} gives
\begin{equation}
\label{g0both}
\frac{1}{2} d((c + \alpha T)^{r-1} - T)^{r-1} \le |E(\mathcal{G}(0))| \le 2d((c + \alpha T)^{r-1} - T)^{r-1}.
\end{equation}
By definition of $T$ (see~\ref{Tdef}), for $N$ sufficiently large we have
\begin{equation}\label{Tchoice}
\frac{1}{2}(\log(N))^{(r-1)^2} \le ((c + \alpha T)^{r-1} - T)^{r-1}\le 2(\log(N))^{(r-1)^2}.
\end{equation}
Combining \eqref{g0both} and \eqref{Tchoice} gives
$$
\frac{d}{4}\cdot (\log N)^{(r-1)^2} \le |E(\mathcal{G}(0))| \le  4d \cdot (\log N)^{(r-1)^2},
$$
as required for the first part of the claim.

Now consider $E(\mathcal{G}(m))$ for $m \ge 1$. First let us bound $Y_v^{s}(m)$ for any $v \in V(\mathcal{H}(m))\setminus I(m)$. As $\omega \notin \mathcal{A}_m$, $v$ was never in an open hyperedge sampled in the second step of any round. So by \eqref{Qvbound}, definition of the process and definition of $M_2$ (see~\ref{M2againsuper}):
\begin{itemize}
\item[(1)] at most $m\cdot d^{\frac{1}{r-1} + \frac{1}{10}} = o(d)$ open hyperedges containing $v$ have been sampled during the second phase up until now,
\item[(2)] $Q_v(m) \le d^{\frac{1}{r-1} + \frac{1}{10}}$. 
\end{itemize}
As $\omega \notin \mathcal{A}_m$, by (\ref{Ystate}) from Definition~\ref{bmdef} we have $Y_v^{0}(0) \ge (1- o(1))d$. So putting this together with (1) and (2) gives that there are at least $(1-o(1))d \ge d/2$ hyperedges containing $v$ in $\mathcal{H}(m) \setminus Q_v(m)$. 

So by the pigeonhole principle, 
$$\frac{d}{2(r-1)} \le Y_v^s(m) \le d.$$
By definition of $\mathcal{G}(m)$, we have
\begin{align*}
|E(\mathcal{G}(m))| = \sum_{Y \in Y_v^{s}(m)} \prod_{u \in V(Y)\setminus (\{v\}\cup I(m))}Q'_u(m).
\end{align*}
Using the definition of $Q_u'(m)$ with the previous two expressions gives
$$ \frac{d}{2(r-1)}(\log N)^{(r-1-s)\left(\frac{3}{2}\right)^{m}} \le |E(\mathcal{G}(m))| \le d \cdot (\log N)^{(r-1-s)\left(\frac{3}{2}\right)^{m}},$$
completing the proof of the claim.
\end{proof}
As $\mathbb{E}(\xi_e) = q^{r-1-s}$, Claim~\ref{EGbound} implies that when $m = 0$ we have
\begin{equation}
\label{m0exp}
\frac{\alpha^{r-1}}{4}\cdot (\log N)^{(r-1)^2} \le \mathbb{E}\left(X \given \mathcal{F}_{m}'\right) \le 2\alpha^{r-1}\cdot (\log N)^{(r-1)^2},
\end{equation}
and when $m > 0$, we have
\begin{equation}\label{explast}
\frac{\alpha^{r-1-s}}{2(r-1)}  d^{\frac{s}{r-1}}\cdot\left(\log N \right)^{(r-1-s)\left(\frac{3}{2}\right)^{m}} \le \mathbb{E}\left(X \given \mathcal{F}_{m}'\right) \le \alpha^{r-1-s}  d^{\frac{s}{r-1}}\cdot\left(\log N \right)^{(r-1-s)\left(\frac{3}{2}\right)^{m}}.
\end{equation}
For any $m \ge 0$, as $0 \le s \le r-2$, the lower bound of \eqref{m0exp} and lower bound in \eqref{explast} yield 
\[
\mathbb{E}\left(X \given \mathcal{F}_{m}'\right) \ge 2\cdot\left(\log N \right)^{\left(\frac{3}{2}\right)^{m+1}}b(m)^{r-1-s}.\]
We require one more claim before we can apply  Theorem~\ref{JansonLower}.
\begin{claim}
\label{exi12}
For $0 \le m \le M_2 -1$, when $\omega \notin \mathcal{A}_m$,
$$\sum_{\substack{e,e' \in E(\mathcal{G}(m))\\ e \cap e' \not= \emptyset}}\mathbb{E}\left(\xi_e\xi_{e'}\given  \mathcal{F}_{m}' \right) = O\left(\mathbb{E}\left(X \given \mathcal{F}_{m}'\right) ^2\log^{-\frac{3}{2}} N\right).$$
\end{claim}
\begin{proof}
We have
$$\sum_{\substack{e,e' \in E(\mathcal{G}(m))\\ e \cap e' \not= \emptyset}}\mathbb{E}\left(\xi_e\xi_{e'}\given  \mathcal{F}_{m}' \right) = \sum_{e \in E(\mathcal{G}(m))} \sum_{j=1}^{r-2}\sum_{\substack{e' \in E(\mathcal{G}(m))\\|e \cap e'| = j}}\mathbb{E}\left(\xi_e\xi_{e'}\given  \mathcal{F}_{m}' \right).$$
For fixed $e \in E(\mathcal{G}(m))$, recall that $e = \{e_1,\ldots,e_{r-s-1}\}$ where each $e_i \in Q'(m)$. The quantity $|\{e' \in  E(\mathcal{G}): |e \cap e'| = j\}|$ is bounded above by the number of copies $Z$ of $Z^{s,r-1-s}$ rooted at $v$ in $\mathcal{H}(m)$ such that $Z$ contains some $\{e_1',\ldots,e_j'\} \subseteq \{e_1,\ldots,e_{r-s-1}\}$. There are $O\left(\Delta_{j+1}(\mathcal{H})\right)$ ways to choose the hyperedge $f$ of $Z$ containing $v$. Given the choice of $f$, for each $u$ that is not contained in a hyperedge of $\{e_1',\ldots,e_j'\}$ there are precisely $Q_u'(m)$ ways to choose the open hyperedge rooted at $u$.

So when $m=0$, as $\omega \notin \mathcal{A}_0$, by (\ref{Wstate}) from Definition~\ref{bmdef} for each $v \in V(\mathcal{H})\setminus I(0)$ we have $Q_v(0) \le \log^{r^4}(d)$ and so
\begin{align*}
\sum_{\substack{e,e' \in E(\mathcal{G}(0))\\ e \cap e' \not= \emptyset}}\mathbb{E}\left(\xi_e\xi_{e'}\given  \mathcal{F}_{m}' \right) &\le \frac{1}{2}\sum_{e \in E(\mathcal{G}(0))} \sum_{j=1}^{r-2}\Delta_{j+1}(\mathcal{H})\cdot (\log N)^{r^4(r-1-j)}q^{2(r-1)-j}\\
&= O\left((\log N)^{r^4(r-2) + (r-1)^2 - K}\right),
\end{align*}
where in the last line we used the upper bound on $|E(\mathcal{G}(0))|$ given by Claim~\ref{EGbound}. Using the upper bound of \eqref{m0exp}, we see the required bound for $m=0$ follows when $K$ is taken to be large with respect to $r$. 

When $m \ge 1$, again using Claim~\ref{EGbound} we have
\begin{align*}
\sum_{\substack{e,e' \in E(\mathcal{G}(m))\\ e \cap e' \not= \emptyset}}\mathbb{E}\left(\xi_e\xi_{e'}\given  \mathcal{F}_{m}' \right) &\le \sum_{e \in E(\mathcal{G}(m))} \sum_{j=1}^{r-2}\Delta_{j+1}(\mathcal{H})\cdot\left((\log N)^{\left(\frac{3}{2}\right)^{m}}\right)^{(r-1-s-j)}q^{2(r-1-s)-j}\\
& = O\left(d^{\frac{2s}{r-1}}(\log N)^{(2(r-1-s)-1)\left(\frac{3}{2}\right)^{m} - K}  \right).
\end{align*}
Using the upper bound in \eqref{explast}, we see that the required bound holds for $m >0$.
\end{proof}
So we can apply Theorem~\ref{JansonLower} with $\epsilon = 1/2$ to give
$$\mathbb{P}\left(X \le \frac{1}{2}\mathbb{E}(X) \Big\vert \mathcal{F}_{m}' \right)\le e^{-\Omega\left(\log^{\frac{3}{2}}N\right)}.$$
So combining the lower bounds given in \eqref{m0exp} and \eqref{explast} with \eqref{Qv0X} yields
$$\mathbb{P}\left(Q_v^0(m+1) < (\log N)^{\left(\frac{3}{2}\right)^{m+1}}  \given[\Big] \mathcal{F}_m'\right) \le N^{-20}.$$
If $v$ is contained in an open hyperedge sampled in the second step of this round, it becomes infected with probability at least $1 - N^{-10\sqrt{\log n}}$ (by the argument culminating in \eqref{Qgetsinf}). As $Q_v^0(m) = Q_v(m)$ for any healthy vertex $v$ not contained in an open hyperedge sampled in the second step of this round, this completes the proof of \eqref{qve} and hence the proof of Lemma~\ref{superlem}.
\end{proof}
This concludes our discussion of the second phase and therefore the proof of Theorem~\ref{hypmainThm}. 

\section{Strictly  \texorpdfstring{$k$}{k}-Balanced Hypergraphs}
\label{balancedGraphs}

In this final section, using Theorem~\ref{mainThmWeak} we prove Theorem~\ref{graphThm} as well as a generalisation of it to strictly $k$-balanced $k$-uniform hypergraphs, which we define now. The following two definition generalise the notion of $2$-density and $2$-balancedness for graphs.

\begin{defn}\label{kden}
Given a $k$-uniform hypergraph $\mathcal{F}$ with at least two hyperedges, the \emph{$k$-density} of $\mathcal{F}$ is defined by $d_k(\mathcal{F}):=\frac{|E(\mathcal{F})| - 1}{|V(\mathcal{F})| - k}$.
\end{defn}

\begin{defn}\label{kbal}
Say that a $k$-uniform hypergraph $\mathcal{F}$ with at least two hyperedges is \emph{$k$-balanced} if $d_k(\mathcal{F}')\leq d_k(\mathcal{F})$ for every proper subhypergraph $\mathcal{F}'$ of $\mathcal{F}$ with at least two hyperedges.  Say that $\mathcal{F}$ is \emph{strictly $k$-balanced} if this inequality is strict for all such subhypergraphs. 
\end{defn}

For hypergraphs $\mathcal{G}$ and $\mathcal{F}$, define $\mathcal{H}_{\mathcal{G},\mathcal{F}}$ to be the hypergraph whose vertices are the hyperedges of $\mathcal{G}$ and whose hyperedges are the sets of hyperedges which form copies of $\mathcal{F}$ in $\mathcal{G}$. Let $K_n^{(k)}$ denote the complete $k$-uniform hypergraph on $n$ vertices. For any $k$-uniform hypergraph $\mathcal{F}$, we have that  $\mathcal{H}_{K_n^{(k)},\mathcal{F}}$ is $|E(\mathcal{F})|$-uniform and $d(n,\mathcal{F})$-regular for some integer $d(n,\mathcal{F})$ such that $d(n,\mathcal{F}) = \Theta\left(n^{|V(\mathcal{F})|-k}\right)$ (the constant factor is related to the number of automorphisms of $\mathcal{F}$ which fix a hyperedge). 

Now we state a generalisation of Theorem~\ref{graphThm} to strictly $k$-balanced $k$-uniform hypergraphs.

\begin{lem}\label{graphthmk}
Let $\mathcal{F}$ be a strictly $k$-balanced $k$-uniform hypergraph with at least three hyperedges satisfying $\delta(\mathcal{F})\geq2$. Define $r:=|E(\mathcal{F})|$, $\mathcal{H}:=\mathcal{H}_{K_n^{(k)},\mathcal{F}}$ and $d:=d(n,\mathcal{F})$. If $q:=\alpha d^{-1/(r-1)}$ for some fixed $\alpha>0$, then
\[p_c\left(\mathcal{H}_{q}\right) = \left(\frac{r-2}{\alpha^{1/(r-2)}(r-1)^{(r-1)/(r-2)}} + o(1)\right)\cdot d^{-1/(r-1)}.\]
\end{lem}

Below, we show that Theorem~\ref{graphThm} follows from Theorem~\ref{mainThmWeak} and the case $k=2$ of the following proposition. Theorem~\ref{graphthmk} follows from the same proof.

\begin{prop}
\label{graphProp}
For $k\geq2$, let $\mathcal{F}$ be a $k$-uniform hypergraph with at least $3$ hyperedges and let $n\geq |V(\mathcal{F})|$. Define $r:=|E(\mathcal{F})|$, $\mathcal{H}:=\mathcal{H}_{K_n^{(k)},\mathcal{F}}$ and $d:=d(n,\mathcal{F})$.
\begin{enumerate}[(i)]
\item\label{balancedWellB} If $\mathcal{F}$ is strictly $k$-balanced and $\delta(\mathcal{F})\geq 2$, then $\mathcal{H}$ is $\left(d,\rho,\nu\right)$-well behaved where $\rho=O\left(d^{-\frac{1}{(|E(\mathcal{F})|-1)(|V(\mathcal{F})|-k)}}\right)$ and $\nu=O\left(d^{\frac{k}{|V(\mathcal{F})|-k}}\right)$. 
\item\label{notStrict} If $\mathcal{F}$ is not strictly $k$-balanced, then for every hyperedge of $\mathcal{H}$, there exists some $2 \le \ell \le r-1$ and a set $S$ of $\ell$ vertices, such that $\degre(S) = \Omega\left(d^{ 1 -\frac{\ell-1}{|E(\mathcal{F})|-1}}\right)$.
\item\label{noMinDeg} If $\delta(\mathcal{F})=1$, then for every $e_1 \in V(\mathcal{H})$ there exists some $e_2\in V(\mathcal{H})$ such that $\left|N_{\mathcal{H}}(e_1)\cap N_{\mathcal{H}}(e_2)\right| = \Omega(d)$.
\end{enumerate}
\end{prop}

The purpose of parts~\ref{notStrict} and~\ref{noMinDeg} of Proposition~\ref{graphProp} is to show that both hypotheses in part~\ref{balancedWellB} of Proposition~\ref{graphProp} are necessary (in a strong sense). That is, in order for $\mathcal{H}_{K_n^{(k)},\mathcal{F}}$ to be $\left(d(n,\mathcal{F}),\rho,\nu\right)$-well behaved where $\rho$ and $\nu$ are functions of $n$ such that $\rho$ tends to zero as $n\to\infty$, one requires both that $\mathcal{F}$ is strictly $k$-balanced and that $\delta(\mathcal{F}) \ge 2$.

As it turns out, every strictly $2$-balanced graph $F$ with at least three edges satisfies $\delta(F)\geq2$ (see the proof of Theorem~\ref{graphThm} below). Thus, in this case, the assumption that $\delta(F)\geq2$ in part~\ref{balancedWellB} of Proposition~\ref{graphProp} is redundant. However, for $k$-uniform hypergraphs with $k\geq3$, this is no longer the case. For example, a ``loose $k$-uniform cycle'' is strictly $k$-balanced and contains vertices of degree one; see Figure~\ref{looseTriangle} for an example.

\begin{figure}[htbp]
\centering
\includegraphics[width=0.25\textwidth]{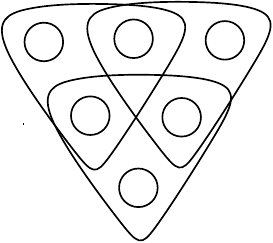}
\caption{A loose $3$-uniform triangle.}
\label{looseTriangle}
\end{figure}

We first show that Proposition~\ref{graphProp} implies Theorem~\ref{graphThm} before proving the proposition itself. 

\begin{proof}[Proof of Theorem~\ref{graphThm}]
First, we show that $d_2(F)\geq1$. If not, then $F$ contains at most $|V(F)|-2$ edges, and so $F$ is disconnected. If $F$ contains a connected component $F'$ with at least two edges, then $d_2(F')\geq 1>d_2(F)$, contradicting the fact that $F$ is strictly $2$-balanced. Otherwise, every component of $F$ contains at most one edge, which implies that $d_2(F)\leq 1/2$. In this case, we let $F'$ be a subgraph of $F$ consisting of two edges and four vertices. Since $F$ has at least three edges, we have that $F'$ is a proper subgraph of $F$. Also, $d_2(F')=1/2 \geq d_2(F)$, which contradicts the fact that $F$ is strictly $2$-balanced. 

By the previous paragraph, we have $d_2(F)\geq1$. In particular, this implies that $F$ cannot contain a vertex $v$ of degree one; otherwise, let $F':=F\setminus \{v\}$ and observe that $F'$ has at least two edges and that $d_2(F')\geq d_2(F)$. Thus, $\delta(F)\geq2$ and so we can apply part~\ref{balancedWellB} of Proposition~\ref{graphProp} to get that there exists $\beta,s>0$ such that $\mathcal{H}_{K_n,\mathcal{F}}$ is $\left(d(n,F),d(n,F)^{-s},d(n,F)^\beta\right)$-well behaved for $n$ sufficiently large. The result now follows by applying Theorem~\ref{mainThmWeak} to $\mathcal{H}_{K_n,\mathcal{F}}$. 
\end{proof}

We now present the proof of Proposition~\ref{graphProp}.

\begin{proof}[Proof of Proposition~\ref{graphProp}]
First suppose that $\mathcal{F}$ is strictly $k$-balanced and that $\delta(\mathcal{F}) \ge 2$. As $\mathcal{H}$ is $d$-regular, conditions~\ref{maxDeg} and~\ref{almostRegular} of Definition~\ref{hypwellBDef} hold for $\mathcal{H}$. Also, $\left|V\left(\mathcal{H}\right)\right|= \binom{n}{k}$ and so, as $d = \Theta\left(n^{|V(\mathcal{F})| - k}\right)$, we have 
\[\left|V\left(\mathcal{H}\right)\right| =O\left(n^k\right) = O\left(n^{\left(|V(\mathcal{F})| - k\right)\frac{k}{|V(\mathcal{F})| - k}}\right) =O\left(d^{\frac{k}{|V(\mathcal{F})|-k}}\right).\]
Therefore, condition~\ref{Nnottoobig} of Definition~\ref{hypwellBDef} holds.

Next, we show that condition~\ref{codegBound} of Definition~\ref{hypwellBDef} is satisfied. For $2\leq \ell\leq r-1$, let $S$ be a set of $\ell$ vertices of $\mathcal{H}$ and let $\mathcal{F}'$ be the subhypergraph of $K_n^{(k)}$ induced by the hyperedges corresponding to elements of $S$. If $\mathcal{F}'$ is not isomorphic to a subhypergraph of $\mathcal{F}$, then $\deg(S)=0$ and we are done. Otherwise, we have
\[\deg(S) = \Theta\left(n^{|V(\mathcal{F})|-|V(\mathcal{F}')|}\right).\]
Now, since $\mathcal{F}$ is strictly $k$-balanced, 
\[\frac{|E(\mathcal{F}')| - 1}{|V(\mathcal{F}')|-k} < \frac{|E(\mathcal{F})| - 1}{|V(\mathcal{F})|-k}\Rightarrow \frac{\left(|E(\mathcal{F}')|-1\right)\left(|V(\mathcal{F})|-k\right)}{|E(\mathcal{F})| - 1} < |V(\mathcal{F}')| - k\]
\[\Rightarrow |V(\mathcal{F})|-|V(\mathcal{F}')| <  (|V(\mathcal{F})|-k)\left(1-\frac{|E(\mathcal{F}')|-1}{|E(\mathcal{F})|-1}\right) = (|V(\mathcal{F})|-k)\left(1-\frac{\ell -1}{|E(\mathcal{F})|-1}\right).\]
Since the above inequality is strict, we get 
\begin{align*}
\deg(S)=\Theta\left(n^{|V(\mathcal{F})|-|V(\mathcal{F}')|}\right) &= O\left(n^{(|V(\mathcal{F})|-k)\left(1-\frac{\ell -1}{|E(\mathcal{F})|-1}\right)} n^{- \frac{1}{|E(\mathcal{F})|-1}}\right)
\\ &=O\left(d^{1-\frac{\ell-1}{|E(\mathcal{F})|-1}}n^{-\frac{1}{|E(\mathcal{F})|-1}}\right)
\\ &=O\left(d^{1-\frac{\ell-1}{|E(\mathcal{F})|-1}}d^{-\frac{1}{(|E(\mathcal{F})|-1)(|V(\mathcal{F})|-k)}}\right)
\end{align*}
which implies condition~\ref{codegBound} of Definition~\ref{hypwellBDef} is satisfied.

Next we show that if $\delta(\mathcal{F})\geq2$, then condition~\ref{neighSim} of Definition~\ref{hypwellBDef} holds. Let $e_1$ and $e_2$ be two distinct hyperedges of $K_n^{(k)}$ and let $v_1$ be a vertex of $e_1$ which is not contained in $e_2$. Suppose that $\mathcal{F}'$ is a copy of $\mathcal{F}$ in $K_n^{(k)}$ containing $e_2$ such that $\left(\mathcal{F}'\setminus\{e_2\}\right)\cup \{e_1\}$ is also a copy of $\mathcal{F}$. Then, as $\delta(\mathcal{F}) \ge 2$, we have that $\mathcal{F}'$ also contains $v_1$. However, the number of copies of $\mathcal{F}$ in $K_n^{(k)}$ containing $e_2$ and $v_1$ is 
\[O\left(n^{|V(\mathcal{F})|-k-1}\right) = O\left(d^{1-\frac{1}{|V(\mathcal{F})|-k}}\right)=o\left(d^{1-\frac{1}{(|E(\mathcal{F})|-1)(|V(\mathcal{F})|-k)}}\right)\]
since $\mathcal{F}$ has at least two hyperedges.  Therefore, $\mathcal{H}$ satisfies condition~\ref{neighSim} of Definition~\ref{hypwellBDef}. This completes the proof of part~\ref{balancedWellB}. 

Now suppose that $\mathcal{F}$ is not strictly $k$-balanced and let $\mathcal{F}'$ be a subgraph of $\mathcal{F}$ with at least $2$ hyperedges such that $d_k(\mathcal{F}')\geq d_k(\mathcal{F})$ and define $\ell=|E(\mathcal{F}')|$. Let $e$ be a hyperedge of $\mathcal{H}$ (by definition the corresponding hyperedges of $K_n^{(k)}$ induce a copy of $\mathcal{F}'$). So we can pick $S$ be a set of $\ell$ vertices of $e$ such that the corresponding hyperedges of $K_n^{(k)}$ induce a copy of $\mathcal{F}'$. Once again, we get that $\deg(S) = \Theta\left(n^{|V(\mathcal{F})|-|V(\mathcal{F}')|}\right)$. This time, though, 
\[|V(\mathcal{F})|-|V(\mathcal{F}')| \geq (|V(\mathcal{F})|-k)\left(1-\frac{\ell-1}{|E(\mathcal{F})|-1}\right).\]
Therefore, $\deg(S) = \Omega\left(d^{ 1 -\frac{\ell-1}{|E(\mathcal{F})|-1}}\right)$. This completes the proof of part~\ref{notStrict}.

To prove~\ref{noMinDeg}, suppose that $\mathcal{F}$ contains a vertex of degree one and let $e$ be a hyperedge of $\mathcal{F}$ containing such a vertex. Let $t$ denote the number of vertices of degree one contained in $e$. Fix $e_1 \in V(\mathcal{H})$ and let $e_2$ be any hyperedge of $K_n^{(k)}$ which intersects $e_1$ on $k-t$ vertices. The number of copies of $\mathcal{F}$ in $K_n^{(k)}$ containing $e_1$ but not containing any vertex of $e_2\setminus e_1$ is at least a constant multiple of $\binom{n - |e_2\setminus e_1|}{|V(\mathcal{F})|-k}=\Omega\left(n^{|V(\mathcal{F})|-k}\right)$.  For any such copy, say $\mathcal{F}'$, we have that $\mathcal{F}'\setminus\{e_1\}\cup\{e_2\}$ is also a copy of $\mathcal{F}$. So 
\[\left|N_{\mathcal{H}}(e_1)\cap N_{\mathcal{H}}(e_2)\right| = \Omega\left(n^{|V(\mathcal{F})|-k}\right)= \Omega(d),\]
 as $d = \Theta\left(n^{|V(\mathcal{F})| - k}\right)$. This completes the proof of part~\ref{noMinDeg}, and of the proposition. 
\end{proof}

As we have proved, we can apply Theorem~\ref{mainThmWeak} to $\mathcal{H}_{K_n^{(k)}, \mathcal{F}}$ if and only if $\mathcal{F}$ is strictly $k$-balanced and $\delta(\mathcal{F}) \ge 2$. However, not only are these properties of $\mathcal{F}$ required to apply the theorem, but if they are violated, then the critical probability may take on a different value; in particular it may be lower than the result stated by our theorem. 

Let us heuristically discuss why this is so. First consider the hypothesis that $\mathcal{F}$ is strictly $k$-balanced. By Proposition~\ref{graphProp} \ref{notStrict}, if $\mathcal{F}$ is not strictly $k$-balanced, then there exists $2 \le \ell \le r-1$ such that every hyperedge contains a set $S$ of $\ell$ vertices such that $\deg(S) = \Omega\left(d^{1 - \frac{\ell - 1}{r-1}}\right)$. In particular, the bounds on the $\ell$-codegrees are not sufficient to ensure that the secondary configurations remain a lower order term to the $Y$ configurations. For example, consider the secondary configuration $X = (\mathcal{G},R,D)$ where $|V(\mathcal{G})|= 2r - \ell$, $E(\mathcal{G}) = \{e_1,e_2\}$, $e_1 \cap e_2 = \{v_1,v_2, \ldots, v_{\ell}\}$, $|R|= 1$, $R \subseteq e_1 \setminus e_2$ and $D = V(\mathcal{G})\setminus R \cup \{v_1\}$. Letting $\ell$ be as in Proposition~\ref{graphProp} \ref{notStrict}, we would expect 
$$X_v(0) = \Theta\left(d \cdot \Delta_{\ell} \cdot p^{2r - \ell - 2}\right) = \Omega\left(d^{\frac{1}{r-1}}\right),$$
which is the same order as $Y_v^{r-2,1}(0)$. So right from the start of the process the $Y$ configurations and hence the number of open hyperedges will grow faster than they would if the infection were uniform. Given this, we would expect the critical probability to be lower.

Similarly, if $\delta(\mathcal{F}) = 1$, then by Proposition~\ref{graphProp} \ref{noMinDeg}, for every vertex $e_1 \in V(\mathcal{H})$ there exists some $e_2$ such that $\left|N_{\mathcal{H}}(e_1)\cap N_{\mathcal{H}}(e_2)\right|$ is large. We needed condition~\ref{neighSim} of Definition~\ref{hypwellBDef} to bound the secondary configurations with no infected vertices. As in the previous paragraph, if this bound is relaxed, the secondary configurations are no longer forced to be a lower order term compared to the $Y$ configurations and we expect this to cause the open hyperedges to grow faster than they would if the infection was uniform.

\begin{ack}
We would like to thank Oliver Riordan for his careful reading of the proof as part of the first author's DPhil thesis and for helpful comments regarding the exposition and presentation of the results. We would also like to thank an anonymous referee for reading the paper thoroughly and proposing several improvements and corrections. 

This work was initiated while the authors were visiting Rob Morris at IMPA in 2016. We are grateful to Rob and IMPA for their hospitality and for providing a stimulating research environment. All of the research and some of the writing of the initial submission was done while the authors were DPhil students at the University of Oxford. The writing of the initial submission was completed while the first author was supported by a Research Fellowship from Sidney Sussex College, Cambridge. Part of the writing of the initial submission was done while the second author was a Postdoctoral Researcher in the Department of Mathematics at ETH Z\"urich and it was completed while he was a Research Fellow in the Department of Computer Science at the University of Warwick supported by the European Research Council (ERC) under the European Union's Horizon 2020 research and innovation programme (grant agreement No 648509). Part of the revision was done while the first author was still a Research Fellow at Sidney Sussex College, Cambridge and the second author was affiliated with the Mathematics Institute at the University of Warwick and supported by the Leverhulme Trust Early Career Fellowship ECF-2018-534. The revision was completed while both authors were affiliated with the Department of Mathematics and Statistics at the University of Victoria. We thank all of these institutions and organisations for their support.
\end{ack}

\newcommand{\SortNoop}[1]{}
\providecommand{\bysame}{\leavevmode\hbox to3em{\hrulefill}\thinspace}
\providecommand{\MR}{\relax\ifhmode\unskip\space\fi MR }
\providecommand{\MRhref}[2]{%
  \href{http://www.ams.org/mathscinet-getitem?mr=#1}{#2}
}
\providecommand{\href}[2]{#2}

\end{document}